\documentclass [11pt,oneside,a4paper,mathscr]{amsart}

\usepackage{amscd,amsmath,amssymb,euscript}
\usepackage[frame,cmtip,curve,arrow,matrix,line,graph]{xy}
\usepackage{tikz-cd}
\usepackage{hyperref}
\usepackage{stmaryrd}
\usepackage{bm}

\title[]{Quasi-BPS categories for Higgs bundles}
\author{Tudor P\u adurariu and Yukinobu Toda}
\newtheorem{thm}{Theorem}[section]
\newtheorem{cor}[thm]{Corollary}
\newtheorem{prop}[thm]{Proposition}
\newtheorem{conj}[thm]{Conjecture}
\newtheorem{lemma}[thm]{Lemma}

\theoremstyle{definition}
\newtheorem{defn}[thm]{Definition}

\newtheorem{thm*}[thm]{Theorem$^*$}
\newtheorem{remark}[thm]{Remark}

\newtheorem{example}[thm]{Example}

\newtheorem{assum}[thm]{Assumption}

\newcommand{\comment}[1]{}

\renewcommand{\leq}{\leqslant}
\renewcommand{\geq}{\geqslant}

\newcommand{\X}{\mathcal{X}}

\newcommand{\Coh}{\operatorname{Coh}}
\newcommand{\Aut}{\operatorname{Aut}}

\newcommand{\Ker}{\operatorname{Ker}}
\newcommand{\id}{\operatorname{id}}
\newcommand{\Ext}{\operatorname{Ext}}

\newcommand{\Hom}{\operatorname{Hom}}

\newcommand{\rank}{\operatorname{rank}}
\newcommand{\Spec}{\operatorname{Spec}}

\newcommand{\GL}{\operatorname{GL}}

\newcommand{\inclusion}{\ar@<-0.3ex>@{^{(}->}[r]}

\newcommand{\ssslash}{/\!\!/}

\newcommand{\colim}{\operatorname{colim}}
\usetikzlibrary{positioning,fit,calc}
\tikzstyle{block}=[draw=black, width=1cm, minimum height=2cm, align=center] 
\tikzstyle{block2}=[draw=black, text width=2cm, minimum height=1cm, align=center] 
\tikzstyle{block3}=[draw=black, text width=2cm, minimum height=1cm, align=center] 

\makeatletter

\@addtoreset{equation}{section}	
\makeatother
\begin{document}

\begin{abstract}
We introduce quasi-BPS categories for twisted Higgs bundles, which are building blocks of the derived category of coherent sheaves on the moduli stack of semistable twisted Higgs bundles 
over a smooth projective curve. 
Under some condition (called BPS condition), 
the quasi-BPS categories are non-commutative 
analogues of Hitchin integrable systems. 

We begin the study of these quasi-BPS categories by focusing on a conjectural symmetry which swaps the Euler characteristic and the weight. 
Our conjecture is inspired by the Dolbeault Geometric Langlands equivalence of Donagi--Pantev, by the Hausel--Thaddeus mirror symmetry, and by the $\chi$-independence phenomenon for BPS invariants of curves on Calabi-Yau threefolds.

We prove our conjecture in the case of rank two and genus zero. In higher genus, 
we prove a derived equivalence of rank two stable twisted Higgs moduli spaces as a 
special case of our conjecture. 

In a separate paper, we prove a version of our conjecture for the topological K-theory of 
quasi-BPS categories and we discuss the relation between quasi-BPS categories and BPS invariants of the corresponding local Calabi-Yau threefold. 
\end{abstract}

	\maketitle

 \setcounter{tocdepth}{2}
\tableofcontents

\section{Introduction}
Let $C$ be a complex smooth projective curve of genus $g$ and let
$L$ be a line bundle on $C$
such that either $\deg L>2g-2$ or $L=\Omega_C$. In what follows, we consider
the tuple of integers $(r, \chi, w)$
consisting of 
the $\textit{rank}$, $\textit{Euler characteristic}$, and $\textit{weight}$, respectively, of a $L$-Higgs bundle on $C$. 
The purpose of this paper is to study the 
categories 
\begin{align}\label{intro:qbps}
    \mathbb{T}^L(r, \chi)_w
\end{align}
called \textit{quasi-BPS categories}, and their reduced versions for $L=\Omega_C$. 
When $(r, \chi)$ are coprime, the category (\ref{intro:qbps})
is equivalent to the derived category of coherent sheaves on the moduli space of 
$L$-twisted Higgs bundles on $C$ of rank $r$ and Euler characteristic $\chi$. If furthermore $L=\Omega_C$, this moduli space is a holomorphic symplectic manifold with an integrable system,  
called \textit{Hitchin system}~\cite{Hitchin}. 
However, if $(r, \chi)$ are not coprime, 
but the tuple $(r, \chi, w)$ satisfies the \textit{BPS condition}, 
the category (\ref{intro:qbps}) provides a categorical smooth and proper
Calabi-Yau fibration over the Hitchin base, which we regard as a 
non-commutative analogue of Hitchin's integrable system. 
Similar categories 
have already been introduced and studied for K3 surfaces in our previous paper~\cite{PTK3}
as a categorification of BPS invariants/cohomologies for local K3 surfaces~\cite{JS, K-S, DM, D, KinjoKoseki, DHSM, DHSM2}. 

The main purpose of this paper is to propose a conjectural equivalence 
between quasi-BPS categories which swaps the Euler characteristic $\chi$ and the
weight $w$. Our conjecture is inspired by the Dolbeault Geometric Langlands equivalence of Donagi--Pantev~\cite{DoPa} (and extended by Arinkin \cite{Ardual} and Melo--Rapagnetta--Viviani \cite{MRVF, MRVF2}) and 
by the mirror symmetry for Higgs bundles of Hausel--Thaddeus~\cite{HauTha}. 

We can also regard our conjecture as a categorical version of the $\chi$-independence phenomenon for BPS invariants on Calabi-Yau threefolds~\cite{TodGV, DMJS, KinjoKoseki}. We discuss the connection between quasi-BPS categories and the BPS cohomology/ invariants of the local Calabi-Yau threefold $\mathrm{Tot}_C(L\oplus L^{-1}\otimes \Omega_C)$ and provide further evidence of our main conjectures in \cite{PThiggs2}.
% We first state the main conjecture, give the main results on 
% our conjecture for rank two twisted Higgs bundles, and explain 
% our motivation of the conjecture. 

\subsection{Quasi-BPS categories}
We first mention some properties of quasi-BPS categories for Higgs bundles which are analogous (and proved similarly) to properties of quasi-BPS categories for K3 surfaces introduced in \cite{PTK3}.
For a smooth projective curve $C$ of genus $g$ 
and a line bundle $L$ on $C$, a $L$-\textit{twisted Higgs bundle} 
\begin{align*}
    (F, \theta), \ \theta \colon F \to F\otimes L
\end{align*}
consists of a vector bundle $F$ on $C$ and a morphism $\theta$. 
We always assume that \[l:=\deg L>2g-2\text{ or }L=\Omega_C.\] 
Let $\mathcal{M}^L(r, \chi)$ be the (derived)
moduli stack of Higgs bundles $(F, \theta)$ such that 
$\rank(F)=r\geq 1$ and $\chi(F)=\chi$. 
Consider the maps 
\begin{align*}
    \mathcal{M}^L(r, \chi)^{\rm{cl}} \to M^L(r, \chi) \to B,
\end{align*}
where the first map is the good moduli space map~\cite{MR3237451} of the classical truncation of $\mathcal{M}^L(r, \chi)$ 
and the second map is the Hitchin fibration. 
The moduli space $M^L(r, \chi)$ is a quasi-projective 
variety, non-singular when $(r, \chi)$ are coprime, 
but singular for general $(r, \chi)$. 

Quasi-BPS categories are used to decompose the derived category of coherent sheaves on the stack $\mathcal{M}^L(r, \chi)$. Before we state the first main property of quasi-BPS categories, we mention two fundamental structures of the category $D^b(\mathcal{M}^L(r, \chi))$.
First, there is an orthogonal decomposition
\[D^b(\mathcal{M}^L(r, \chi))=\bigoplus_{w\in\mathbb{Z}}D^b(\mathcal{M}^L(r, \chi))_w,\] where $D^b(\mathcal{M}^L(r, \chi))_w$ is the subcategory of $D^b(\mathcal{M}^L(r, \chi))$ of weight $w$ complexes with respect to the action of the scalar automorphisms $\mathbb{C}^*$. Second, for a partition $(r_1,\chi_1)+\cdots+(r_k,\chi_k)=(r,\chi)$ with $\chi/r=\chi_i/r_i$ for all $1\leq i\leq k$, there is a categorical Hall product (defined by Porta--Sala~\cite{PoSa}):
\begin{equation}\label{Hallproductdef}
\boxtimes_{i=1}^k D^b(\mathcal{M}^L(r_i,\chi_i))\to D^b(\mathcal{M}^L(r,\chi)).
\end{equation}
% Such a semiorthogonal decomposition can be also used to define quasi-BPS categories.
% We first collect some basic properties of quasi-BPS categories, 
% which also uniquely determine them. 

The following result is 
proved for K3 surfaces in~\cite{PTK3},
and the proof for the Higgs bundle case is the same. 
%We give its outline in this paper. 
Note that the semiorthogonal decomposition below can be also used to define quasi-BPS categories inductively on $(r,\chi)$. 
%In what follows, the ratio $\chi/r$ is infinity if $r=0$ and $\chi$ is arbitrary. 
\begin{thm}\emph{(Theorem~\ref{thm:sod})}\label{thm:intro0}
Let $(r,\chi)\in \mathbb{Z}_{\geq 1}\times\mathbb{Z}$ and let $w\in\mathbb{Z}$.
There exists a subcategory (called quasi-BPS category)
\begin{align}\label{intro:T}
  \mathbb{T}^L(r, \chi)_w \subset D^b(\mathcal{M}^L(r, \chi))_w
\end{align}
such that there is a 
semiorthogonal decomposition 
\begin{align}\label{SODthm11}
    D^b(\mathcal{M}^L(r, \chi))_w
    =\left\langle \boxtimes_{i=1}^k \mathbb{T}^L(r_i, \chi_i)_{w_i} \,\bigg|\, 
    \frac{v_1}{r_1}<\cdots<\frac{v_k}{r_k}\right\rangle. 
\end{align}
The right hand side is after all partitions 
$(r, \chi, w)=(r_1, \chi_1, w_1)+\cdots+(r_k, \chi_k, w_k)$ with 
$\chi/r=\chi_i/r_i$ for $1\leq i\leq k$, and 
$v_i \in \frac{1}{2}\mathbb{Z}$ are defined by 
\begin{align*}
    v_i:=w_i-\frac{l}{2}r_i \left(\sum_{i>j}r_j-\sum_{i<j}r_j   \right). 
\end{align*}
The (fully faithful) functor from the summands of the right hand side of \eqref{SODthm11} is the Hall product \eqref{Hallproductdef}. The order of the summands is as in the analogous theorems in \cite{PTK3, PT0}.
\end{thm}

In Subsection \ref{subsec:bun}, we also discuss a version of Theorem \ref{thm:intro0} for the moduli of semistable vector bundles on a curve that may be of independent interest.

Locally on $M^L(r, \chi)$, the category \eqref{intro:T} is described in terms of a non-commutative crepant resolutions of GIT quotients defined by \v{S}penko--Van den Bergh~\cite{SVdB}, or its variants for matrix factorizations
when $L=\Omega_C$. 
Below we omit the notation $L$ in the case of $L=\Omega_C$. 
Note that one can also consider quasi-BPS categories for $r=0$, when a pair $(F, \theta)$ corresponds to a zero-dimensional sheaf on the surface $\mathrm{Tot}_C(L)$. In this case, the 
analogous result of Theorem~\ref{thm:intro0} is proved in~\cite{P2}. 

We also consider 
the reduced version 
\begin{align}\label{intro:Tred}
    \mathbb{T}(r, \chi)^{\rm{red}}_w \subset D^b(\mathcal{M}(r, \chi)^{\rm{red}})_w
\end{align}
by removing a redundant derived structure from $\mathcal{M}(r, \chi)$. 
The categories (\ref{intro:T}), (\ref{intro:Tred}) recover the BPS cohomology/ invariants of the local Calabi-Yau threefold $\mathrm{Tot}_C(L\oplus L^{-1}\otimes \Omega_C)$ if the tuple $(r, \chi, w)$ satisfies the
\textit{BPS condition}, 
that is, if the vector 
\begin{align}\label{intro:bpscond}
    (r, \chi, w+1-g^{\rm{sp}}) \in \mathbb{Z}^3
\end{align}
is primitive (alternatively, that $\gcd(r, \chi, w+1-g^{\rm{sp}})=1$). Here $g^{\rm{sp}}$ is the genus of the spectral curve, see the formula (\ref{formula:gD}). 
For details on the comparison of topological K-theory of quasi-BPS categories with BPS cohomology/ invariants, see \cite{PThiggs2}.

If the integers $(r,\chi,w)$ satisfy the BPS condition, the categories $\mathbb{T}^L(r,\chi)_w$ and $\mathbb{T}(r,\chi)_w^{\mathrm{red}}$ are called \textit{BPS categories}. 
The BPS categories are categorical versions of twisted crepant resolutions of $M^L(r,\chi)$ (as in \cite{PTK3} for K3 surfaces). 
We remark that, in general, the moduli spaces $M^L(r,\chi)$ do not admit (geometric) crepant resolutions \cite{KaLeSo}. In this paper, we regard BPS categories 
as categorical analogues of \textit{abelian fibrations over the Hitchin base}.

\begin{thm}\emph{(Corollary~\ref{cor:proper0}, Corollary~\ref{cor:proper})}\label{thm:intro1}
Suppose that $(r, \chi, w)$ satisfies the BPS condition. 

(1) If $l>2g-2$, then $\mathbb{T}^L(r, \chi)_w$ is a smooth dg-category over $\mathbb{C}$, 
which is proper and Calabi-Yau over $B$.

(2) If $L=\Omega_C$, then $\mathbb{T}(r, \chi)^{\rm{red}}_{w}$
is a smooth dg-category over $\mathbb{C}$, 
which is proper and Calabi-Yau over $B$.
\end{thm}

We refer to Subsection~\ref{subsec:sheaves} for the notions of 
smooth, proper, and Calabi-Yau for dg-categories. 
We note that if $(r, \chi)$ are not coprime, then 
$D^b(\mathcal{M}^L(r, \chi))_w$ is not smooth, and
not proper over $B$ for $L=\Omega_C$. 
The BPS categories take into account both the singularities of $M^L(r, \chi)$ and the stacky structure of $\mathcal{M}^L(r, \chi)$, but, in contrast to these (geometric) spaces, they have nice properties reminiscent of smooth varieties.

\subsection{Main conjectures}\label{subsec:intro:main}
We now state the main conjectures in this paper.

\begin{conj}\label{conj:intro}
    Suppose the tuple $(r, \chi, w)$ satisfies the BPS condition. Then there is an equivalence
    \begin{align}\label{equivT:intro}
 \mathbb{T}^L(r, w+1-g^{\rm{sp}})_{-\chi+1-g^{\rm{sp}}} \stackrel{\sim}{\to}
\mathbb{T}^L(r, \chi)_w. 
    \end{align}
\end{conj}

The left hand side in (\ref{equivT:intro})
is equivalent to $\mathbb{T}^L(r, w)_{-\chi}$ if $l$ is even. 
In the case of $L=\Omega_C$, we also propose the following:
\begin{conj}\label{conj:intro2}
Suppose that $(r, \chi, w)$ is primitive. Then there is an equivalence 
\begin{align}\label{equivT:intro2}
    \mathbb{T}(r, w)^{\rm{red}}_{-\chi} \stackrel{\sim}{\to} \mathbb{T}^L(r, \chi)_w^{\rm{red}}. 
\end{align}
    
\end{conj}

We next give a heuristic explanation for the equivalences (\ref{equivT:intro}), (\ref{equivT:intro2}), and in particular we explain that the above equivalences are inspired by the work of Donagi--Pantev~\cite{DoPa} (and its extensions \cite{Ardual, MRVF, MRVF2}).
It is useful to use the following parametrization
\begin{align*}
    \widetilde{\mathbb{T}}^L(r, e)_w :=\mathbb{T}^L(r, \chi=e+1-g^{\rm{sp}})_w
\end{align*}
where $e$ is the degree of the bundle on the spectral curve. 
Let $B^{\rm{sm}} \subset B$ be the open subset corresponding 
to smooth and connected spectral curves. Over the locus $B^{\rm{sm}}$, 
the equivalence 
\begin{align}\label{equivT:intro3}
    \widetilde{\mathbb{T}}^L(r, w)_{-e}|_{B^{\rm{sm}}} \stackrel{\sim}{\to} \widetilde{\mathbb{T}}^L(r, e)_{w}|_{B^{\rm{sm}}}
\end{align}
is constructed as a relative Fourier-Mukai transform of dual smooth abelian 
fibrations~\cite{Mu1}. Indeed, both sides in (\ref{equivT:intro3}) are derived categories of 
twisted coherent sheaves over
relative Picard schemes over the smooth family of spectral curves \cite{DoPa}. 
Now, from Theorem~\ref{thm:intro1}, the two sides in (\ref{equivT:intro})
are \textit{categorical smooth Calabi-Yau compactifications over the Hitchin base $B$} of the two sides of \eqref{equivT:intro3}. 
Thus, by analogy to the D/K equivalence in birational geometry~\cite{B-O2, MR1949787} or to SYZ mirror symmetry~\cite{SYZ}, we expect that 
the equivalence (\ref{equivT:intro3}) extends to an equivalence
\begin{align*}
    \widetilde{\mathbb{T}}^L(r, w)_{-e}\stackrel{\sim}{\to} \widetilde{\mathbb{T}}^L(r, e)_{w}. 
\end{align*}
By unraveling the definition, we obtain the equivalence (\ref{equivT:intro}). The equivalence (\ref{equivT:intro2}) has a similar heuristic explanation. 

For $G=\mathrm{SL}(r)$ or $G=\mathrm{PGL}(r)$, we introduce quasi-BPS categories
\begin{align*}
    \mathbb{T}_G^L(\chi)_w \subset D^b(\mathcal{M}_G^L(\chi)),
\end{align*}
where $\mathcal{M}_G^L(\chi)$ is the moduli stack of $G$-Higgs bundles with 
Euler characteristic $\chi$, see Subsection~\ref{subsec:SLPGL}. 
Inspired by the mirror symmetry conjecture of Hausel--Thaddeus \cite{HauTha},
we also propose the following $\mathrm{SL}/\mathrm{PGL}$-version of the conjecture. 
\begin{conj}\label{conj:intro3}
Suppose the tuple $(r, \chi, w)$ satisfies the BPS condition. 
Then there is an equivalence 
\begin{align*}
    \mathbb{T}^L_{\mathrm{PGL}(r)}(w+1-g^{\rm{sp}})_{-\chi+1-g^{\rm{sp}}} \stackrel{\sim}{\to}
    \mathbb{T}^L_{\mathrm{SL}(r)}(\chi)_w. 
\end{align*}
\end{conj}

\subsection{Main results}
The first main result towards the conjectures mentioned above is an analogue of the main theorem in \cite{ArFe}. 
We consider open substacks 
\begin{align*}
    \mathcal{M}^L(r, \chi)^{\rm{sreg}} \subset \mathcal{M}^L(r, \chi)^{\rm{reg}} \subset \mathcal{M}^L(r, \chi),
\end{align*}
where $(-)^{\rm{reg}}$ means the regular part that corresponds to 
line bundles on spectral curves and $(-)^{\rm{sreg}}$ means the stable and regular part. 
There are good moduli space maps
\begin{align*}
      \mathcal{M}^L(r, \chi)^{\rm{sreg}} \to M^L(r, \chi)^{\rm{sreg}} \subset M^L(r, \chi),
\end{align*}
where the first map is a $\mathbb{C}^{\ast}$-gerbe and the second 
one is an open immersion. 
Following~\cite{Ardual}, consider the Poincaré line bundle 
\begin{align*}
\mathcal{P}^{\rm{reg}} \to \mathcal{M}^L(r, w+1-g^{\rm{sp}})^{\rm{reg}}
\times_B \mathcal{M}^L(r, \chi)   
\end{align*}
whose fiber at $(F, E)$ over the spectral curve $\mathcal{C}_b$ for $b\in B$ is 
\begin{align*}
    \mathcal{P}^{\rm{reg}}|_{(F, E)}=\det R\Gamma(F \otimes E)\otimes \det R\Gamma(F)^{-1}
    \otimes \det R\Gamma(E)^{-1}\otimes \det R\Gamma(\mathcal{O}_{\mathcal{C}_b}). 
\end{align*}
We will consider the category 
$\mathbb{T}_{\rm{qcoh}}^L(r, \chi)_w$, which is a 
quasi-BPS category for quasi-coherent sheaves, and prove the following: 
\begin{thm}\emph{(Theorem~\ref{thm:ff})}\label{thm:intro1.5}
The line bundle $\mathcal{P}^{\rm{reg}}$ induces the functor 
\begin{align}\label{intro:ff}
\mathbb{T}_{\rm{qcoh}}^L(r, w+1-g^{\rm{sp}})_{-\chi+1-g^{\rm{sp}}}|_{M^L(r, w+1-g^{\rm{sp}})^{\rm{sreg}}}\to 
    \mathbb{T}_{\rm{qcoh}}^L(r, \chi)_w. 
\end{align}
Moreover, if either $r=2, l>3g-2$ or $r=3, l>4g-3$, then the above 
functor is fully-faithful. 
    \end{thm}
    Note that the left hand side of \eqref{intro:ff} is $D_{\rm{qcoh}}(M^L(r, w+1-g^{\rm{sp}})^{\rm{sreg}})_{-\chi+1-g^{\rm{sp}}}$.
    Assume Conjecture~\ref{conj:intro} is true.
Then there is a fully-faithful functor 
\begin{align*}
   \mathbb{T}_{\rm{qcoh}}^L(r, w+1-g^{\rm{sp}})_{-\chi+1-g^{\rm{sp}}}|_{M^L(r, w+1-g^{\rm{sp}})^{\rm{sreg}}}
    \hookrightarrow \mathbb{T}_{\rm{qcoh}}^L(r, w+1-g^{\rm{sp}})_{-\chi+1-g^{\rm{sp}}},
\end{align*}
so its composition with 
the ind-completion of (\ref{equivT:intro}) may coincide with 
the functor in Theorem~\ref{thm:intro1.5}.
Thus, Theorem~\ref{thm:intro1.5} provides evidence for Conjecture~\ref{conj:intro}. 

We next try to extend the functor in Theorem~\ref{thm:intro1.5} to the non-regular locus. 
The line bundle $\mathcal{P}^{\rm{reg}}$ naturally extends to a line bundle 
\begin{align*}
    \mathcal{P}^{\sharp} \to (\mathcal{M}^L(r, w+1-g^{\rm{sp}})\times_B \mathcal{M}^L(r, \chi))^{\sharp}
\end{align*}
where $(-)^{\sharp}$ means that either one of the two factors is regular. 
If $r=2$ and $l>2g$, 
Li~\cite{MLi} proved that the above line bundle $\mathcal{P}^{\sharp}$ 
uniquely extends to a maximal Cohen-Macaulay sheaf
\begin{align}\label{maxP:intro}
    \mathcal{P} \in \Coh(\mathcal{M}^L(r, w+1-g^{\rm{sp}})\times_B \mathcal{M}^L(r, \chi)). 
\end{align}
We use the above Cohen-Macaulay sheaf to prove the following: 
\begin{thm}\emph{(Corollary~\ref{cor:r=2}, Corollary~\ref{thm:r=2:PS})}\label{thm:intro2}
Let $r=2$ and assume both $\chi$ and $w+l$ are odd.  
Then Conjecture~\ref{conj:intro} holds for $l>\mathrm{max}\{3g-2, 2g\}$, 
and Conjecture~\ref{conj:intro3} holds for $l>2g$. 
\end{thm}
In the situation of Theorem~\ref{thm:intro2}, 
there are no contributions from strictly semistable Higgs bundles, 
so we 
prove a derived equivalence for $L$-twisted Higgs moduli 
spaces. 
For example, when $l>2g$ is even, 
Theorem~\ref{thm:intro2} implies an equivalence 
\begin{align*}
    D^b(M^L_{\mathrm{PGL(2)}}(1), \alpha)\stackrel{\sim}{\to} D^b(M_{\mathrm{SL}(2)}^L(1))
\end{align*}
induced by the maximal Cohen-Macaulay sheaf (\ref{maxP:intro}), where $\alpha$ is a natural Brauer class. 

Equivalence as above have been expected in the study (of mirror symmetry) of 
Higgs bundles. So far, the study of such derived equivalences has been restricted to 
the locus where the spectral curves are reduced~\cite{Ardual, MRVF, MRVF2}, and 
to our knowledge the above equivalence is the first result which extends 
such an equivalence to the global nilpotent cone (i.e. the Hitchin fiber at $0 \in B$). 

Quasi-BPS categories are new categories if there exist strictly semistable Higgs bundles. 
In the rank two and genus zero case, we also prove equivalences which involve 
strictly semistable Higgs bundles: 
\begin{thm}\label{thm:intro3}
Conjecture~\ref{conj:intro} and
Conjecture~\ref{conj:intro3} hold for $r=2$ and $g=0$. 
\end{thm}
We note that for $g=0$, Conjecture~\ref{conj:intro2} and Conjecture~\ref{conj:intro3}
for $L=\Omega_C$ are immediate, see~Example~\ref{exam:1}. 
However, Conjectures~\ref{conj:intro} and
Conjecture~\ref{conj:intro3} are more interesting if the degree of $L$ is large enough, 
as the dimensions of the $L$-twisted Higgs moduli stacks are large and there exist complicated 
singular Hitchin fibers. 

For a higher genus $g>0$ and $L=\Omega_C$, we use the result of Theorem~\ref{thm:intro2} and
then apply matrix factorizations for a regular function to deduce some evidence towards
Conjecture~\ref{conj:intro3}.
This approach is inspired by the proof of $\chi$-independence of the BPS cohomologies for Higgs bundles~\cite{MSendscopic, KinjoKoseki}. However, in our
categorical setting, we only obtain an equivalence up to deformation for $\mathbb{Z}/2$-graded categories: 
\begin{thm}\emph{(Theorem~\ref{thm:eq:can})}\label{thm:intro4}
Let $r=2$, $L=\Omega_C$, and assume that $\chi$ and $w$ are odd. 
Then a $\mathbb{Z}/2$-periodic version of 
Conjecture~\ref{conj:intro3} holds up to deformation equivalence, i.e. 
there is a deformation equivalence 
\begin{align*}
    D^{\mathbb{Z}/2}(M_{\mathrm{PGL(2)}}(1), \alpha)
   \simeq D^{\mathbb{Z}/2}(M_{\mathrm{SL}(2)}(1)). 
\end{align*}
\end{thm}

\subsection{Motivation for the conjectures}
\subsubsection{Categorical $\chi$-independence}
We are interested in the equivalences (\ref{equivT:intro}), (\ref{equivT:intro2})
as in many cases they reveal how quasi-BPS categories look like. For example, 
Conjecture~\ref{conj:intro2} implies that 
\begin{align}\label{equiv:chi:intro}
\mathbb{T}(r, \chi)^{\rm{red}}_1 \simeq \mathbb{T}(r, 1)_{-\chi}^{\rm{red}} \simeq 
D^b(M(r, 1)),
\end{align}
where $M(r, 1)$ is the usual moduli space of Higgs bundles $(F, \theta)$ with 
$\rank(F)=r$ and $\chi(F)=1$, which is a holomorphic 
symplectic manifold. On the other hand, the category $\mathbb{T}(r, \chi)_1^{\rm{red}}$ is hard 
to investigate for non-coprime $(r, \chi)$, see~\cite{PT1} for 
the conjectural structure of quasi-BPS categories for $\mathbb{C}^3$. 

Since (\ref{equiv:chi:intro}) is independent of $\chi$, we regard the equivalence (\ref{equiv:chi:intro}) as a \textit{categorical $\chi$-independence phenomenon}. The (rational) topological K-theory of a BPS category is isomorphic (as a vector space over $\mathbb{Q}$) to the corresponding BPS cohomology~\cite{PThiggs2, PTtop, PTK3}. Thus, by taking the Euler characteristic of the topological K-theory of the two sides of \eqref{equiv:chi:intro}, we recover the equality of BPS invariants of curves of class $r[C]$ and Euler characteristics $\chi$ and $1$, respectively, in the Calabi-Yau threefold $\mathrm{Tot}_C(\Omega_C)\times \mathbb{A}^1_\mathbb{C}$.
%The topological K-theory of a BPS category is isomorphic (as a vector space over $\mathbb{Q}$) to the corresponding BPS cohomology~\cite{PTtop, PTK3}. 
For the most recent developments on cohomological $\chi$-independence,
see~\cite{TodGV, DMJS, KinjoKoseki}. We discuss more about the topological K-theory of quasi-BPS categories of Higgs bundles in \cite{PThiggs2}.

We also motivate the construction of this paper from the point of view of refined Donaldson-Thomas theory in loc. cit., in particular we explain in loc. cit. the name of \textit{quasi-BPS categories} and provide references for BPS invariants/ cohomology.

\subsubsection{Dolbeault Geometric Langlands}
Our conjecture is also inspired by the proposal of Donagi--Pantev of derived equivalence on Higgs moduli stacks for Langlands dual groups~\cite{DoPa}. 
Such an equivalence is regarded as a classical limit of the de Rham geometric Langlands correspondence~\cite[Conjecture~1.1.6]{AG}. 
Recall from loc. cit. that 
for a Langlands dual pair of reductive groups $(G, G^{\vee})$, the de Rham geometric Langlands 
conjecture says that there is an equivalence of categories (with certain properties):
\begin{align}\label{GLC}
\mathrm{Ind}_{\mathcal{N}}D^b(\mathrm{LocSys}_{G^{\vee}}) \simeq \text{D-mod}(\mathrm{Bun}_G). 
\end{align}
Here $\mathrm{LocSys}_{G^{\vee}}$ is the moduli stack of $G^{\vee}$-flat connections, alternatively of local systems, on $C$ 
and $\mathrm{Bun}_G$ is the moduli stack of $G$-bundles on $C$. 
As a classical limit of the equivalence (\ref{GLC}), it is conjectured 
in~\cite[Conjecture~2.5]{DoPa}, ~\cite[Conjecture~1.3]{ZN} 
that there is an equivalence
\begin{align}\label{equiv:DP}
    D_{\rm{qcoh}}(\mathcal{H}iggs_{G^{\vee}}) \simeq D_{\rm{qcoh}}(\mathcal{H}iggs_{G}). 
\end{align}
Here, $\mathcal{H}iggs_G$ is the moduli stack of $G$-Higgs bundles \textit{without stability}, 
thus for $G=GL(r)$ it is not of finite type and much bigger than the union of
$\mathcal{M}(r, \chi)$ for all $\chi \in \mathbb{Z}$. 
As remarked in~\cite[Remark~1.4]{ZN}, the categories in 
(\ref{equiv:DP}) should be suitably modified due to singularities and non-finiteness 
issue as above. 

According to~\cite[Proposition~4.1]{Simp}, 
the stack $\mathrm{LocSys}_{G^{\vee}}$ degenerates to the semistable 
locus $\mathcal{H}iggs_{G^{\vee}}^{\rm{ss}} \subset \mathcal{H}iggs_{G^{\vee}}$ which 
is a stack of finite type. 
It is therefore more natural to consider the derived category of $\mathcal{H}iggs_{G^{\vee}}^{\rm{ss}}$ as a classical limit of the left hand side in (\ref{GLC}). 
Instead of (\ref{equiv:DP}), we expect an equivalence 
\begin{align}\label{equiv:modify}
    D^b(\mathcal{H}iggs_{G^{\vee}}^{\rm{ss}}) \simeq \mathcal{D}(\mathcal{H}iggs_{G}),
\end{align}
where the right hand side is a (yet to be defined) subcategory of $D^b(\mathcal{H}iggs_{G})$. 
Moreover, we expect that
the semiorthogonal decomposition in Theorem~\ref{thm:intro1} for $L=\Omega_C$ 
and $G=\mathrm{GL}(r)$ 
corresponds to a semiorthogonal decomposition in the right hand side induced 
by the Harder-Narasimhan stratification of $\mathcal{H}iggs_{G}$, 
which might also shed light on the Conjectures~\ref{conj:intro2} and \ref{conj:intro3}. 
We will pursue this direction in future research.

\subsubsection{Mirror symmetry for moduli of Higgs bundles}
In the case of $(G, G^{\vee})=(\mathrm{SL}(r), \mathrm{PGL}(r))$, 
if both $(r, \chi)$ and $(r, w)$ are coprime, then 
the pair 
\begin{align*}(M_{\mathrm{PGL}(r)}(w), M_{\mathrm{SL}(r)}(\chi))
\end{align*}
together with some Brauer classes is expected to be a mirror pair~\cite{HauTha}. 
At the numerical or cohomological level, this is proved in~\cite{GrWy, MSendscopic}. 
A categorical version is 
a conjectural derived equivalence
\begin{align}\label{intro:mirror}
    D^b(M_{\mathrm{PGL}(r)}(w), \alpha^{-\chi}) \stackrel{\sim}{\to} D^b(M_{\mathrm{SL}(r)}(\chi)), 
\end{align}
where $\alpha$ is some Brauer class, 
see~\cite[Section~2.4]{HauICM}. 
The statement of Conjecture~\ref{conj:intro3} extends the above 
equivalence to possibly non-coprime $(r, \chi)$, $(r, w)$, but primitive $(r, \chi, w)$, e.g. 
$(r, 0, 1)$ for $r>1$. 
An equivalence of topological K-theories of both sides in 
(\ref{intro:mirror}) is proved by Groechenig--Shen~\cite{GS}.  
In~\cite{PThiggs2}, we will prove a topological K-theory version of 
Conjecture~\ref{conj:intro}, Conjecture~\ref{conj:intro2}
and Conjecture~\ref{conj:intro3}.

\subsection{Acknowledgements}
T.~P. thanks MPIM Bonn and CNRS for their support during part of the preparation of this paper.
This material is partially based upon work supported by the NSF under Grant No. DMS-1928930 and by the Alfred P. Sloan Foundation under grant G-2021-16778, while T.~P. was in residence at SLMath in Berkeley during the Spring 2024 semester. T.~P. thanks Kavli IPMU for their hospitality and excellent conditions during a visit in May 2024.

This work started while Y.~T.~was visiting to Hausdorff Research Institute for Mathematics in Bonn on November 12-18, 2023. 
Y.~T.~thanks the hospitality of HIM
Bonn during his visit. 
Y.~T.~is supported by World Premier International Research Center
	Initiative (WPI initiative), MEXT, Japan, and JSPS KAKENHI Grant Numbers JP19H01779, JP24H00180.

\subsection{Notations and conventions}
In this paper, all the (derived) stacks are defined over $\mathbb{C}$. 
For a derived stack $\X$, we denote by $\X^{\rm{cl}}$ 
its classical truncation. 
For a stack $\X$, we use the notion of \textit{good moduli space}
from~\cite{MR3237451}. It generalizes the notion of GIT quotient 
$R/G \to R\ssslash G$, where $R$ is an affine variety on which 
a reductive group $G$ acts. 

For a locally ringed space $(X, \mathscr{A})$, we denote by 
$\Coh(\mathscr{A})$ the abelian category of coherent right $\mathscr{A}$-modules. 
For a variety or stack $\X$, we denote by $D^b(\X)$ the bounded 
derived category of coherent sheaves on $\X$, which is a pre-triangulated 
dg-category. 
We denote by 
$\mathrm{Perf}(\X)$ the category of perfect complexes, by $D_{\rm{qcoh}}(\X)$ the 
unbounded derived category of quasi-coherent sheaves, and by $D^{-}(\X)\subset D_{\rm{qcoh}}(\X)$ the category of bounded above complexes. 

For stacks $\X_i$ over a base scheme $B$ and 
 pre-triangulated dg-subcategories $\mathcal{C}_i \subset D^b(\X_i)$ for $i=1, 2$, 
we denote by $\mathcal{C}_1\boxtimes_B \mathcal{C}_2$ the smallest 
pre-triangulated dg-subcategory of $D^b(\X_1 \times_B \X_2)$ which contains 
objects $E_1 \boxtimes E_2$ for $E_i \in \mathcal{C}_i$ and closed under direct summands. When $B=\Spec \mathbb{C}$, 
we omit $B$ from the notation. 

We use both of underived and derived functors. For 
a functor $F$, its right/left derived functors 
are denoted by $RF$, $LF$ respectively, following the 
classical notation. If $F$ is an exact functor, we 
omit $R$ or $L$ in the notation. 

For a torus $T$ and for a character $\chi$ and cocharacter $\lambda$ of $T$, 
we denote by $\langle \lambda, \chi \rangle \in \mathbb{Z}$ the natural 
pairing. 
For a $T$-representation $V$, we denote by $V^{\lambda>0} \subset V$ the 
subspace spanned by $T$-weights $\beta$ with $\langle \lambda, \beta\rangle>0$. 
We also write $\langle \lambda, V^{\lambda>0}\rangle:=\langle \lambda, \det(V^{\lambda>0})\rangle$. 

\section{Moduli stacks of Higgs bundles}
In this section, we recall some basic properties of moduli stacks of Higgs bundles, in particular the spectral construction \cite{BeNaRa}, the Hitchin system, and the local description via the étale slice theorem.

\subsection{Twisted Higgs bundles}
Let $C$ be a smooth projective curve of genus $g$. 
Let $L$ be a line bundle on $C$ such that either  
\begin{align*}l:=\deg L>2g-2 \mbox{ or }
L=\Omega_C.
\end{align*}
By definition, a $L$-\textit{Higgs bundle} is a pair 
$(F, \theta)$, where $F$ is a vector bundle on 
$C$ and $\theta$ is a morphism 
\begin{align*}
    \theta \colon F \to F \otimes L. 
\end{align*}
When $L=\Omega_C$, it is just called a Higgs bundle. 
The (semi)stable $L$-twisted Higgs bundle is defined 
using the slope $\mu(F)=\chi(F)/\rank(F)$
in the usual way: a $L$-twisted Higgs bundle $(F, \theta)$ 
is (semi)stable if we have 
\begin{align*}
    \mu(F') <(\leq) \mu(F), 
\end{align*}
for any sub Higgs bundle $(F', \theta') \subset (F, \theta)$
with 
$\rank(F')<\rank(F)$. 

A $L$-twisted Higgs bundle is identified with a 
compactly supported pure one-dimensional 
coherent sheaf on the non-compact surface
\begin{align*}
p \colon 
    S=\mathrm{Tot}_C(L) \to C. 
\end{align*}
The correspondence (called \textit{spectral construction})
is given as follows: 
for a given $L$-twisted Higgs pair $(F, \theta)$, 
the Higgs field $\theta$ determines 
the $p_{\ast}\mathcal{O}_S$-module 
structure on $F$, which in turn 
gives a coherent sheaf on $S$. 
Conversely, a pure one-dimensional 
compactly supported sheaf $E$ on 
$S$ pushes forward to a vector bundle $F$ 
with Higgs field $\theta$ given by the 
$p_{\ast}\mathcal{O}_S$-module structure on it. 

\subsection{Moduli stacks of Higgs bundles}
We denote by 
\begin{align}\label{H:stack}\mathcal{M}^L(r, \chi)
\end{align}
the derived moduli stack of semistable $L$-twisted Higgs bundles 
$(F, \theta)$ with 
\begin{align*}(\rank(F), \chi(F))=(r, \chi).
\end{align*}
It is smooth (in particular classical)
when $\deg L>2g-2$, and 
quasi-smooth when $L=\Omega_C$. 
We omit $L$ in the notation when $L=\Omega_C$. 
We also denote by 
$(\mathcal{F}, \vartheta)$ the universal Higgs bundle 
\begin{align}\label{univ:F}
    \mathcal{F} \in \mathrm{Coh}(C \times \mathcal{M}^L(r, \chi)), \ 
    \vartheta \colon \mathcal{F} \to \mathcal{F} \boxtimes L. 
\end{align}

The stack (\ref{H:stack}) is equipped with the Hitchin map
\begin{align*}
  h \colon \mathcal{M}^L(r, \chi) \to B^L(r,\chi) :=\bigoplus_{i=1}^r
    H^0(C, L^{\otimes i})
\end{align*}
sending $(F, \theta)$ 
to $\mathrm{tr}(\theta^{\otimes i})$
for $1\leq i\leq r$. When $L, r, \chi$ are clear from the context, we only write $B:=B^L(r,\chi)$.
The map $h$ factors through the good moduli space morphism $\pi$: 
\begin{align*}
    h \colon \mathcal{M}^L(r, \chi)^{\rm{cl}} 
    \stackrel{\pi}{\to} M^L(r, \chi) \to B.
\end{align*}
A closed point $y \in M^L(r, \chi)$ corresponds to a polystable 
Higgs bundle 
\begin{align}\label{polystable}
E=\bigoplus_{i=1}^k V_i \otimes E_i
\end{align}
where $E_i$ is a stable $L$-twisted Higgs bundle 
such that $(\rank(E_i), \chi(E_i))=(r_i, \chi_i)$ 
satisfies $\chi_i/r_i=\chi/r$
and 
$V_i$ is a finite dimensional vector space. 
By abuse of notation, we also denote by $y \in \mathcal{M}^L(r, \chi)$ the closed point represented by (\ref{polystable}).
It is the unique closed point in the fiber of 
$\mathcal{M}^L(r, \chi)^{\rm{cl}} \to M^L(r, \chi)$ at $y$. 
We also have the Cartesian square
\begin{align}\label{dia:stable}
\xymatrix{
\mathcal{M}^L(r, \chi)^{\rm{st}} \inclusion \ar[d] & \mathcal{M}^L(r, \chi)^{\rm{cl}} \ar[d] \\
M^L(r, \chi)^{\rm{st}} \inclusion & M^L(r, \chi). 
}
\end{align}
Here $(-)^{\rm{st}}$ means the stable part, the horizontal 
arrows are open immersions and the left vertical arrow is a good 
moduli space morphism which is a $\mathbb{C}^{\ast}$-gerbe.

A point $b \in B$ corresponds to a support
$\mathcal{C}_b \subset S$ of 
the sheaf on $S$, called the \textit{spectral curve}. 
We denote by $g^{\rm{sp}}$ the 
arithmetic genus of the spectral curve, which is given by  
\begin{align}\label{formula:gD}
    g^{\rm{sp}}=1+(g-1)r-\frac{rl}{2}+\frac{r^2 l}{2}. 
\end{align}
Here we note that, for any compactly supported 
effective divisor $D \subset S$, we have 
\begin{align}\label{formula:chiD}
    \chi(\mathcal{O}_D)=-\frac{1}{2}D(K_S+D). 
\end{align}
The formula (\ref{formula:gD}) can be deduced from (\ref{formula:chiD}), alternatively see the computation in \cite[Remark 3.2]{BeNaRa}. 

Let $\mathcal{C} \to B$ be the universal spectral curve, 
which is a closed subscheme of $S \times B$. 
By the spectral construction, the universal Higgs bundle 
corresponds to a universal sheaf 
\begin{align}\label{univ:E}
    \mathcal{E} \in \Coh(\mathcal{C}\times_B \mathcal{M}^L(r, \chi))
\end{align}
which is also regarded as a coherent sheaf on $S \times \mathcal{M}^L(r, \chi)$
by the closed immersion 
$\mathcal{C}\times_B \mathcal{M}^L(r, \chi) \hookrightarrow S \times \mathcal{M}^L(r, \chi)$.

\subsection{Stratification of the Hitchin base}
Consider the open subsets
\begin{align}\label{open:B}
    B^{\rm{ell}} \subset B^{\rm{red}} \subset B
\end{align}
where $B^{\rm{ell}}$ corresponds to irreducible spectral 
curves and $B^{\rm{red}}$ corresponds to reduced 
spectral curves. 
\begin{lemma}\label{lem:codim}
We have 
\begin{align}\label{ineq:reg}
\mathrm{codim}(B \setminus B^{\rm{red}}) \geq 2lr -2l+1-g. 
    \end{align}
    In particular, we have 
    $\mathrm{codim}(B\setminus B^{\rm{red}}) >g^{\rm{sp}}$ if 
    \begin{align}\label{r=23}
     r=2, l>3g-2 \mbox{ or } r=3, l>4g-3. 
    \end{align}
\end{lemma}
\begin{proof}
    We write $B_r=B$. 
    Note that 
    \begin{align}\label{dimBr}
        \dim B_r=\sum_{i=1}^r h^0(L^i)
        =\frac{l}{2}r(r+1)+r(1-g). 
    \end{align}
Also for a decomposition $r=r_1+r_2$, there is an 
addition map
\begin{align*}
   \oplus \colon B_{r_1}\times B_{r_2} \to B_{r}
\end{align*}
defined to be adding the spectral curves. 
        The locus $C_r \subset B_r$ 
    corresponding to divisors on $S$ with multiplicity $r$
    is given by the image of a map 
    \begin{align*}
        C_r=\mathrm{Im}(B_1 \to B_r), \quad x \mapsto 
        \overbrace{x\oplus\cdots\oplus x}^r
    \end{align*}
    and $\dim C_r \leq \dim B_1=l+1-g$. 
    The complement $B \setminus B^{\rm{red}}$ is contained in 
    the image of the map 
    \begin{align*}
    \oplus \colon 
        \bigcup_{r_1+\cdots+r_k=2, r_1 \geq 2}
        (C_{r_1} \times B_{r_2} \times \cdots \times B_{r_k}) \to B_r. 
    \end{align*}
    A stratum given by the image from $C_{r_1} \times B_{r_2} \times \cdots \times B_{r_k}$ 
    corresponds to a divisor on $S$ of the form $D_1+\cdots+D_k$
    where $D_i$ is effective and $D_1$ has multiplicity $r_1 \geq 2$. 
    Among such strata, the stratum with the largest dimension
    is given by $k=2$, $r_1=2$ and $r_2=r-2$. 
    It follows that 
    \begin{align*}
        \dim(B\setminus B^{\rm{red}}) &\leq 
        \dim C_{2}+\dim B_{r-2} \\
        &\leq \frac{l}{2}(r^2-3r+2)+r(1-g)+g+l-1. 
    \end{align*}
        By combining with (\ref{dimBr}), we obtain the inequality (\ref{ineq:reg}). 
        From (\ref{formula:gD}) and (\ref{ineq:reg}), we have 
\begin{align*}
    \mathrm{codim}(B\setminus B^{\rm{red}})-g^{\rm{sp}} \geq -\frac{l}{2}(r-1)(r-4)-gr+r-g. 
\end{align*}
The right hand side is positive if (\ref{r=23}) holds. 
    \end{proof}

\subsection{The local description of moduli stacks of Higgs bundles}\label{subsec:loc}
The good moduli space 
\begin{align}\label{good:pi}
\pi \colon \mathcal{M}^L(r, \chi)^{\rm{cl}} \to M^L(r, \chi)
\end{align}
is, locally on $M^L(r, \chi)$, 
described in terms of the representations of the Ext-quiver. 
In the case of $L=\Omega_C$, the description is similar 
to the case of K3 surfaces, see~\cite[Section~4.3]{PTK3}. 
Here we treat the case of $l>2g-2$. 
In what follows, we write 
\begin{align}\label{rchid}
(r, \chi)=d(r_0, \chi_0)
\end{align}
where $d\in \mathbb{Z}_{>0}$ and $(r_0, \chi_0)$ is coprime. 

Let $y \in M^L(r, \chi)$ be a closed point corresponding 
to the polystable object (\ref{polystable}).
The associated Ext-quiver $Q_y$ has 
vertices $\{1, \ldots, k\}$ 
and the number of edges between two vertices is given by 
\begin{align*}
    \sharp(i \to j)=\dim \Ext_S^1(E_i, E_j)=
    r_i r_j l+\delta_{ij}.
\end{align*}
Here, by the spectral construction, we regard $E_i$ as a coherent 
sheaf on $S$. 
The representation space of $Q_y$-representations of dimension vector $\bm{d}=\{d_i\}_{i=1}^k$ 
for 
$d_i=\dim V_i$ is given by 
\begin{align*}
    R_{Q_y}(\bm{d}):=\bigoplus_{(i\to j) \in Q_y}\Hom(V_i, V_j) =\Ext_S^1(E, E). 
\end{align*}
Let $G(\bm{d})$ be the algebraic group
\begin{align*}
    G(\bm{d}):=\prod_{i=1}^k GL(V_i)=\mathrm{Aut}(E)
\end{align*}
which acts on $R_{Q_y}(\bm{d})$ by the conjugation. 
By Luna étale slice theorem, 
étale locally at $y$ the 
map $\pi \colon \mathcal{M}^L(r, \chi) \to M^L(r, \chi)$
is isomorphic to 
\begin{align}\label{etslice}
R_{Q_y}(\bm{d})/G(\bm{d}) \to R_{Q_y}(\bm{d}) \ssslash G(\bm{d}).     
\end{align}
Note that the above quotient stack is the moduli stack 
of $Q_y$-representations of dimension $\bm{d}$. 

The moduli space $M^L(r, \chi)$ is stratified where 
each stratum is indexed by data $(d_i, r_i, \chi_i)_{1\leq i\leq k}$
of the polystable object (\ref{polystable}). 
The deepest stratum corresponds to 
$k=1$ with $(d_1, r_1, \chi_1)=(d, r_0, \chi_0)$, 
which consist of polystable objects
$V\otimes E_0$ where $\dim V=d$ and $E_0$ is stable 
with 
\begin{align*}(\rank(E_0), \chi(E_0))=(r_0, d_0).
\end{align*}
The associated Ext-quiver at the deepest 
stratum has one vertex and $(1+lr_0^2)$-loops. 
We have the following lemma, 
also see~\cite[Lemma~4.3]{PTK3} for an analogous
    statement for K3 surfaces. 

\begin{lemma}\label{lem:equiver}
    For each closed point $y \in M^L(r, \chi)$, 
    and a closed point 
    $x \in M^L(r, \chi)$ which lies in the deepest
    stratum, there exists a closed point 
    $y' \in M^L(r, \chi)$ which is sufficiently close to $x$
    such that $Q_y=Q_{y'}$. 
\end{lemma}
\begin{proof}
    Let $y$ corresponds to the polystable object (\ref{polystable})
    with $(r_i, \chi_i)=m_i(r_0, \chi_0)$ for $m_i \in \mathbb{Z}_{>0}$. 
    Let $R_i$ be a simple $Q_x$-representation with dimension $m_i$. 
    Recall that $Q_x$ is a one vertex quiver with 
    loops $1+lr_0^2 \geq 2$, so there exists 
    such a simple representation $R_i$. 
    Then we have 
    \begin{align*}\dim \Ext_{Q_x}^1(R_i, R_j)=m_i m_j r_0^2+\delta_{ij},
    \end{align*}
    hence the Ext-quiver 
    for $R=\bigoplus_{i=1}^k V_i \otimes R_i$ equals $Q_y$. 
    The $Q_x$-representation $R$ corresponds to a point $p \in R_{Q_x}(d)\ssslash G(d)$, 
which can be taken to be sufficiently close to $x$ 
 by applying the $\mathbb{C}^{\ast}$-action on the loops of $Q_x$. 
 By the \'etale slice theorem, the point $p$
 corresponds to $y' \in M^L(r, \chi)$ which is sufficiently 
 close to $x$  
    such that $Q_{y'}=Q_y$. 
       \end{proof}

\subsection{The canonical line bundles on Higgs moduli spaces}
We show that $M^L(r, \chi)$ is Gorenstein with trivial canonical line bundle.
The following lemma is probably well-known, but we include it here as we cannot 
find a reference. 
\begin{lemma}\label{lem:gorenstein}
If $l>0$, then $M=M^L(r, \chi)$ is Gorenstein with trivial  
dualizing sheaf $\omega_M=\mathcal{O}_M$, and the map
(\ref{good:pi}) is generically a 
$\mathbb{C}^{\ast}$-gerbe. 
\end{lemma}
\begin{proof}
We first show that $M$ is Gorenstein. 
For $L=\Omega_C$, the condition $l>0$ is equivalent to $g\geq 2$. 
Then the argument is the same 
for the K3 surface case, see~\cite[Remark~7.1]{PTK3}. 
In the case that $l>2g-2$, 
we prove the lemma except the case 
that $l=1$, $r=2$ and $\chi$ even, 
which will be proven in Lemma~\ref{subsec:proof1} separately. 
By the local description (\ref{etslice})
and Lemma~\ref{lem:equiver}, it is enough to prove 
that $M$ is Gorenstein at a point in the deepest stratum. 
Namely we need to check that the GIT quotient 
\begin{align}\label{Gquotient}
\mathfrak{gl}(V)^{\oplus (1+lr_0^2)}\ssslash GL(V)
\end{align}
is Gorenstein, where $\dim V=d$. 
This is proved in~\cite[Lemma~5.7]{PTquiver} (and the argument in it)
when $1+lr_0^2 \geq 3$ or $1+lr_0^2=2$ and $d\geq 3$. 
The only exceptional case is $1+lr_0^2=2$
and $d=2$, or equivalently $l=1$, $r=2$ and $\chi$ is even, 
where the locus in $\mathfrak{gl}(V)^{\oplus 2}$
with positive dimensional stabilizers in $PGL(d)$ is of codimension one
(while it is of codimension bigger than or equal to two in the other cases). 
The above argument also shows that the map (\ref{good:pi}) is 
generically a $\mathbb{C}^{\ast}$-gerbe. 

The canonical line bundle of the stack $\mathcal{M}^L(r, \chi)$ is trivial, 
see~\cite[Proposition~3.8]{KinjoKoseki}. 
Therefore the canonical line bundle on the smooth part $M^{\rm{sm}} \subset M$
is trivial. Since the singular locus has codimension bigger than or equal to two, 
and $M$ is Gorenstein, it follows that
$\omega_M=\mathcal{O}_M$. 
\end{proof}

The following is a treatment of the exceptional case in the 
previous lemma. This is just for the completeness, and the 
readers can skip it at the first reading. 

\begin{lemma}\label{subsec:proof1}
Lemma~\ref{lem:gorenstein} holds for $l=1$, $r=2$ and even $\chi$. 
\end{lemma}
\begin{proof}
We may assume that $\chi=2$. 
   In the case that $l=1$, $r=2$ and $\chi=2$, 
   the GIT quotient (\ref{Gquotient}) 
   is 
   \begin{align*}
       \mathfrak{gl}(V)^{\oplus 2}\ssslash GL(V)
       \stackrel{\cong}{\to} \mathbb{C}^5
   \end{align*}
   where $\dim V=2$, and the above map is 
   \begin{align*}
       (A_1, A_2) \mapsto (\mathrm{tr} A_1, \mathrm{tr} A_2, \det A_1, \det A_2, \mathrm{tr}(A_1 A_2)). 
   \end{align*}
   Therefore $M^L(2, 0)$ is smooth, in particular 
   it is Gorenstein. 
   We are left to show that $M^L(2, 0)$ has trivial 
   canonical line bundle. 

   As we assumed $l>2g-2$, the case $l=1$ happens
   only when $g=0$ or $g=1$. In the case of $g=0$, 
   the stack $\mathcal{M}^L(2, 2)$ consists of 
   $\mathcal{O}_{\mathbb{P}^1} \otimes V \to 
   \mathcal{O}_{\mathbb{P}^1}(1) \otimes V$
   for $\dim V=2$, 
   thus 
   \begin{align*}
       M^L(2, 2) \cong \mathfrak{gl}(V)^{\oplus 2}\ssslash GL(V). 
   \end{align*}
   As we observed above, the right hand side is
   isomorphic to $\mathbb{C}^5$, in particular has a trivial canonical line bundle. 

   Suppose that $g=1$. In this case, the Hitchin map
   is \[M=M^L(2, 0) \to B=\mathbb{C}^3.\] 
As the reduced locus $B^{\rm{red}} \subset B$ has 
complement at least codimension two, and the 
Hitchin map is flat, it is enough to show that 
the relative canonical bundle $\omega_{M/B}$ is 
trivial \'etale locally at any point in $B^{\rm{red}}$. 
Let $b \in B^{\rm{red}}$ corresponds to a spectral 
curve $\mathcal{C}_b$. It is enough to consider 
an \'etale neighborhood at $b$ where $\mathcal{C}_b$ is 
not irreducible, i.e. $\mathcal{C}_b=C_1 \cup C_2$
and $C_1 \neq C_2$.
Note that each $C_i$ is isomorphic to the smooth 
elliptic curve $C$ and we have that $C_1 \cdot C_2=1$. 
Let $\mathcal{C} \to B$ be the universal spectral curve. 
There is an \'etale neighborhood $g \colon B' \to B^{\rm{red}}$
such that $\mathcal{C}':=\mathcal{C}\times_B B'$
admits a divisor $D$ with $D\cdot C_1=1$, $D\cdot C_2=0$. 
Let $H$ be a divisor on $\mathcal{C}'$ 
given by pull-back of an ample divisor on $C$ via projections
$\mathcal{C}' \to \mathcal{C} \to C$. Set 
$H_{\varepsilon}:=H+\varepsilon D$ for $0<\varepsilon \ll 1$. 
Let $\mathcal{M}_{\varepsilon}$ be the relative moduli stack of 
rank one $H_{\varepsilon}$-stable sheaves on $\mathcal{C}' \to B'$ with Euler characteristic $2$. 
Then it is an open substack of $\mathcal{M}^L(2, 0)\times_B B'$. The composition 
\begin{align}\label{map:compose}
    \mathcal{M}_{\varepsilon} \hookrightarrow 
    \mathcal{M}^L(2, 0)\times_B B' \to M^L(2, 0)\times_B B'
\end{align}
is bijective on closed points. 
In fact, for any strictly polystable 
sheaf $L_1 \oplus L_2$ on on $\mathcal{C}_b$,
where $L_i$ is a line bundle on $C_i$ with degree one, 
there is a unique non-trivial extension 
$0\to L_1 \to E \to L_2 \to 0$, giving 
a unique closed point in the 
fiber of the map (\ref{map:compose}).
Thus (\ref{map:compose}) is a $\mathbb{C}^{\ast}$-gerbe, 
and as $\mathcal{M}^L(2, 0)\times_B B'$ has trivial 
canonical bundle by~\cite[Proposition~3.8]{KinjoKoseki}, it follows that 
$M^L(2, 0)\times_B B'$ also has trivial canonical bundle.

\end{proof}

\section{Quasi-BPS categories for Higgs bundles}
In this section, we introduce quasi-BPS and BPS categories for $L$-twisted 
Higgs bundles, and we begin their study. 
Throughout this section, let $(r_0,\chi_0)\in \mathbb{Z}_{\geq 1}\times \mathbb{Z}$ be coprime, let $d\in \mathbb{Z}_{\geq 1}$, and let $(r,\chi):=d(r_0,\chi_0)$.

The main results discussed in this section are the semiorthogonal decomposition of the derived category of the stack of semistable Higgs bundles in products of quasi-BPS categories (Theorem \ref{thm:sod}) and an analogue of the main theorem in \cite{PTK3} for Higgs bundles, which says that BPS categories are smooth over $\mathbb{C}$, and proper and Calabi-Yau over the Hitchin base (Theorems \ref{thm:proper0} and \ref{thm:proper}). The methods and results are analogous to those in \cite{PTK3}, but the results are slightly stronger (see the proof of Theorem \ref{thm:proper}) than in loc. cit. because of the global description of the stack of semistable Higgs bundles as a critical locus \cite{KinjoMasuda, MSendscopic}.

\subsection{Quasi-BPS categories for $l>2g-2$}
We consider the bounded derived category of coherent sheaves 
$D^b(\mathcal{M}^L(r, \chi))$. 
Each sheaf parametrized by $\mathcal{M}^L(r, \chi)$ has scalar automorphisms $\mathbb{C}^{\ast}$, so $\mathcal{M}^L(r, \chi)$ is naturally a $\mathbb{C}^{\ast}$-gerbe.
Thus there is an orthogonal decomposition 
\begin{align*}
 D^b(\mathcal{M}^L(r, \chi))=\bigoplus_{w\in \mathbb{Z}}
 D^b(\mathcal{M}^L(r, \chi))_w,
\end{align*}
where $D^b(\mathcal{M}^L(r, \chi))_w$ is the subcategory of $D^b(\mathcal{M}^L(r, \chi))$ of complexes of weight $w$ with respect to the scalar automorphisms $\mathbb{C}^{\ast}$.

We next define the \textit{quasi-BPS category} of $D^b(\mathcal{M}^L(r, \chi))_w$. 
The case of $L=\Omega_C$ is defined as in the case of K3 surfaces~\cite[Section~4.4]{PTK3}, 
and will be discussed in Subsection~\ref{subsec:qbps:nc}. 
In this subsection, we focus on the 
case that $l>2g-2$. 

We first construct some line bundle on $\mathcal{M}^L(r, \chi)$. 
We write $(r, \chi)=d(r_0, \chi_0)$ as in (\ref{rchid}), 
and take $(a, b) \in \mathbb{Z}^2$ such that 
\begin{align}\label{cond:uv}
    a\chi_0+br_0+(1-g)ar_0=1
\end{align}
which is possible as $(r_0, \chi_0)$ is coprime. 
%Consider the numerical K-theory $K^{\mathrm{num}}(C):=\mathbb{Z}^{\oplus 2}$.
Let  
$u\in K(C)$ be such that $(\rank(u), \chi(u))=(a, b)$. 
We define the following line bundle on $\mathcal{M}^L(r, \chi)$
\begin{align}\label{delta}
    \delta:=\det(Rp_{\mathcal{M}\ast}(u\boxtimes \mathcal{F})) \in 
    \mathrm{Pic}(\mathcal{M}^L(r, \chi)). 
\end{align}
Here $\mathcal{F}$ is the universal Higgs bundle (\ref{univ:F}) and $p_{\mathcal{M}}$ is the 
projection onto $\mathcal{M}^L(r, \chi)$. It has diagonal $\mathbb{C}^{\ast}$-weight $d$ 
because of the condition (\ref{cond:uv})
and the Riemann-Roch theorem. 

We next introduce some general notation.
An object $A \in D^b(B\mathbb{C}^{\ast})$ decomposes into a direct sum $\bigoplus_{w\in \mathbb{Z}}A_w$ where $A_{w}$ is of $\mathbb{C}^{\ast}$-weight
$w$. Denote by $\mathrm{wt}(A)$ the set of $w \in \mathbb{Z}$
such that $A_{w} \neq 0$. 
In the case that $A$ is a line bundle on $B\mathbb{C}^{\ast}$, then $\mathrm{wt}(A)$ 
consists of one element $\mathrm{wt}(A) \in \mathbb{Z}$. 
We also write $A^{>0}:=\oplus_{w>0}A_{w}$. 

\begin{defn}\label{def:qbps}
In the case of $l>2g-2$, define the 
quasi-BPS category 
\begin{align}\label{T:qbps}
    \mathbb{T}^L(r, \chi)_w \subset D^b(\mathcal{M}^L(r, \chi))_w. 
\end{align}
to be the subcategory of objects $\mathcal{E}$ such that, for any 
map $\nu \colon B\mathbb{C}^{\ast} \to \mathcal{M}=\mathcal{M}^L(r, \chi)$, we have 
\begin{align}\label{cond:qbps}
\mathrm{wt}(\nu^{\ast}\mathcal{E}) \subset 
\left[-\frac{1}{2}\mathrm{wt} \det ((\nu^{\ast}\mathbb{L}_{\mathcal{M}})^{>0}), 
\frac{1}{2}\mathrm{wt} \det ((\nu^{\ast}\mathbb{L}_{\mathcal{M}})^{>0})
\right]+\frac{w}{d}\mathrm{wt}(\nu^{\ast}\delta). 
\end{align}
\end{defn}
\begin{remark}\label{rmk:nu}
Under the assumption $l>2g-2$, the stack $\mathcal{M}$ is smooth, 
so $\mathrm{wt}(\nu^{\ast}\mathcal{E})$ consists of a finite set of integers. 
\end{remark}
\begin{remark}\label{rmk:qbps}
The subcategory (\ref{T:qbps}) is independent of a choice 
of $(a, b)$ satisfying the condition (\ref{cond:uv})
and $u\in K(C)$ with $(\rank(u), \chi(u))=(a, b)$. 
See the argument of~\cite[Lemma~4.6]{PTK3}. 
\end{remark}

\begin{remark}\label{rmk:coprime}
If $(r, \chi)$ is coprime, then 
$\mathcal{M}^L(r, \chi)$ consists of stable 
$L$-twisted Higgs bundles, and any map $\nu \colon B\mathbb{C}^{\ast} \to \mathcal{M}^L(r, \chi)$ corresponds to 
a scalar automorphism of stable $L$-twisted Higgs bundles. 
In this case, the good moduli space 
morphism $\mathcal{M}^L(r, \chi) \to M^L(r, \chi)$
is a trivial $\mathbb{C}^{\ast}$-gerbe
because of the existence of the line bundle (\ref{delta}) 
of $\mathbb{C}^{\ast}$-weight one. 
Therefore, in this case we have 
\begin{align*}
    \mathbb{T}^L(r, \chi)_w=D^b(\mathcal{M}^L(r, \chi))_w \simeq D^b(M^L(r, \chi)). 
\end{align*}
    However if there exists a strictly semistable 
    $L$-twisted Higgs bundles, the condition (\ref{cond:qbps}) 
    is non-trivial 
    and the quasi-BPS category is a strict subcategory of 
    $D^b(\mathcal{M}^L(r, \chi))_w$. 
\end{remark}

The following lemma is immediate from Definition~\ref{def:qbps}. 
\begin{lemma}\label{lem:qbps}
Suppose that $l>2g-2$. An object $\mathcal{E} \in D^b(\mathcal{M}^L(r, \chi))$ lies 
in $\mathbb{T}^L(r, \chi)_w$ if and only if for any closed point 
$y \in \mathcal{M}^L(r, \chi)$ corresponding to the polystable object (\ref{polystable}), 
one parameter subgroup $\lambda \colon \mathbb{C}^{\ast} \to T(\bm{d}) \subset G(\bm{d})$
where $T(\bm{d})$ is the maximal torus, and any $T(\bm{d})$-weight $\chi$ of $\mathcal{E}|_{y}$, 
we have 
\begin{align}\label{ineq:n}
    -\frac{n_{\lambda}}{2} \leq \left\langle \lambda, \chi-\frac{w}{d}\delta_y \right\rangle \leq 
    \frac{n_{\lambda}}{2}. 
\end{align}
Here, $n_{\lambda}$ and $\delta_y$ are defined by 
\begin{align}\notag
    n_{\lambda}:=\big\langle \lambda, \det(R_{Q_y}(\bm{d})^{\lambda>0})-\det(\mathfrak{g}(\bm{d})^{\lambda>0}) \big\rangle, \ 
    \delta_y:=\bigotimes_{i=1}^k \det(V_i)^{\otimes m_i},
\end{align}
where $m_i$ is determined by $(r_i, \chi_i)=m_i(r_0, \chi_0)$
and $\mathfrak{gl}(\bm{d})$ is the Lie algebra of $G(\bm{d})$. 
\end{lemma}

\begin{remark}\label{deltax}
The $G(\bm{d})$-character $\delta_y$ is 
the restriction of the line bundle $\delta$ in (\ref{delta}) at $y$. 
\end{remark}

\begin{remark}\label{rmk:domi}
For an irreducible $G(\bm{d})$-representation $V$, 
any of its $T(\bm{d})$-weights $\chi$ satisfies the condition (\ref{ineq:n})
if and only if the highest weight $\gamma$ of $V$ satisfies 
\begin{align}\label{weight:W}
\gamma+\rho \in \frac{1}{2}\mathrm{sum}[0, \beta]+\frac{w}{d}\delta_y \subset M(\bm{d})_{\mathbb{R}},
\end{align}
where 
$M(\bm{d})$ is the weight lattice of $T(\bm{d})$, 
$\mathrm{sum}[0, \beta]$ is the Minkowski sum of 
$T(\bm{d})$-weights in $R_{Q_y}(\bm{d})$, and $\rho$ is half the sum of 
positive roots, see~\cite[Lemma~2.9]{hls}. 
\end{remark}

We have the following immediate periodicity of 
quasi-BPS categories: 
\begin{lemma}\label{lem:period}
For $(a, b) \in \mathbb{Z}^2$, there is an equivalence 
\begin{align}\label{equiv:ab}
\mathbb{T}^L(r, \chi)_w \stackrel{\sim}{\to} \mathbb{T}^L(r, \chi+ar)_{w+bd}.
\end{align}
    \end{lemma}
    \begin{proof}
    It is enough to construct equivalences for $(a, b)=(1, 0)$ and $(a, b)=(0, 1)$. 
    For a closed point $p\in C$, there is an isomorphism of stacks 
    \begin{align*}
        \mathcal{M}^L(r, \chi)\stackrel{\cong}{\to} \mathcal{M}^L(r, \chi+r)
    \end{align*}
    given by $(F, \theta) \mapsto (F(p), \theta(p))$.
    The push-forward by the above isomorphism gives an equivalence (\ref{equiv:ab})
    for $(a, b)=(1, 0)$. 
    The equivalence for $(a, b)=(0, 1)$ is given by 
    the tensor product with the determinant line bundle $\delta$ defined by (\ref{delta}). 
    \end{proof}

\subsection{Quasi-BPS categories for $L=\Omega_C$}\label{subsec:qbps:nc}
We next define and study quasi-BPS categories for the moduli spaces of 
usual Higgs bundles, i.e. $L=\Omega_C$. 
One difference between this and the $l>2g-2$ case is that 
$\mathcal{M}(r, \chi)$ is a singular derived stack. 
Since $S=\mathrm{Tot}_C(\Omega_C)$ is a (non-compact) Calabi-Yau surface, 
the quasi-BPS categories in this case can be defined 
similarly to the case of K3 surfaces~\cite{PTK3}, i.e. 
defined to be intrinsic window subcategories~\cite[Subsection~6.2.2]{T}. We note that this definition is inspired by the work of Halpern-Leistner~\cite{HalpK32}. 
In the case of Higgs bundles, the relevant moduli stacks are 
a \textit{global} derived zero locus \cite[Corollary 1.2]{KinjoMasuda}, \cite[Theorem 4.5]{MSendscopic}, as we recall below.

Fix $p\in C$. There is a closed embedding 
\begin{align}\label{emb:M}
j \colon  \mathcal{M}(r, \chi) \hookrightarrow 
    \mathcal{M}^{\Omega_C(p)}(r, \chi)
\end{align}
sending 
$(F, \theta)$ 
for $\theta \colon F \to F \otimes \Omega_C$ to 
$(F, \theta')$, where $\theta'$ is the composition 
\begin{align*}
\theta' \colon 
    F \stackrel{\theta}{\to} F \otimes \Omega_C \hookrightarrow 
    F \otimes \Omega_C(p). 
\end{align*}
A Higgs bundle $(F, \theta')$ in $\mathcal{M}^{\Omega_C(p)}(r, \chi)$ is in the image of (\ref{emb:M})
if and only if $\theta'|_{p} \colon F|_p \to F|_{p}\otimes \Omega_C(p)|_{p}$ is zero. 
Globally, let $(\mathcal{F}, \vartheta)$ be the universal 
Higgs bundle (\ref{univ:F}) for $L=\Omega_C(p)$, 
and set 
\begin{align*}
\mathcal{F}_p :=\mathcal{F}|_{p\times \mathcal{M}^{\Omega_C(p)}(r, \chi)} \in \mathrm{Coh}(\mathcal{M}^{\Omega_C(p)}(r, \chi)).
\end{align*}
By fixing an isomorphism $\Omega_C(p)|_{p} \cong \mathbb{C}$, 
the correspondence 
$(F, \theta') \mapsto (F|_{p}, \theta'|_{p})$ gives a section 
$s$ of the vector bundle 
\begin{align}\label{sec:s}
\xymatrix{
    \mathcal{V}:=\mathcal{E}nd(\mathcal{F}_p) \ar[r] & \ar@/_18pt/[l]^s \mathcal{M}^{\Omega_C(p)}(r, \chi). 
    }
\end{align}
Then there is an equivalence of derived stacks 
\begin{align}\label{equiv:stack}
    \mathcal{M}(r, \chi) \stackrel{\sim}{\to} s^{-1}(0),
\end{align}
where the right hand side is the derived zero locus of the section $s$. 
\begin{defn}\label{def:qbps2}
Suppose that $L=\Omega_C$. The \textit{quasi-BPS category} 
\begin{align}\label{def:qbps:T}
    \mathbb{T}(r, \chi)_w \subset D^b(\mathcal{M}(r, \chi))_w
\end{align}
is the subcategory with
objects $\mathcal{E} \in D^b(\mathcal{M}(r, \chi))_w$ such 
that, 
for all $\nu \colon B\mathbb{C}^{\ast} \to \mathcal{M}(r, \chi)$, 
we have 
\begin{align}\label{cond:nu}
\mathrm{wt}(\nu^{\ast}j^{\ast}j_{\ast}\mathcal{E}) \subset 
\left[-\frac{1}{2}\mathrm{wt} \det (\nu^{\ast}\mathbb{L}_{\mathcal{V}})^{\nu>0}, 
\frac{1}{2}\mathrm{wt} \det (\nu^{\ast}\mathbb{L}_{\mathcal{V}})^{\nu>0}
\right]+\frac{w}{d}\mathrm{wt}(\nu^{\ast}\delta). 
\end{align}
Here $j$ is the closed immersion (\ref{emb:M}). 
\end{defn}

\begin{remark}
    The subcategory (\ref{def:qbps:T}) is intrinsic to the 
    derived stack $\mathcal{M}(r, \chi)$, and in particular independent 
    of a choice of $p\in C$. See~\cite[Subsection~6.2.2]{T}. 
\end{remark}

\begin{remark}\label{rmk:coprime2}
If $(r, \chi)$ are coprime, then, similarly to 
Remark~\ref{rmk:coprime}, we have 
\begin{align*}
    \mathbb{T}(r, \chi)_w \simeq 
    D^b(M(r, \chi) \times \Spec \mathbb{C}[\epsilon]),
\end{align*}
where $\deg \epsilon=-1$. 
The moduli space $M(r, \chi)$ is the usual moduli 
space of Higgs bundles, which is a holomorphic symplectic 
manifold. Contrary to the case of $l>2g-2$ in Remark~\ref{rmk:coprime}, 
there is a derived structure $\Spec \mathbb{C}[\epsilon]$
because $\Ext^2_S(F, F)=\mathbb{C}$ for a stable 
Higgs bundle $F$. 
\end{remark}

\subsection{The semiorthogonal decomposition for Higgs bundles}\label{subsec:sod}
For a decomposition 
$d=d_1+\cdots+d_k$, 
let $\mathcal{F}il^L(d_1, \ldots, d_k)$ be the 
(derived) moduli stack of filtrations
of $L$-twisted Higgs bundles
\begin{align}\label{higgs:filt}
    0=E_0 \subset E_1 \subset \cdots \subset E_k
\end{align}
where $E_i/E_{i-1}$ has $(\rank(E_i/E_{i-1}), \chi(E_{i}/E_{i-1}))=d_i(r_0, \chi_0)$. 
There are natural morphisms 
\begin{align*}
    \times_{i=1}^k \mathcal{M}^L(d_i(r_0, \chi_0)) \stackrel{q}{\leftarrow} 
    \mathcal{F}il^L(d_1, \ldots, d_k) \stackrel{p}{\to} \mathcal{M}^L(r, \chi), 
\end{align*}
where $p$ sends a filtration (\ref{higgs:filt}) to $E_k$
and $q$ sends (\ref{higgs:filt}) to $(E_i/E_{i-1})_{i=1}^k$. 
Note that $p$ is proper and $q$ is quasi-smooth. 
Therefore the 
categorical Hall product~\cite{PoSa} is defined by  
\begin{align}\label{cathall}
\ast:=Rp_{\ast}Lq^{\ast} \colon 
\boxtimes_{i=1}^k
D^b(\mathcal{M}^L(d_i(r_0, \chi_0)))   
\to D^b(\mathcal{M}^L(r, \chi)). 
\end{align}

\begin{thm}\label{thm:sod}
Let $(r,\chi)\in \mathbb{Z}_{\geq 1}\times\mathbb{Z}$ and let $w\in \mathbb{Z}$.
There is a semiorthogonal decomposition 
\begin{align}\label{sod:main}
    D^b(\mathcal{M}^L(r, \chi))_w
    =\left\langle \boxtimes_{i=1}^k \mathbb{T}^L(r_i, \chi_i)_{w_i} \,\bigg|\, 
    \frac{v_1}{r_1}<\cdots<\frac{v_k}{r_k}\right\rangle. 
\end{align}
The right hand side is after all partitions 
$(r, \chi, w)=(r_1, \chi_1, w_1)+\cdots+(r_k, \chi_k, w_k)$ with 
$\chi/r=\chi_i/r_i$, and 
$v_i \in \frac{1}{2}\mathbb{Z}$ are defined by 
\begin{align}\label{def:wi}
    v_i:=w_i-\frac{l}{2}r_i \left(\sum_{i>j}r_j-\sum_{i<j}r_j   \right). 
\end{align}
The fully-faithful functor 
\begin{align*}
\boxtimes_{i=1}^k \mathbb{T}^L(r_i, \chi_i)_{w_i}
\to D^b(\mathcal{M}^L(r, \chi))_w
\end{align*}
is given by the restriction of the categorical Hall product (\ref{cathall}). The order of the summands is as in \cite{PTK3, PT0}.
\end{thm}
\begin{proof}
In the case of $L=\Omega_C$, the proof is the 
same as in the case of K3 surface, see~\cite[Theorem~5.1]{PTK3}. 

The proof for $l>2g-2$ case is also the same, so we only give a sketch of the proof. Using Lemma~\ref{lem:equiver},
we can reduce it to the semiorthogonal decomposition of the local 
case at a point in the deepest stratum. 
Let $y\in M^L(r, \chi)$ lies in the deepest stratum, 
and $Q_y$ the associated Ext-quiver. 
Then the quotient stack 
\begin{align*}\mathcal{X}_{Q_y}(d)=\mathfrak{gl}(V)^{\oplus (l r_0^2+1)}/GL(V)
\end{align*}
for $\dim V=d$ 
together with its good moduli space $\mathcal{X}_{Q_y}(d) \to 
X_{Q_y}(d)$
gives a local model of $\mathcal{M}^L(r, \chi) \to M^L(r, \chi)$ at $y$, see Subsection~\ref{subsec:loc}. 
There is a semiorthogonal decomposition~\cite[Theorem~4.2]{PTquiver}:
\begin{align}\label{sod:local}
    D^b(\mathcal{X}_{Q_y}(d))_w=\left\langle \boxtimes_{i=1}^k \mathbb{T}_{Q_y}(d_i)_{w_i} \,\bigg|\, 
    \frac{v_1}{d_1}<\cdots <\frac{v_k}{d_k} \right\rangle,
\end{align}
where the right hand side is after all partitions $(d,w)=(d_1,w_1)+\cdots+(d_k,w_k)$ such that $v_i \in \frac{1}{2}\mathbb{Z}$
defined in (\ref{def:wi}) satisfy the above inequality. 
The subcategory 
\begin{align}\label{qbps:loc}
    \mathbb{T}_{Q_y}(d)_w \subset D^b(\mathcal{X}_{Q_y}(d))_w
\end{align}
is generated by $\Gamma \otimes \mathcal{O}_{\mathcal{X}_{Q_y}}(d)$
for $GL(V)$-representations $\Gamma$ whose $T(d)$-weights satisfy the 
condition (\ref{ineq:n}). 
As in the proof of~\cite[Theorem~5.1]{PTK3}, we can reduce 
the semiorthogonal decomposition (\ref{sod:main}) to the local one (\ref{sod:local}). 
\end{proof}

\begin{remark}
If $lr_0^2 \in 2\mathbb{Z}$, 
the condition $w_i \in \mathbb{Z}$ is equivalent to $v_i \in \mathbb{Z}$.
So we have $v_i \in \mathbb{Z}$ in (\ref{sod:main})
if either $l$ or $r_0$ is even. 
\end{remark}

\begin{remark}\label{rmk:open}
The construction of the category (\ref{T:qbps}) is Zariski (or even étale) local over $M=M^L(r, \chi)$. It follows that, for any open subset $U \subset M^L(r, \chi)$,
there is an associated subcategory 
\begin{align*}
    \mathbb{T}^L(r, \chi)_{w}|_{U} \subset D^b(\mathcal{M}^L(r, \chi)\times_{M} U), 
\end{align*}
defined as in Definition~\ref{def:qbps}. 
Similarly, the construction is also local on $B$, so for 
any open subset $U' \subset B$, we have the associated subcategory 
\begin{align*}
    \mathbb{T}^L(r, \chi)_{w}|_{U'} \subset D^b(\mathcal{M}^L(r, \chi)\times_{B} U'). 
\end{align*}
Also see Lemma~\ref{lem:sdo} for the base change property of semiorthogonal decompositions.  
\end{remark}

\begin{remark}\label{remark316}
We will also use the version for quasi-coherent sheaves
\begin{align*}
    \mathbb{T}_{\rm{qcoh}}^L(r, \chi)_w \subset D_{\rm{qcoh}}(\mathcal{M}^L(r, \chi)). 
\end{align*}
It is defined to be the full subcategory split generated by $\mathbb{T}^L(r, \chi)_w$, also see Subsection~\ref{subsec:basechange}.     
\end{remark}
\begin{remark}\label{rmk:open2}
For $l>2g-2$, it follows directly from the definition of (\ref{T:qbps}) that: 
\begin{align*}
    \mathbb{T}^L(r, \chi)_w|_{M^L(r, \chi)^{\rm{st}}} \simeq 
    D^b(\mathcal{M}^L(r, \chi)^{\rm{st}})_w \simeq D^b(M^L(r, \chi)^{\rm{st}}, \alpha^w)
\end{align*}
where $\alpha$ is the Brauer class classifying the $\mathbb{C}^{\ast}$-gerbe 
in the left vertical arrow in (\ref{dia:stable}), and the right hand side is the 
derived category of $\alpha^w$-twisted sheaves, see~\cite{MR2700538, MR2309155}. 
Similarly, for $L=\Omega_C$, we have that
\begin{align*}
    \mathbb{T}(r, \chi)_w|_{M(r, \chi)^{\rm{st}}} \simeq 
    D^b(\mathcal{M}(r, \chi)^{\rm{st}})_w \simeq D^b(M(r, \chi)^{\rm{st}}\times \mathrm{Spec}\,\mathbb{C}[\epsilon], \alpha^w),
\end{align*} where $\deg \epsilon=-1$ and $\alpha$ is a Brauer class as above. 
\end{remark}

\subsection{The semiorthogonal decomposition for vector bundles on a curve}\label{subsec:bun}
The result of this subsection is not used later in the paper, but it may be of independent interest.
We discuss a version of Theorem \ref{thm:sod} for the moduli stack of slope semistable vector bundles $\mathcal{B}un^{\mathrm{ss}}(r,\chi)$ of rank $r$ and degree $\chi$ on a smooth projective curve $C$ of genus $g\geq 1$. Its proof is completely analogous to the proof of Theorem \ref{thm:sod}, based on the local description of the good moduli space 
\[\mu\colon \mathcal{B}un^{\mathrm{ss}}(r,\chi)\to \mathrm{Bun}^{\mathrm{ss}}(r,\chi)\] in terms of Ext quivers, that we now recall.
Write $(r,\chi)=d(r_0,\chi_0)$ with $(r_0,\chi_0)$ coprime. Let $y\in \mathrm{Bun}^{\mathrm{ss}}(r,\chi)$ be a closed point corresponding to the polystable vector bundle
\[E=\bigoplus_{i=1}^k V_i\otimes E_i.\] In the above, $E_i$ is a stable vector bundle of rank $r_i$ and degree $\chi_0$, $E_i$ and $E_j$ are not isomorphic for $1\leq i\neq j\leq k$, and $V_i$ is a finite dimensional vector space of dimension $d_i$. 
The Ext quiver of $y$ is the quiver $Q_y$ with vertices $\{1,\ldots, k\}$ and with the following number of edges between two vertices:
\begin{align*}
    \sharp(i \to j)=\dim \Ext_C^1(E_i, E_j)=
    r_i r_j (g-1)+\delta_{ij}.
\end{align*}
As in Subsection \ref{subsec:loc}, the map $\mu$ can be étale locally described using the good moduli space map for the moduli stack of dimension $(d_i)_{i=1}^k$ representations of $Q_y$. Note that $Q_y$ is a symmetric quiver. Further, Lemma \ref{lem:equiver} holds in this case. 

There are categories
\[\mathbb{B}(r,\chi)_w\subset D^b(\mathcal{B}un^{\mathrm{ss}}(r,\chi))_w\]
as in Definition \ref{def:qbps}, see also \cite{SVdB, P3}. These categories are twisted non-commutative resolutions of singularities of $\mathrm{Bun}^{\mathrm{ss}}(r,\chi)$.
If $(r,\chi)$ are coprime, then $\mu$ is a $\mathbb{C}^{\ast}$-gerbe and there exists a universal vector bundle, so in this case \[\mathbb{B}(r,\chi)_w\simeq D^b(\mathcal{B}un^{\mathrm{ss}}(r,\chi))_w\simeq D^b(\mathrm{Bun}^{\mathrm{ss}}(r,\chi))\] for all $w\in \mathbb{Z}$.

Consider a partition $(r,\chi)=(r_1,\chi_1)+\cdots+(r_k,\chi_k)$ such that $\chi_i/r_i=\chi/r$ for all $1\leq i\leq k$.
There is also a categorical Hall product as in \eqref{cathall} defined via the stack of flags of semistable vector bundles:
\[*\colon \boxtimes_{i=1}^k D^b(\mathcal{B}un^{\mathrm{ss}}(r_i,\chi_i))\to D^b(\mathcal{B}un^{\mathrm{ss}}(r,\chi)).\] 
An argument completely analogous to that of Theorem \ref{thm:sod} shows the following.

\begin{thm}\label{thm:sodBun}
Let $C$ be a smooth projective curve of genus $g\geq 1$,
let $(r_0,\chi_0)\in \mathbb{Z}_{\geq 1}\times\mathbb{Z}$ be coprime, let $d\in \mathbb{Z}_{\geq 1}$, let $(r,\chi):=d(r_0,\chi_0)$, and let $w\in \mathbb{Z}$.
There is a semiorthogonal decomposition 
\[ D^b(\mathcal{B}un^{\mathrm{ss}}(r, \chi))_w
    =\left\langle \boxtimes_{i=1}^k \mathbb{B}(r_i, \chi_i)_{w_i} \,\bigg|\, 
    \frac{v_1}{d_1}<\cdots<\frac{v_k}{d_k}\right\rangle. 
\]
The right hand side is after all partitions 
$(r, \chi, w)=(r_1, \chi_1, w_1)+\cdots+(r_k, \chi_k, w_k)$ with 
$\chi/r=\chi_i/r_i$, and where
$v_i \in \frac{1}{2}\mathbb{Z}$ are defined by 
\[ v_i:=w_i-\frac{g-1}{2}r_i \left(\sum_{i>j}r_j-\sum_{i<j}r_j   \right). \]
The fully-faithful functor 
\begin{align*}
\boxtimes_{i=1}^k \mathbb{B}(d_i r_0, d_i \chi_0)_{w_i}
\to D^b(\mathcal{B}un^{\mathrm{ss}}(r, \chi))_w
\end{align*}
is given by the restriction of the categorical Hall product. The order of the summands is as in \cite{PTK3, PT0}.
\end{thm}
\begin{example}
If $C=\mathbb{P}^1$, then $\mathcal{B}un^{\mathrm{ss}}(r,\chi)\cong BGL(r)$. A version of the above semiorthogonal decomposition holds such that the categories $\mathbb{B}(r,\chi)_w$ are trivial unless $r=1$ when $\mathbb{B}(1,\chi)_w\cong D^b(\Spec \mathbb{C})$. 
\end{example}
\begin{example}
If $C$ has genus $1$, then $\mathbb{B}(r,\chi)_w$ is empty unless $d|w$, and in this case
\[\mathbb{B}(r,\chi)_w\simeq \mathbb{B}(r,\chi)_0\simeq D^b(\mathrm{Sym}^d(E)),\] and the argument is as in \cite[Section 3]{Toquot2}.
\end{example}
\begin{remark}
The Betti numbers of the moduli spaces $\mathrm{Bun}^{\mathrm{ss}}(r,\chi)$ and $\mathrm{Bun}^{\mathrm{ss}}(r,\chi')$ for $\chi, \chi'$ coprime with $r$ are different unless $r|\chi\pm \chi'$, see \cite[Theorem 3.3.2]{HaNa}.
Thus a straightforward version of $\chi$-independence is false, and we thus do not expect the immediate analogue of Conjecture \ref{conj:intro} for the categories $\mathbb{B}(r,\chi)_w$ to hold.
Further, whereas BPS categories for Higgs bundles are indecomposable (by Lemma \ref{lem:sodT}), the categories $\mathbb{B}(r,\chi)_w$ are expected to have further semiorthogonal decompositions, see the introduction of \cite{Tevelev} for further details for the case $(r,\chi)=(2,1)$.
\end{remark}

\subsection{Sheaves of non-commutative algebras}\label{subsec:sheaves}
We define the following BPS condition for the 
numerical data $(r, \chi, w)$. 
Recall the formula (\ref{formula:gD}) for the genus $g^{\rm{sp}}$ of the spectral curve. 
\begin{defn}\label{def:bps}
We say that $(r, \chi, w)$ satisfies the \textit{BPS condition}
if the following vector is primitive (i.e. its entries are coprime):
\begin{align*}
(r, \chi, w+1-g^{\rm{sp}}) \in \mathbb{Z}^3. 
\end{align*}
 \end{defn}

Recall that we write $(r, \chi)=d(r_0, \chi_0)$ as in (\ref{rchid}) for $(r_0,\chi_0)\in \mathbb{Z}_{\geq 1}\times\mathbb{Z}$ coprime and $d\in\mathbb{Z}_{\geq 1}$. 
The following lemma is elementary and the proof is left to the reader. 
\begin{lemma}\label{lem:bps}
    The tuple $(r, \chi, w)$ satisfies the BPS condition 
    if and only if it satisfies either of the following conditions
     \begin{enumerate}
     \item 
 $lr_0^2$ is even and $(d, w)$ is coprime;
 \item $lr_0^2$ is odd and $(d, w)$ is coprime with $d\not\equiv 2 \pmod 4$;
 \item $lr_0^2$ is odd, $\gcd(d, w)=2$, and $d \equiv 2 \pmod 4$. 
 \end{enumerate}
\end{lemma}

\begin{remark}\label{rmk:bps}
If $l$ is even, then $(r, \chi, w)$ satisfies the BPS condition 
if and only if it is primitive in $\mathbb{Z}^3$. 

If $l$ is odd, then $(r, \chi, w)$ satisfies the BPS condition 
if and only if either one of the following conditions
holds:
\begin{enumerate}
    \item $r_0$ is even and $(r, \chi, w)$ is primitive;
    \item $r_0$ is odd, $(r, \chi, w)$ is primitive and $d \not\equiv 2 \pmod 4$;
    \item $r_0$ is odd, $\gcd(r, \chi, w)=2$, and $d \equiv 2 \pmod 4$. 
\end{enumerate}
\end{remark}

\begin{defn}\label{def:bpscat}
Suppose that $(r, \chi, w)$ satisfies the BPS condition. Then the category $\mathbb{T}^L(r, \chi)_w$ is called a \textit{BPS category}. 
\end{defn}

\begin{remark} 
We explain the terms \textit{BPS condition} and \textit{BPS category} introduced above. For simplicity, assume that $l>2g-2$, but the argument for $L=\Omega_C$ is analogous.

    For the quiver $Q$ with one vertex and $(lr_0^2+1)$-loops (i.e. the quiver for the deepest stratum of $M^L(r,\chi)$), 
    consider its quasi-BPS category $\mathbb{T}_{Q_y}(d)_w\subset D^b(\X_{Q_y}(d))_w$
    defined as in (\ref{qbps:loc}).
    By Lemma~\ref{lem:bps}, the total topological K-theory of $\mathbb{T}_{Q_y}(d)_w$ is isomorphic to the total BPS cohomology of the good moduli space $X_{Q_y}(d)$ of $\X_{Q_y}(d)$, which is a local version of the BPS cohomology of $M^L(r,\chi)$, see~\cite{PTtop}.
    Note that the topological K-theory of a quasi-BPS category contains the BPS cohomology as a direct summand, but it may be larger.
    The BPS cohomologies of both $X_{Q_y}(d)$ and $M^L(r,\chi)$ are the intersection cohomologies of these two spaces~\cite{Mein, Mere}. In \cite{PThiggs2}, we further expand on this remark and show that the total topological K-theory of a BPS category is isomorphic to the BPS cohomology of $M^L(r,\chi)$ (for both cases $l>2g-2$ and $L=\Omega_C$).
    
\end{remark}

We will show in this section (Proposition \ref{prop:TA}, Theorem \ref{thm:proper}, Corollary \ref{cor:proper}) that BPS categories behave like derived categories of 
stable Hitchin moduli spaces with coprime rank and Euler characteristics, so 
may be regarded as ``non-commutative Hitchin moduli spaces''.

\begin{prop}\label{prop:TA}
Suppose that $l>2g-2$ 
and $(r, \chi, w)$ satisfies the BPS condition. 
Then there is a vector bundle 
$\mathbb{V} \to \mathcal{M}^L(r, \chi)$
such that $\mathbb{V} \in \mathbb{T}^L(r, \chi)_w$, 
and  
the sheaf of non-commutative algebras
\begin{align}\label{def:A}
    \mathscr{A}:=\pi_{\ast}\mathcal{E}nd(\mathbb{V})
\end{align}
on $M^L(r, \chi)$ is a maximal Cohen-Macaulay $\mathcal{O}_{M^L(r, \chi)}$-module 
    and is of finite global dimension. Moreover, there is an equivalence 
    \begin{align*}
    \pi_{\ast}\mathcal{H}om(\mathbb{V}, -) \colon 
        \mathbb{T}^L(r, \chi)_w \stackrel{\sim}{\to} D^b(\mathscr{A}):=D^b(\Coh(\mathscr{A})). 
    \end{align*}
\end{prop}
\begin{proof}
 We first consider the local case. 
 Let $y \in M^L(r, \chi)$ lies in the deepest stratum, 
 so it corresponds to the direct sum $V \otimes E_0$ with $\dim V=d$ and a stable $L$-twisted Higgs bundle $E_0$. 
 Let $Q_y$ be the Ext-quiver at $y$, which is a $(lr_0^2+1)$-loop quiver. 
 Let $\mathbb{T}_{Q_y}(d)_w$ be the subcategory of $D^b(\mathcal{X}_{Q_y}(d))$ defined as in (\ref{qbps:loc}). 
 It is generated by $\mathbb{S}^{\gamma}(V) \otimes \mathcal{O}_{\mathcal{X}_{Q_x}(d)}$
 for Schur powers $\mathbb{S}^{\gamma}(V)$
 of $V$
 whose highest weight $\gamma$ satisfy the condition (\ref{weight:W})
 for the Ext-quiver $Q_y$. 
 Let $\Lambda$ be the set of $\gamma \in M(d)$ which appears as a generator of 
 $\mathbb{T}_{Q_y}(d)_w$ as above. 
 Then the direct sum
 \begin{align*}
     \mathbb{V}_y :=\bigoplus_{\gamma \in \Lambda}\mathbb{S}^{\gamma}(V)\otimes \mathcal{O}_{\mathcal{X}_{Q_y}(d)} \in 
     \mathbb{T}_{Q_y}(d)_w
 \end{align*}
 is a tilting generator of $\mathbb{T}_{Q_y}(d)_w$, i.e. 
 $\Hom^i(\mathbb{V}_y, \mathbb{V}_y)=0$ for $i\neq 0$ and 
 $\Hom^{\ast}(\mathbb{V}_y, P)=0$ for $P \in \mathbb{T}_{Q_y}(d)_w$ implies 
 $P=0$. 
 
     The BPS condition of $(r, \chi, w)$ 
     implies that the highest weight $\gamma$ in Remark~\ref{rmk:domi} lies in 
     the strict interior of the right hand side in (\ref{weight:W}), 
     see~\cite[Lemma~5.11]{PTquiver} in the case that $lr_0^2$ is even. 
     In the case that $lr_0^2$ is odd, 
     this claim follows from~\cite[Lemma~8.7]{PTtop} together with the argument in~\cite[Lemma~5.11]{PTquiver}. 
     Then by~\cite[Proposition~4.4]{SVdB}, 
     under the BPS condition, the 
     non-commutative algebra $\mathscr{A}_y:=\mathrm{End}(\mathbb{V}_y)$ is a maximal Cohen-Macaulay
     module over $R_{Q_y}(d)\ssslash G(d)$. 
     Moreover $\mathscr{A}_y$ is homologically homogeneous
by~\cite[Theorem~1.5.1]{SVdB}.      

 Therefore
there is an equivalence, see the argument of~\cite[Lemma~3.4]{TU} and also Lemma~\ref{lem:AT}
 \begin{align}\label{equiv:loc}
     \Hom(\mathbb{V}_y, -) \colon 
     \mathbb{T}_{Q_y}(d)_w \stackrel{\sim}{\to} D^b(\mathscr{A}_y). 
 \end{align} 
 The inverse functor of (\ref{equiv:loc}) is given by 
 \begin{align*}
    D^b(\mathscr{A}_y) \to \mathbb{T}_{Q_y}(d)_w, \ 
    (-) \mapsto (-) \otimes^L_{\mathscr{A}_y} \mathbb{V}_y. 
 \end{align*}
 Note that the above functor makes sense as $\mathscr{A}_y$ is homologically homogeneous.

     We globalize the above local equivalence (\ref{equiv:loc}). 
     We take $m\gg 0$, a closed point $p\in C$, and consider the following vector bundle on $\mathcal{M}^L(r, \chi)$
     \begin{align}\label{vec:U}
         \mathcal{U}:=p_{\mathcal{M}\ast}(\mathcal{O}_C(mp)\boxtimes \mathcal{F}) \to \mathcal{M}^L(r, \chi), 
     \end{align} where recall that $\mathcal{F}$ is the universal Higgs bundle.
     Note that for a closed point $y \in M^L(r, \chi)$
     as above 
     and its lift to a closed point $y \in \mathcal{M}^L(r, \chi)$, 
     we have that $\mathcal{U}|_{y}$ is a direct sum of $V$ as a $GL(V)$-representation. 
     We then set $\mathbb{V}$ to be the direct 
     sum of Schur powers of $\mathcal{U}$
     \begin{align*}
         \mathbb{V} :=\bigoplus_{\gamma \in \Lambda}\mathbb{S}^{\gamma}(\mathcal{U}) \in \Coh(\mathcal{M}^L(r, \chi)). 
     \end{align*}
     Then \'etale locally at $y$, the set of direct summands of $\mathbb{V}$ is the same as that of $\mathbb{V}_y$. Since this holds for any 
     closed point $y$ as above, we have $\mathbb{V} \in \mathbb{T}^L(r, \chi)_w$. 
     
    Let $\mathscr{A}=\pi_{\ast}\mathcal{E}nd(\mathbb{V})$ be the sheaf of algebras as in (\ref{def:A}). 
    Then \'etale locally at $y$, 
    the algebra $\mathscr{A}$ is Morita equivalent to $\mathscr{A}_y$.  
    Let $\Phi$ be the functor
     \begin{align}
         \Phi(-)=\pi_{\ast}\mathcal{H}om(\mathbb{V}, -) \colon \mathbb{T}^L(r, \chi)_{w} \to D^b(\mathscr{A}).
     \end{align}
     It admits a left adjoint  
     \begin{align*}
         \Phi^L(-)=(-)\otimes_{\mathscr{A}}\mathbb{V} \colon D^b(\mathscr{A}) \to \mathbb{T}^L(r, \chi)_w.  
     \end{align*}
     We have the adjunction morphisms 
     \begin{align*}\Phi^L \circ \Phi \to \id, \ \id \to \Phi \circ \Phi^L.
     \end{align*}
     Note that both of $\Phi$ and $\Phi^L$ are linear 
     over $\mathrm{Perf}(M^L(r, \chi))$. 
     By the above local argument together with Lemma~\ref{lem:equiver}, 
     we see that the above adjunction morphisms are isomorphisms 
     \'etale locally at any point in $M^L(r, \chi)$, hence they are isomorphisms. 
     Therefore $\Phi$ is an equivalence. 
\end{proof}

 Proposition~\ref{prop:TA} implies that 
the BPS category is a (twisted) non-commutative crepant 
resolution of the singular moduli space $M^L(r, \chi)$. 
Note that for all $d\geq 2$ and $1+lr_0^2\geq 2$ except the case $(d, l, r_0)=(2,1,1)$, the moduli space $M^L(r, \chi)$ does not have a geometric crepant resolution \cite{KaLeSo}. 

We now briefly recall some terminology related to dg-categories 
(cf.~\cite[Section~3.3]{Orsmooth}). 

A dg-category $\mathscr{D}$
over $\mathbb{C}$ is called \textit{smooth} if the diagonal dg-module 
on $\mathscr{D}\otimes_{\mathbb{C}}\mathscr{D}^{\rm{op}}$ is perfect. 

A dg-category $\mathscr{D}$ over a
$\mathbb{C}$-scheme $B$ (i.e. a dg-category
over $\mathbb{C}$ with a module over $\mathrm{Perf}(B)$) is \textit{proper over $B$} if 
the inner homomorphism $\mathcal{H}om(\mathcal{E}_1, \mathcal{E}_2) \in D_{\rm{qc}}(B)$ is an object in 
$D^b(B)$. 

The \textit{relative Serre functor} $S_{\mathscr{D}/B}$ is 
a functor $S_{\mathscr{D}/B} \colon \mathscr{D} \to \mathscr{D}$
such that there exist functorial isomorphisms
\begin{align*}
    \mathcal{H}om(\mathcal{E}_1, S_{\mathscr{D}/B}(\mathcal{E}_2)) \cong \mathcal{H}om_B(\mathcal{H}om(\mathcal{E}_2, \mathcal{E}_1), \mathcal{O}_B). 
\end{align*}
If the relative Serre functor is isomorphic to a shift functor, we say that $\mathscr{D}$ is \textit{Calabi-Yau} over $B$.

\begin{thm}\label{thm:proper0}
Suppose that $l>2g-2$ and assume that $(r, \chi, w)$ satisfies 
the BPS condition.  
Then the BPS category $\mathbb{T}^L(r, \chi)_w$
is smooth, proper over $M=M^L(r, \chi)$, and there is a relative 
Serre functor $S_{\mathbb{T}/M}$ over $M$ satisfying $S_{\mathbb{T}/M} \cong \id$, i.e. 
for $E_i \in \mathbb{T}^L(r, \chi)_w$ with $i=1, 2$, we have
\begin{align}\label{isom:STM}
    \pi_{\ast}R\mathcal{H}om(E_1, E_2) \cong 
    \mathcal{R}\mathcal{H}om_{M}(\pi_{\ast}R\mathcal{H}om(E_2, E_1), \mathcal{O}_M). 
\end{align}
\end{thm}
\begin{proof}
The smooth and properness follows from 
Proposition~\ref{prop:TA} as $\mathscr{A}$ has finite global dimension and is
finitely generated as an $\mathcal{O}_{M}$-module. To show $S_{\mathbb{T}/M} \cong \id$, see Lemma~\ref{assum:perf} for a more 
general statement. 
\end{proof}

We also have the following corollary, showing that the BPS category is a 
smooth Calabi-Yau compactification of 
the abelian fibration $M^L(r, \chi)^{\rm{sm}} \to B^{\rm{sm}}$ of 
relative dimension $g^{\rm{sp}}$ over the locus of 
smooth spectral curves. 
\begin{cor}\label{cor:proper0}
    In the situation of Theorem~\ref{thm:proper0}, 
    the relative Serre functor $S_{\mathbb{T}/B}$ of $\mathbb{T}^L(r, \chi)_w$ 
    over $B$ is 
    isomorphic to the shift functor $[g^{\rm{sp}}]$, i.e. 
    for $E_i \in \mathbb{T}^L(r, \chi)_w$ with $i=1, 2$, we have
\begin{align*}
    Rh_{\ast}R\mathcal{H}om(E_1, E_2) \cong 
    \mathcal{R}\mathcal{H}om_{B}(Rh_{\ast}R\mathcal{H}om(E_2, E_1), \mathcal{O}_B[g^{\rm{sp}}]). 
\end{align*}
\end{cor}
\begin{proof}
The corollary follows by applying the Grothendieck duality for \[h_M \colon M^L(r, \chi) \to B\]
and noticing that the relative dualizing sheaf of $h_M$ is trivial by Lemma~\ref{lem:gorenstein}. 
\end{proof}

\subsection{BPS categories via non-commutative matrix factorizations}
By Proposition \ref{prop:TA}, the BPS category for $l>2g-2$ is 
equivalent to the category of coherent sheaves over 
a twisted non-commutative crepant resolution of $M^L(r, \chi)$. In the case of $L=\Omega_C$, 
we relate the BPS category with a category of matrix factorizations 
over a sheaf of non-commutative algebras. 
This result will not be used later, but we believe it may be of 
independent interest.

Recall the notation from Subsection \ref{subsec:qbps:nc}.
Let $f$ be the function 
\begin{align*}
f \colon \mathcal{V}^{\vee} \to \mathbb{C}, (y, v) \mapsto \langle s(y), v \rangle
\end{align*}
where $y \in \mathcal{M}^{\Omega_C(p)}(r, \chi)$ and 
$v \in \mathcal{V}^{\vee}|_{y}$. 
Note that $\mathcal{V}^{\vee} \cong \mathcal{V}$ and 
$f$ is identified with the function 
\begin{align*}
    f \colon \mathcal{V} \to \mathbb{C}, \ f(y, u)=\mathrm{tr}(u \circ s(y))
\end{align*}
where $y \in \mathcal{M}^{\Omega_C(p)}(r, \chi)$ and 
$u \in \mathcal{V}|_{y}$. 
There is an isomorphism 
\begin{align*}
    \Omega_{\mathcal{M}(r, \chi)}[-1]^{{\rm{cl}}} \stackrel{\cong}{\to}
    \mathrm{Crit}(f). 
\end{align*}
Here, the left hand side is the $(-1)$-shifted cotangent
\begin{align*}
    \Omega_{\mathcal{M}(r, \chi)}[-1]:=
    \Spec \mathrm{Sym}(\mathbb{T}_{\mathcal{M}(r, \chi)}[1])
\end{align*}
which is identified with the moduli stack 
of semistable sheaves on the non-compact 
CY 3-fold
\begin{align*}X=\mathrm{Tot}_C(\Omega_{C} \oplus \mathcal{O}_C),
\end{align*}
see~\cite[Lemma~3.4.1]{T}.

There is a Koszul equivalence~\cite{I, Hirano, T}: 
\begin{align}\label{equiv:koszul}
    D^b(\mathcal{M}(r, \chi)) \stackrel{\sim}{\to} \mathrm{MF}^{\rm{gr}}(\mathcal{V}, f)
\end{align}
where $\mathbb{C}^{\ast}$ acts on fibers of $\mathcal{V} \to \mathcal{M}^{\Omega_C(p)}(r, \chi)$ 
by weight one, and the right hand side is the category of graded 
matrix factorizations of the function $f$. 
Its objects consist of tuples
\begin{align}\label{tuplet:graded}
(E, F, \alpha \colon 
E\to F(1), \beta \colon F\to E),
\end{align}
where $E$ and $F$ are $\mathbb{C}^{\star}$-equivariant coherent sheaves on $\mathcal{V}$, 
$(1)$ is the twist by the weight one $\mathbb{C}^{\ast}$-character, 
and $\alpha$ and $\beta$ are $\mathbb{C}^{\ast}$-equivariant morphisms
such that $\alpha\circ\beta$ and $\beta\circ\alpha$ are multiplication by $f$.
We refer to~\cite[Section~2.6]{PT0} for a review of 
graded matrix factorizations. 

There is a commutative diagram 
\begin{align*}
\xymatrix{
\mathcal{V} \ar[r]_-{\pi_{N}}  \ar@/^18pt/[rr]^f  \ar[d]_-{g} & N \ar[r]_-{f_N} \ar[d] & \mathbb{C}. \\
\mathcal{M}^{\Omega_C(p)}(r, \chi) \ar[r]^-{\pi} & M^{\Omega_C(p)}(r, \chi) & 
}
\end{align*}
Here $\pi_N \colon \mathcal{V} \to N$ is the good moduli space of $\mathcal{V}$, that is 
\begin{align*}
    N:=\Spec_{M^{\Omega_C(p)}(r, \chi)} \pi_{\ast}\mathrm{Sym}(g_{\ast}\mathcal{V}^{\vee}).  
\end{align*}
Let $\mathbb{C}^{\ast}$ acts on fibers of $\mathcal{V}$ by weight one, 
and consider the induced action on $N$. 
For a $\mathbb{C}^{\ast}$-equivariant coherent sheaf of $\mathcal{O}_N$-algebras $\mathscr{A}$
on $N$, 
we consider the category 
\begin{align*}
    \mathrm{MF}^{\rm{gr}}(\mathscr{A}, f_N)
\end{align*}
which consists of tuples (\ref{tuplet:graded}) such that $E, F$ are 
$\mathbb{C}^{\ast}$-equivariant right coherent $\mathcal{A}$-modules, 
and $\alpha \circ \beta$, $\beta \circ \alpha$ are multiplications by $f_N$. 
See~\cite[Definition~4.6]{ncHirano} for the precise definition of non-commutative factorization categories. 

\begin{prop}\label{prop:MFA}
Let $L=\Omega_C$ and $(r, \chi, w)$ satisfies the BPS condition (i.e. it is primitive). 
Then there is a sheaf of non-commutative algebras $\mathscr{A}_N$ on $N$, which is 
a Cohen-Macaulay $\mathcal{O}_N$-module of finite global dimension, such that 
there is an equivalence 
\begin{align}\label{equiv:TMFA}
    \mathbb{T}(r, \chi)_w \stackrel{\sim}{\to} \mathrm{MF}^{\rm{gr}}(\mathscr{A}_N, f_N). 
\end{align}
\end{prop}
\begin{proof}
Let $\mathbb{T}' \subset D^b_{\mathbb{C}^{\ast}}(\mathcal{V})$ be the 
subcategory of objects $\mathcal{P}$ such that for all morphisms $\nu \colon B\mathbb{C}^{\ast} \to \mathcal{V}$
the set of weights $\mathrm{wt}(\nu^{\ast}\mathcal{P})$ lies in the right hand side of (\ref{cond:nu}).
Then the Koszul equivalence (\ref{equiv:koszul}) restricts to the equivalence, see~\cite[Lemma~2.6, Corollary~3.15]{PTquiver}
\begin{align}\label{equiv:koszul2}
    \mathbb{T}(r, \chi)_w \stackrel{\sim}{\to} \mathrm{MF}^{\rm{gr}}(\mathbb{T}', f_N).
\end{align}
Here the right hand side is the category of matrix factorizations 
whose factors are objects in $\mathbb{T}'$. 

Let $y \in M^{\Omega_C(p)}(r, \chi)$ lies in the 
deepest stratum, corresponding to $V\otimes E_0$ where 
$\dim V=d$ and $E_0$ is a stable $\Omega_C(p)$-Higgs bundle 
with $(\rank(E_0), \chi(E_0))=(r_0, \chi_0)$. 
Note that we have $\mathcal{V}|_{y}=\mathrm{End}(V)^{\oplus r_0^2}$. 
Therefore étale locally at $y$, the diagram 
\begin{align*}
    \mathcal{V} \to \mathcal{M}^{\Omega_{C}(p)}(r, \chi) \to M^{\Omega_{C}(p)}(r, \chi)
\end{align*}
is isomorphic to 
\begin{align*}
    \mathfrak{gl}(V)^{\oplus (e+r_0^2)}/GL(V) \to 
    \mathfrak{gl}(V)^{\oplus e}/GL(V) \to 
     \mathfrak{gl}(V)^{\oplus e} \ssslash GL(V). 
\end{align*}
Here we have set $e=1+r_0^2(2g-1)$, so $e+r_0^2=1+2r_0^2 g$. 
Let $\Lambda$ be the set of dominant 
weights $\gamma \in M(d)$
satisfying the condition (\ref{weight:W}) for 
weights $\beta$ of $\mathfrak{gl}(V)^{\oplus (1+2r_0^2g)}$. 
Let $\mathcal{U}\to \mathcal{M}^{\Omega_C(p)}(r, \chi)$ be the vector bundle as in (\ref{vec:U}). 
We use the same symbol $\mathcal{U}$ for its pull-back to 
$\mathcal{V} \to \mathcal{M}^{\Omega_C(p)}(r, \chi)$. 
We set
\begin{align*}
    \mathbb{V}_N :=\bigoplus_{\gamma \in \Lambda}
    \mathbb{S}^{\gamma}(\mathcal{U})
\end{align*}
which is a $\mathbb{C}^{\ast}$-equivariant vector bundle on $\mathcal{V}$. 
Let $\mathscr{A}_N=\pi_{N\ast}\mathcal{E}nd(\mathbb{V}_N)$, which is a 
$\mathbb{C}^{\ast}$-equivariant 
sheaf of non-commutative algebras on $N$. 
Similarly to the proof of Proposition~\ref{prop:TA},
we have the equivalence
\begin{align*}
\pi_{N\ast}\mathcal{H}om(\mathbb{V}_N, -) \colon 
    \mathbb{T}' \stackrel{\sim}{\to} D^b_{\mathbb{C}^{\ast}}(\mathscr{A}_N),
\end{align*}
and $\mathscr{A}_N$ is a maximal Cohen–Macaulay $\mathcal{O}_N$-module of finite global dimension. 
By applying the matrix factorizatin and using (\ref{equiv:koszul2}), we obtain 
the equivalence (\ref{equiv:TMFA}). 
\end{proof}

\subsection{The reduced quasi-BPS category}
As we observed in Remark~\ref{rmk:coprime2}, 
the derived stack $\mathcal{M}(r, \chi)$ contains 
some redundant derived structure, and it is 
useful to eliminate it. 
We first note the following lemma. 

\begin{lemma}
    The section $s$ in (\ref{sec:s}) factors through 
    the section $s_0$ of the subbundle 
    \begin{align*}
        \mathcal{V}_0 := \mathcal{E}nd(\mathcal{F}_p)_0 \subset \mathcal{V}
    \end{align*}
    where the subscript $0$ means the traceless part, i.e. the kernel of 
    \begin{align}\label{tracemap}
        \mathrm{tr} \colon \mathcal{E}nd(\mathcal{F}_p) \to \mathcal{O}_{\mathcal{M}^{\Omega_C(p)}(r, \chi)}. 
    \end{align}
\end{lemma}
\begin{proof}
It is enough to show that
for any $\Omega_C(p)$-Higgs bundle $(F, \theta)$, 
we have 
$\mathrm{tr}(\theta|_{p})=\mathrm{tr}(\theta)|_{p}=0$. 
Note that $\mathrm{tr}(\theta)$ gives a global 
section $\mathrm{tr}(\theta) \in H^0(\Omega_C(p))$. 
We have the exact sequence 
\begin{align}\label{seq:Omega}
    0 \to H^0(\Omega_C) \to H^0(\Omega_C(p)) \to \Omega_C(p)|_{p}. 
\end{align}
We have $h^0(\Omega_C)=h^0(\Omega_C(p))=g$, 
where the latter follows from 
the Riemann-Roch theorem $\chi(\Omega_C(p))=g$
and the Serre duality $h^1(\Omega_C(p))=h^0(\mathcal{O}_C(-p))=0$. 
Therefore the map $H^0(\Omega_C) \to H^0(\Omega_C(p))$ is 
an isomorphism, hence $\mathrm{tr}(\theta)|_{p}=0$. 

\end{proof}

We define the reduced stack $\mathcal{M}(r, \chi)^{\rm{red}}$ to be the derived 
zero locus of $s_0$
\begin{align*}
    \mathcal{M}(r, \chi)^{\rm{red}} \stackrel{\sim}{\to} s_0^{-1}(0). 
\end{align*}
Similarly to~\cite[Lemma~4.4]{PTK3}, 
the reduced stack is classical when $g\geq 2$. 
Let 
\begin{align*}
    \mathbb{T}(r, \chi)_w^{\rm{red}} \subset D^b(\mathcal{M}(r, \chi)^{\rm{red}})_w
\end{align*}
be the reduced quasi-BPS category, which is defined 
similarly to Definition~\ref{def:qbps}
using the following 
closed embedding instead of (\ref{emb:M})
\begin{align*}
    j\colon \mathcal{M}(r, \chi)^{\rm{red}} \hookrightarrow 
    \mathcal{M}^{\Omega_C(p)}(r, \chi). 
\end{align*}

\begin{remark}
    If $(r,\chi,w)$ is a primitive vector (i.e. satisfies the BPS condition), the category $\mathbb{T}(r, \chi)_w^{\rm{red}}$ is called \textit{a reduced BPS category}.
\end{remark}

\begin{remark}
    There is a version of Theorem \ref{thm:sod} for the category $D^b(\mathcal{M}(r, \chi)^{\rm{red}})_w$ completely analogous to \cite[Theorem 5.2]{PTK3}. 
\end{remark}
\begin{remark}
    The construction of $\mathbb{T}(r, \chi)_w^{\rm{red}}$ is Zariski (and étale) local over $M(r,\chi)$ and $B$. As in Remark \ref{remark316} and for $\alpha$ the Brauer class defined there, there is an equivalence \begin{align*}
    \mathbb{T}(r, \chi)_w^{\rm{red}}|_{M(r, \chi)^{\rm{st}}} \simeq 
 D^b(M(r, \chi)^{\rm{st}}, \alpha^w).
\end{align*} 
\end{remark}

Since the trace map (\ref{tracemap}) canonically splits, 
there are decompositions
\begin{align*}
    \mathcal{V}=\mathcal{V}_0 \oplus \mathcal{O}, \ 
    N=N_0 \times \mathbb{A}^1
\end{align*}
where $\mathcal{V}_0 \to N_0$ is the good 
moduli space, 
and the function $f_N$ factors through 
the projection 
\begin{align*}
    f_N \colon N \to N_0 \stackrel{f_{N_0}}{\to} \mathbb{C}
\end{align*}
where the first morphism is the projection. 
\begin{prop}\label{prop:MFA2}
Let $L=\Omega_C$ and $(r, \chi, w) \in \mathbb{Z}^3$ is primitive. 
Then there is a sheaf of non-commutative algebras $\mathscr{A}_{N_0}$ on $N_0$, which is 
a Cohen-Macaulay $\mathcal{O}_{N_0}$-module and of finite global dimension, such that 
there is an equivalence 
\begin{align*}
    \mathbb{T}(r, \chi)_w^{\rm{red}} \stackrel{\sim}{\to} \mathrm{MF}^{\rm{gr}}(\mathscr{A}_{N_0}, f_{N_0}). 
\end{align*}
\end{prop}
\begin{proof}
The same argument used to prove Proposition~\ref{prop:MFA2} also applies here. 
\end{proof}

The following is a version of Corollary~\ref{thm:proper0} for the reduced BPS category, 
and it is the Higgs bundle analogue of~\cite[Theorem~1.2]{PTK3}. 

\begin{thm}\label{thm:proper}
Let $L=\Omega_C$, $g \geq 2$, and let $(r, \chi, w)$ be a primitive vector. 
The reduced BPS category $\mathbb{T}(r, \chi)^{\rm{red}}_w$
is smooth, proper over $M(r, \chi)$, and there is a relative 
Serre functor $S_{\mathbb{T}/M}$ satisfying $S_{\mathbb{T}/M} \cong \id$. 
\end{thm}
\begin{proof}
The smoothness follows from that the algebra $\mathscr{A}_{N_0}$ in Proposition~\ref{prop:MFA} is 
homologically homogeneous. The properness follows by the 
same argument of~\cite[Theorem~6.7]{PTK3}, which 
itself follows from the local result in~\cite[Proposition~5.9]{PTquiver}. 

The triviality of the relative Serre functor also follows from the same 
argument of~\cite[Theorem~7.4]{PTK3}. 
Note that in loc. cit. the relative Serre functor is shown to be 
trivial only étale locally on the good moduli space. 
This is because in the K3 surface case, we only have étale local 
description of the derived moduli stack as derived zero locus. 
In our situation, we can write $\mathcal{M}^L(r, \chi)$ as a global 
derived zero locus as in (\ref{equiv:stack}). We can then use the above 
global chart and apply the argument of~\cite[Theorem~7.4]{PTK3} to deduce that the relative Serre functor is indeed trivial. 
\end{proof}

\begin{cor}\label{cor:proper}
    In the situation of Theorem~\ref{thm:proper}, 
    the relative Serre functor $S_{\mathbb{T}/B}$ of $\mathbb{T}(r, \chi)_w^{\rm{red}}$ 
    over $B$ is 
    isomorphic to $(-) \mapsto (-)[g^{\rm{sp}}]$. 
\end{cor}
\begin{proof}
The proof is similar to Corollary~\ref{cor:proper0}, using the Grothendieck 
duality for $h_M \colon M(r, \chi) \to B$ and Theorem~\ref{thm:proper}. 
\end{proof}

\begin{remark}\label{rmk:coprime3}
If $(r, \chi)$ is coprime, then by Remark~\ref{rmk:coprime2},
taking the reduced stack corresponds to removing the 
factor of $\Spec \mathbb{C}[\epsilon]$ on $M(r,\chi)\times \Spec \mathbb{C}[\epsilon]$. Therefore 
there is an equivalence
\begin{align*}
    \mathbb{T}(r, \chi)_w^{\rm{red}} \simeq D^b(M(r, \chi)),
\end{align*}
i.e. the reduced BPS category is the derived category of usual Higgs moduli space, which is 
a smooth and Calabi-Yau fibration over $B$. 
However, if $(r, \chi)$ is not coprime but $(r, \chi, w)$ is primitive, 
e.g. $(r, 0, 1)$ for $r>1$, then the category 
$\mathbb{T}(r, \chi)_w^{\rm{red}}$ provides a new smooth dg-category 
which is proper and Calabi-Yau over $B$. 
\end{remark}

\section{The conjectural symmetry of quasi-BPS categories: the $\mathrm{GL}$ case}
In this section, we state our main conjectures on 
symmetries (which interchange the degree and the weight) of quasi-BPS categories for $GL(r)$-bundles. 
As mentioned in the introduction, these conjectures can be regarded as versions of the D/K equivalence~\cite{B-O2, MR1949787} or SYZ mirror symmetry~\cite{SYZ} for moduli of Higgs bundles, or as a version of the Dolbeault Langlands equivalence~\cite{DoPa} for categories of bounded complexes of coherent sheaves.

\subsection{Fourier-Mukai transforms for Picard stacks}
We first recall the Fourier-Mukai equivalence for 
moduli stacks of line bundles on smooth projective curves \cite{Mu1, DPA}. 
Let $D$ be a smooth projective curve of genus $g(D)$, 
and let
\begin{align}\label{pic:gerb}
    \mathcal{P}ic^e(D) \to \mathrm{Pic}^e(D)
\end{align}
be the classical moduli stack of 
line bundles on $D$ of degree $e$
and its good moduli space. 
Note that $\mathrm{Pic}^e(D)$ is an abelian variety of dimension $g(D)$, 
and the morphism (\ref{pic:gerb}) is a trivial $\mathbb{C}^{\ast}$-gerbe. 
Let 
\begin{align*}
\mathcal{L} \to D \times \mathcal{P}ic^e(D) 
\end{align*}
be the universal line bundle. 
Following~\cite{Ardual}, we define the following line bundle 
on $\mathcal{P}ic^{e_1}(D) \times \mathcal{P}ic^{e_2}(D)$:
\begin{align}\label{line:P}
    \mathcal{P}=\det Rp_{13\ast}(p_{12}^{\ast}\mathcal{L}_1 \otimes
    p_{23}^{\ast}\mathcal{L}_2) &\otimes 
 \det Rp_{13\ast}(p_{12}^{\ast}\mathcal{L}_1)^{-1} \\ \notag 
 &\otimes 
  \det Rp_{13\ast}(p_{23}^{\ast}\mathcal{L}_2)^{-1} 
  \otimes \det Rp_{13\ast}\mathcal{O}. 
\end{align}
Here $\mathcal{L}_i$ is the universal line bundle on 
$D \times \mathcal{P}ic^{e_i}(D)$ 
and $p_{ij}$ is the projection from 
$\mathcal{P}ic^{e_1}(D) \times D \times \mathcal{P}ic^{e_2}(D)$
onto the corresponding factors. 
Informally over a pair $(L_1, L_2)$
where $L_i$ is a line bundle on $D$ of degree $e_i$, we have 
\begin{align*}
    \mathcal{P}|_{(L_1, L_2)}=\det R\Gamma(L_1 \otimes L_2)
    \otimes \det R\Gamma(L_1)^{-1} \otimes \det R\Gamma(L_2)^{-1}
    \otimes \det R\Gamma(O_D). 
\end{align*}
It has $(\mathbb{C}^{\ast} \times \mathbb{C}^{\ast})$-weight $(e_2, e_1)$, and
induces a Fourier-Mukai equivalence~\cite{Mu1}:
\begin{align}\label{FM:smooth}
    D^b(\mathcal{P}ic^{e_1}(D))_{-e_2} \stackrel{\sim}{\to} 
    D^b(\mathcal{P}ic^{e_2}(D))_{e_1}. 
\end{align}
\begin{remark}
    By writing $L_2=\mathcal{O}_C(\sum_{i=1}^k a_i x_i)$
    for distinct $x_1, \ldots, x_k \in C$, we can also write, 
    see~\cite[Section~1.1]{Ardual}:
    \begin{align*}
        \mathcal{P}|_{(L_1, L_2)}=\otimes_{i=1}^k L_1^{\otimes a_k}|_{x_k}. 
    \end{align*}
\end{remark}

\subsection{Fourier-Mukai equivalences for families of smooth spectral curves}
The equivalence (\ref{FM:smooth}) naturally extends to the smooth 
family of spectral curves \cite{DPA}. 
For simplicity, we assume that $l>2g-2$, but the same construction also applies to the case $L=\Omega_C$ with $\mathcal{M}^L(r,\chi)$ replaced by $\mathcal{M}(r,\chi)^{\mathrm{red}}$.
Let \[\mathcal{C} \to B\] be
the universal spectral curve
and let
\begin{align*}
    \mathcal{E} \in \Coh(\mathcal{M}^L(r, \chi)\times_B \mathcal{C})
\end{align*}
the universal sheaf. 
We denote by $B^{\rm{sm}} \subset B$
the open subset corresponding to smooth and irreducible spectral 
curves. Let 
$\mathcal{E}^{\rm{sm}}$ be the restriction of $\mathcal{E}$ to $B^{\rm{sm}}$. 
We denote by 
\begin{align*}
    \mathcal{M}^L(r, \chi)^{\rm{sm}}:=
    \mathcal{M}^L(r, \chi)\times_B B^{\rm{sm}}, \ 
    M^L(r, \chi)^{\rm{sm}}:=M^L(r, \chi)\times_B B^{\rm{sm}}. 
\end{align*}
Then there is an isomorphism
\begin{align*}
    \mathcal{M}^L(r, \chi)^{\rm{sm}} \cong \mathcal{P}ic^{e}(\mathcal{C}^{\rm{sm}}/B^{\rm{sm}})
\end{align*}
where $e=\chi+g^{\rm{sp}}-1$ and the right hand side is the relative Picard stack, 
which is a $\mathbb{C}^{\ast}$-gerbe
\begin{align}\label{pic:gerb0}
    \mathcal{P}ic^e(\mathcal{C}^{\rm{sm}}/B^{\rm{sm}}) \to 
    \mathrm{Pic}^e(\mathcal{C}^{\rm{sm}}/B^{\rm{sm}}). 
\end{align}
\begin{remark}
From Remarks~\ref{rmk:open} and~\ref{rmk:open2}, there are equivalences
\begin{align}\label{rest:sm}
    \mathbb{T}^L(r, \chi)_w|_{B^{\rm{sm}}}\simeq 
    D^b(\mathcal{M}^L(r, \chi)^{\rm{sm}})_w \simeq D^b(\mathrm{Pic}^e(\mathcal{C}^{\rm{sm}}/B^{\rm{sm}}), \alpha^w),
\end{align}
where $\alpha$ is the Brauer class classifying the $\mathbb{C}^{\ast}$-gerbe (\ref{pic:gerb0}). 
\end{remark}

Note that the formula (\ref{line:P}) using the universal 
sheaf $\mathcal{E}^{\rm{sm}}$ determines the 
line bundle 
\begin{align}\label{line:P2}
    \mathcal{P}^{\rm{sm}} \to \mathcal{M}^L(r, w+1-g^{\rm{sp}})^{\rm{sm}}\times_{B^{\rm{sm}}} \mathcal{M}^L(r, \chi)^{\rm{sm}}
\end{align}
of bi-weight $(\chi+g^{\rm{sp}}-1, w)$. More precisely, 
the line bundle $\mathcal{P}^{\rm{sm}}$ is given by 
\begin{align}\label{line:P22}
    \mathcal{P}^{\rm{sm}}=\det Rp_{13\ast}(p_{12}^{\ast}\mathcal{E}' \otimes
    p_{23}^{\ast}\mathcal{E}) &\otimes 
 \det Rp_{13\ast}(p_{12}^{\ast}\mathcal{E}')^{-1} \\ \notag 
 &\otimes 
  \det Rp_{13\ast}(p_{23}^{\ast}\mathcal{E})^{-1} 
  \otimes \det Rp_{13\ast}\mathcal{O}
\end{align}
where $p_{ij}$ are the projections 
\begin{align*}
\xymatrix{
& \mathcal{M}'^{\rm{sm}}\times_{B^{\rm{sm}}} \times \mathcal{C}^{\rm{sm}} \times_{B^{\rm{sm}}}
\times \mathcal{M}^{\rm{sm}} \ar[ld]^-{p_{12}} \ar[d]_-{p_{13}} \ar[rd]_-{p_{23}} & \\
\mathcal{M}'^{\rm{sm}}\times_{B^{\rm{sm}}} \times \mathcal{C}^{\rm{sm}} & 
\mathcal{M}'^{\rm{sm}}\times_{B^{\rm{sm}}}
\times \mathcal{M}^{\rm{sm}} & 
\mathcal{C}^{\rm{sm}} \times_{B^{\rm{sm}}}
\times \mathcal{M}^{\rm{sm}}.
}
\end{align*}
Here for simplicity, we have written that
\begin{align}\label{write:simple}
\mathcal{M}=\mathcal{M}^L(r, \chi), \ \mathcal{M}'=\mathcal{M}^L(r, w+1-g^{\rm{sp}}),
\end{align}
and 
$\mathcal{E}'$ is the universal sheaf on 
$\mathcal{C}\times_B \mathcal{M}'$. 
Similarly to (\ref{FM:smooth}), it determines the Fourier-Mukai equivalence 
\begin{align}\label{equiv:family}
 \Phi_{\mathcal{P}^{\rm{sm}}} \colon   D^b(\mathcal{M}^L(r, w+1-g^{\rm{sp}})^{\rm{sm}})_{-\chi+1-g^{\rm{sp}}}\stackrel{\sim}{\to} D^b(\mathcal{M}^L(r, \chi)^{\rm{sm}})_w 
    \end{align}
    by the formula 
    \begin{align*}
        \Phi_{\mathcal{P}^{\rm{sm}}}(-)=Rp_{\mathcal{M}\ast}(p_{\mathcal{M}'}^{\ast}(-) \otimes \mathcal{P}^{\rm{sm}}).
    \end{align*}
    Here $p_{\mathcal{M}}$, $p_{\mathcal{M}'}$ are the projections 
    from $\mathcal{M}'^{\rm{sm}}\times_{B^{\rm{sm}}} \mathcal{M}^{\rm{sm}}$ onto the corresponding factors. 

    \subsection{The main conjectures}
    We now state our main conjectures. 
By (\ref{rest:sm}), the equivalence (\ref{equiv:family}) can be rewritten as 
\begin{align}\label{equiv:Bsm}
    \mathbb{T}^L(r, w+1-g^{\rm{sp}})_{-\chi+1-g^{\rm{sp}}}|_{B^{\rm{sm}}}
\stackrel{\sim}{\to} 
    \mathbb{T}^L(r, \chi)_w|_{B^{\rm{sm}}}. 
\end{align}
By Theorem~\ref{thm:proper0} and Corollary~\ref{cor:proper0},
under the BPS condition both sides (before restriction to $B^{\rm{sm}}$) 
are smooth dg-categories, which are proper 
and Calabi-Yau over $B$. Therefore as in the discussion of Subsection~\ref{subsec:intro:main}, 
we conjecture that the equivalence (\ref{equiv:Bsm}) extends to an equivalence of BPS categories. 
\begin{conj}\label{conj:T0}
Let $l>2g-2$ and suppose that $(r, \chi, w)$ satisfies the BPS condition (see Definition \ref{def:bps}). 
Then there is a $B$-linear equivalence 
\begin{align}\label{equiv:conjT}
\mathbb{T}^L(r, w+1-g^{\rm{sp}})_{-\chi+1-g^{\rm{sp}}}
\stackrel{\sim}{\to} 
    \mathbb{T}^L(r, \chi)_w 
\end{align}
which extends the equivalence (\ref{equiv:Bsm}), i.e. it commutes with 
the restrictions to $B^{\rm{sm}}$. 
\end{conj}

  Using the periodicity in Lemma~\ref{lem:period},
  we can remove the $g^{\rm{sp}}$ terms in the degree and weight as follows: 

\begin{lemma}\label{lem:simple}
An equivalence (\ref{equiv:conjT}) is simplified 
to either 
\begin{align}\label{equiv:simple}
&\mathbb{T}^L(r, w)_{-\chi} \stackrel{\sim}{\to} \mathbb{T}^L(r, \chi)_w  
\mbox{ if }l \mbox{ is even or }
    lr \mbox{ is odd, or} \\
&\notag    \mathbb{T}^L(r, w+r/2)_{-\chi+r/2} \stackrel{\sim}{\to} \mathbb{T}^L(r, \chi)_w  
\mbox{ if } l \mbox{ is odd and }
    r \mbox{ is even}.
\end{align}
    \end{lemma}
\begin{proof}
The simplifications follow from Lemma~\ref{lem:period}. 
\end{proof}
\begin{remark}\label{rmk:replace}
    If $(r, \chi, w)$ satisfies the BPS condition, then 
    $(r, w+1-g^{\rm{sp}}, -\chi+1-g^{\rm{sp}})$ also satisfies the 
    BPS condition. Therefore an equivalence (\ref{equiv:conjT}) is an 
    equivalence between BPS categories. 
\end{remark}

In the case of $L=\Omega_C$, similarly to (\ref{rest:sm}) there is an 
equivalence 
\begin{align*}
    \mathbb{T}(r, \chi)_w^{\rm{red}}|_{B^{\rm{sm}}} \simeq D^b(\mathrm{Pic}^e(\mathcal{C}^{\rm{sm}}/B^{\rm{sm}}), \alpha^w). 
\end{align*}
Here we note that considering the reduced category is necessary, as otherwise 
we need to impose a derived structure on $\mathrm{Pic}^e(\mathcal{C}^{\rm{sm}}/B^{\rm{sm}})$. 
As in (\ref{equiv:Bsm}) (also together with Lemma~\ref{lem:period}),
there is an equivalence 
\begin{align}\label{equiv:Bsm2}
\mathbb{T}(r, w)_{-\chi}^{\rm{red}}|_{B^{\rm{sm}}} \stackrel{\sim}{\to}
        \mathbb{T}(r, \chi)_{w}^{\rm{red}}|_{B^{\rm{sm}}}. 
\end{align}
From Theorem~\ref{thm:proper} and Corollary~\ref{cor:proper},
similarly to Conjecture~\ref{conj:T0}, we also propose the following conjecture: 
\begin{conj}\label{conj:reduced}
    Suppose that $L=\Omega_C$ and $(r, \chi, w)$ is primitive. 
    Then there is a $B$-linear equivalence
    \begin{align*}
        \mathbb{T}(r, w)_{-\chi}^{\rm{red}} \stackrel{\sim}{\to}
        \mathbb{T}(r, \chi)_{w}^{\rm{red}} 
    \end{align*}
    which extends the equivalence (\ref{equiv:Bsm2}). 
\end{conj}

\begin{example}\label{exam:1}
Suppose that $g=0$ and $(r, \chi, w)$ is primitive. In this case, the 
same computation as in~\cite[Proposition~4.17]{PTK3} shows that
\begin{align*}
    \mathbb{T}(r, \chi)_w^{\rm{red}}=\begin{cases}
        D^b(\Spec \mathbb{C}), & r=1, \\
        0, & r>1. 
    \end{cases}
\end{align*}
In particular, Conjecture~\ref{conj:reduced} is true.  
It is also easy to check Conjecture~\ref{conj:T0} for $g=0$ and 
$\deg L=-1, 0$. 
On the other hand, Conjecture~\ref{conj:T0} for $g=0$ and $\deg L>0$ is not obvious 
and more interesting, as
there exist more objects which contribute to $\mathbb{T}^L(r, \chi)_w$
and the singular fibers of the Hitchin fibration are more complicated. 
\end{example}

\begin{example}\label{exam:2}
Suppose that $g=1$ so that $C$ is an elliptic curve, and $(r, \chi, w)$ is primitive. 
In this case, similarly to~\cite[Conjecture~4.18]{PTK3}, 
we expect an equivalence 
\begin{align}\label{equiv:g=1}
    D^b(S) \stackrel{\sim}{\to} \mathbb{T}(r, \chi)_w^{\rm{red}}
\end{align}
where $S=C \times \mathbb{A}^1$. 
The above equivalence should hold from a conjectural computation of the
quasi-BPS categories for the doubled quiver of the Jordan quiver 
studied in~\cite{PT0, PT1}, see~\cite[Conjecture~4.19]{PTK3} and~\cite[Proposition~4.20]{PTK3}. 
Then Conjecture~\ref{conj:reduced} follows from (\ref{equiv:g=1}).    
\end{example}

\section{Equivalences of quasi-BPS categories}
In this section, we prove Theorems~\ref{thm:intro1.5}, \ref{thm:intro2}, 
and \ref{thm:intro3} in the $GL(r)$-case. Note that these results are in the case $l>2g-2$. 
Throughout this section,
we use the notation for the base-change of pre-triangulated subcategories
from Subsection~\ref{subsec:basechange}. We also postpone the proof of several technical results to Section~\ref{s6}.
\subsection{Poincaré line bundles over the regular locus}

In this subsection, we prove Proposition \ref{prop:Px}, which is a compatibility result between the Poincaré line bundle and quasi-BPS categories, and then we use it to prove Theorem \ref{thm:ff}.

A $L$-twisted Higgs bundle is called \textit{regular} if 
it corresponds (under the spectral construction) to a line bundle on a spectral curve in 
$S=\mathrm{Tot}_C(L)$. 
Throughout this section, we assume that $l>2g-2$. 
We denote by 
\begin{align}\label{def:reg}
    \mathcal{M}^L(r, \chi)^{\rm{reg}} \subset \mathcal{M}^L(r, \chi)
\end{align}
the open substack consisting of regular Higgs bundles. The formula (\ref{line:P}) 
gives a natural line 
bundle 
\begin{align}\label{def:Preg}
\mathcal{P}^{\rm{reg}} \to \mathcal{M}^L(r, w+1-g^{\rm{sp}})^{\rm{reg}} \times_B 
\mathcal{M}^L(r, \chi). 
\end{align}
Namely, let 
\begin{align*}(F, E) \in \mathcal{M}^L(r, w+1-g^{\rm{sp}})^{\rm{reg}} \times_B 
\mathcal{M}^L(r, \chi)
\end{align*}
be a pair of Higgs bundles 
over $b \in B$. 
Let $\mathcal{C}_b \subset S$ be the spectral curve corresponding to $b$. 
Then $F$ is a line bundle on $\mathcal{C}_b$ and $E$ is a coherent sheaf on $\mathcal{C}_b$. 
We have 
\begin{align}\label{line:2}
    \mathcal{P}^{\rm{reg}}|_{(F, E)}=
    \det R\Gamma(F \otimes_{\mathcal{O}_{\mathcal{C}_b}} E)
    &\otimes \det R\Gamma(F)^{-1} \\
    &\notag \otimes \det R\Gamma(E)^{-1}
    \otimes \det R\Gamma(\mathcal{O}_{\mathcal{C}_b}). 
\end{align}
Note that the above formula makes sense since $F$ is a line bundle 
on $\mathcal{C}_b$. 
The line bundle $\mathcal{P}^{\rm{reg}}$ 
has bi-weight $(\chi+g^{\rm{sp}}-1, w)$. 
Let $x \in \mathcal{M}^L(r, w+1-g^{\rm{sp}})^{\rm{reg}}$ be a point which corresponds to 
a line bundle $F$ on $\mathcal{C}_b$. 
Then the formula (\ref{line:2}) 
determines an object 
\begin{align*}
    \mathcal{P}^{\rm{reg}}_x \in \Coh(\mathcal{M}^L(r, \chi))
\end{align*}
which is a line bundle on 
the fiber of the Hitchin map 
$\mathcal{M}^L(r, \chi) \to B$
at $b \in B$. 
We show that the above object lies in the quasi-BPS category: 
\begin{prop}\label{prop:Px}
For any $x \in \mathcal{M}^L(r, w+1-g^{\rm{sp}})^{\rm{reg}}$
corresponding to a line bundle bundle $F$ on a spectral 
curve $\mathcal{C}_b$, 
we have 
\begin{align*}\mathcal{P}^{\rm{reg}}_x \in \mathbb{T}^L(r, \chi)_{w}.
\end{align*}
\end{prop}
\begin{proof}
    Let $y \in \mathcal{M}^L(r, \chi)$ be 
    a closed point corresponding to the polystable 
    Higgs bundle with spectral curve $\mathcal{C}_b$:
    \begin{align*}
        E=\bigoplus_{i=1}^k V_i \otimes E_i,
    \end{align*}
    where, for all $1\leq i\leq k$, $V_i$ is a finite dimensional 
    vector space of dimension $d_i$ and $E_i$ is a 
    stable Higgs bundle such that 
    $(r_i, \chi_i)=(\rank(E_i), \chi(E_i))$ satisfies 
    \begin{align}\label{ratio:i}
        \frac{\chi_i}{r_i}=\frac{\chi}{r}. 
    \end{align}
    Let $G:=\mathrm{Aut}(E)=\prod_{i=1}^k GL(V_i)$. 
    Then $\mathcal{P}^{\rm{reg}}_x|_{y}$ is the following 
    character of $G$
    \begin{align}\label{char:xy}
        \mathcal{P}^{\rm{reg}}_x|_{y}=\bigotimes_{i=1}^k 
        (\det V_i)^{\chi(F \otimes E_i)-\chi(E_i)}.
    \end{align}
Note that we have 
\begin{align*}
\dim \Ext_S^1(E_i, E_j)=r_i r_j l +\delta_{ij},
    \end{align*}
    where $S$ is the local surface $\mathrm{Tot}_C(L)$. 
    Then the representation space of the Ext-quiver 
    associated with $\{E_1, \ldots, E_k\}$
of dimension vector $\bm{d}=(d_i)_{1\leq i \leq k}$
is given by the affine space    
\begin{align*}
R(\bm{d})=\bigoplus_{1\leq i, j \leq k} \Hom(V_i, V_j)^{\oplus (r_i r_j l+\delta_{ij})}.    \end{align*}
Let $T \subset G$ be a maximal torus. 
As in Lemma~\ref{lem:qbps}, we set 
\begin{align*}
    n_{\lambda}=\big\langle \lambda, \det(R(\bm{d})^{\lambda>0})
    -\det(\mathfrak{gl}(\bm{d})^{\lambda>0})\big\rangle, \ 
    \delta_y=\bigotimes_{i=1}^k (\det V_i)^{\otimes m_i},
\end{align*}
where $\lambda$ is a cocharacter 
$\lambda \colon \mathbb{C}^{\ast} \to T$, 
$(r, \chi)=d(r_0, \chi_0)$ with  $(r_0, \chi_0)$ coprime,
and $m_i$ is determined by $(r_i, \chi_i)=m_i(r_0, \chi_0)$. 
Let $M=\Hom(T, \mathbb{C}^{\ast})$ be the weight lattice 
of $T$, and set 
\begin{align*}
    \nabla_{\delta}:=\left\{\chi \in M_{\mathbb{R}} \,\Big|\,
    -\frac{n_{\lambda}}{2} \leq \left\langle \lambda, \chi-\frac{w}{d}\delta_y \right\rangle 
    \leq \frac{n_{\lambda}}{2} \right\},
\end{align*}
where the right hand side is after all the cocharacter 
$\lambda$ of $T$. 
By Lemma~\ref{lem:qbps}, it is enough to show that
$\mathcal{P}^{\rm{reg}}_x|_{y} \in \nabla_{\delta}$. 

Since $\mathcal{P}^{\rm{reg}}_x|_{y}$ is a Weyl-invariant
character of $T$, 
it is enough to show that 
\begin{align}\label{ets:lambda}
  -\frac{n_{\lambda}}{2} \leq \left\langle \lambda, \mathcal{P}^{\rm{reg}}_x|_{y}-\frac{w}{d}\delta_y \right\rangle 
    \leq \frac{n_{\lambda}}{2}
    \end{align}
    for any Weyl-invariant cocharacter $\lambda$. 
    So we may assume that $\lambda$ acts on $V_i$ by 
weight $\lambda_i$ for all $1\leq i\leq k$. 
There is a decomposition \[\{1, \ldots, k\}=I_1 \sqcup 
\cdots \sqcup I_l,\]
where $\lambda_i$ is constant on each $I_j$. 
It is enough to check the case of $l=2$, as
a codimension one face 
of the polytope $\nabla_{\delta} \cap M_{\mathbb{R}}^W$
is parallel to 
$\lambda^{\perp}$ for such a cocharacter $\lambda$. 
Therefore we can assume that there is $1\leq m \leq k$
such that 
$\lambda$ acts on $V_i$ for $1\leq i\leq m$ by weight 
one, and on $V_i$ for $m<i\leq k$ by weight zero. 

For such a cocharacter $\lambda$, by (\ref{char:xy}) one can compute that 
\begin{align}\label{P:delta}
\left\langle \lambda, \mathcal{P}^{\rm{reg}}_x|_{y} -\frac{w}{d}\delta_y \right\rangle 
    =\sum_{i=1}^m d_i \left(\chi(F \otimes E_i)-\frac{r_i}{r}(w+\chi) \right)
\end{align}
and 
\begin{align*}
    \frac{n_{\lambda}}{2}=\frac{l}{2}\sum_{\begin{subarray}{c}1\leq i\leq m \\
    m<j\leq k\end{subarray}}d_i d_j r_i r_j=\frac{l}{2} r' r'',
\end{align*}
where $r'$ and $r''$ are given by 
\begin{align*}
    r'=\sum_{i=1}^m d_i r_i, \ r''=\sum_{i=m+1}^k d_i r_i. 
\end{align*}
Note that $r=r'+r''$. 
Let $D_i \subset S$ be the pure one-dimensional 
support of $E_i^{\oplus d_i}$, so that 
we have $\mathcal{C}_b=D_1+\cdots+D_k$
as a divisor on $S$. 
Using (\ref{formula:chiD}), we have 
\begin{align*}
\chi(\mathcal{O}_{D_1+\cdots+D_m})
&=\left(\sum_{i=1}^m d_i r_i\right)\left(1-g+\frac{l}{2}
-\frac{l}{2}\sum_{i=1}^m d_i r_i  \right)\\
&=\frac{r'}{r}(1-g^{\rm{sp}})+\frac{l}{2}r' r'',
\end{align*}
where $g^{\rm{sp}}$ is given by (\ref{formula:gD}). 
On the other hand, using (\ref{ratio:i})
we have that:
\begin{align*}
    \frac{\sum_{i=1}^m d_i\chi(E_i)}{r'}=\frac{\chi}{r}. 
\end{align*}
It follows that, as $F$ is a line bundle on $\mathcal{C}_b$, we have 
\begin{align*}
    \sum_{i=1}^m d_i \chi(F \otimes E_i) -\chi(F|_{D_1+\cdots+D_m})
    &=\sum_{i=1}^m d_i \chi(E_i) -\chi(\mathcal{O}_{D_1+\cdots+D_m}) \\
    &=\frac{r'}{r} \chi-\frac{r'}{r}(1-g^{\rm{sp}})-\frac{l}{2}r' r''. 
\end{align*}
Therefore 
the right hand side of (\ref{P:delta}) 
is equal to 
\begin{align}\label{equals:chi}
   \chi(F|_{D_1+\cdots+D_m})
-\frac{r'}{r}(w+1-g^{\rm{sp}})-\frac{l}{2} r' r''. 
\end{align}
There are exact sequences
\begin{align*}
    &0 \to F|_{D_{m+1}+\cdots+D_k}\left(-\sum_{\begin{subarray}{c}1\leq i\leq k \\
    m<j\leq k\end{subarray}}D_i D_j  \right)
    \to F \to F|_{D_1+\cdots+D_m} \to 0, \\
    &0 \to F|_{D_{1}+\cdots+D_m}\left(-\sum_{\begin{subarray}{c}1\leq i\leq k \\
    m<j\leq k\end{subarray}}D_i D_j  \right)
    \to F \to F|_{D_{m+1}+\cdots+D_k} \to 0.
\end{align*}
By the stability of $F$, there are inequalities
\begin{align*}
    \frac{w+1-g^{\rm{sp}}}{r} \leq 
    \frac{\chi(F|_{D_1+\cdots+D_m})}{r'}
    \leq \frac{w+1-g^{\rm{sp}}}{r}+lr' r''. 
\end{align*}
By the formula (\ref{equals:chi}), 
the inequalities in (\ref{ets:lambda}) hold. 
\end{proof}

Denote by 
\begin{align*}
    \mathcal{M}^L(r, \chi)^{\rm{sreg}} \subset \mathcal{M}^L(r, \chi)^{\rm{reg}}
\end{align*}
the open substack 
consisting of line bundles on spectral curves which are also stable. 
Consider its good moduli space 
\begin{align*}
    \mathcal{M}^L(r, \chi)^{\rm{sreg}} \to M^L(r, \chi)^{\rm{sreg}} \subset 
    M^L(r, \chi),
\end{align*}
where the first map is a $\mathbb{C}^{\ast}$-gerbe and the second one is an 
open immersion. 
Recall the notation \eqref{def:minuscategory} and the category $D^{-}(-)$ of bounded above complexes of quasi-coherent sheaves.
Since the semiorthogonal decomposition (\ref{sod:main})
is linear over the Hitchin base $B$, 
there is an induced semiorthogonal decomposition, see Lemma~\ref{lem:sdo}:
\begin{align}\label{sod:prod}
    D^-&(\mathcal{M}^L(r, w+1-g^{\rm{sp}})^{\rm{sreg}} \times_B 
    \mathcal{M}^L(r, \chi)) \\ \notag
    &=\left\langle \left(\boxtimes_{i=1}^k \mathbb{T}^L(d_i r_0, d_i \chi_0)_{w_i}\right)^{-}_{\mathcal{M}^L(r, w+1-g^{\rm{sp}})^{\rm{sreg}}} \right\rangle, 
\end{align}
where the right hand side is after all partitions $d=d_1+\cdots+d_k$ and $w_i \in \mathbb{Z}$
as in (\ref{sod:main}). 
We set
\begin{align*}    \mathcal{P}^{\rm{sreg}}:=\mathcal{P}^{\rm{reg}}|_{\mathcal{M}^L(r, w+1-g^{\rm{sp}})^{\rm{sreg}}\times_B \mathcal{M}^L(r, \chi)}. 
\end{align*}
By Proposition~\ref{prop:Px} and Lemma~\ref{lem:Cix}, we have 
    \begin{align}\label{prop:sreg}
        \mathcal{P}^{\rm{sreg}} \in \left(\mathbb{T}^L(r, \chi)_w  \right)_{\mathcal{M}^L(r, w+1-g^{\rm{sp}})^{\rm{sreg}}}.
    \end{align}
Here we refer to Subsection~\ref{subsec:basechange} for the notation $(\mathbb{T}^L(r, \chi)_w)_{(-)}$. 

The main result of this subsection is the following theorem. An important ingredient in its proof is Proposition \ref{prop:CMQ} which helps reducing it to a statement over the locus $B^{\mathrm{red}}$ of reduced spectral curves.
    \begin{thm}\label{thm:ff}
    Suppose that 
    \begin{align*}
        r=2, l>3g-2, \mbox{ or } r=3, l>4g-3. 
    \end{align*}
    Then the induced functor 
    \begin{align}\label{Mqcoh:phi}
  \Phi_{\mathcal{P}^{\rm{sreg}}} \colon  D_{\rm{qcoh}}(\mathcal{M}^L(r, w+1-g^{\rm{sp}})^{\rm{sreg}})_{-\chi+1-g^{\rm{sp}}}
        \to \mathbb{T}_{\rm{qcoh}}^L(r, \chi)_{w}
    \end{align}
    is fully-faithful. 
            \end{thm}
\begin{proof}
For simplicity, let $\mathcal{M}=\mathcal{M}^L(r, \chi)$
and $\mathcal{T}=\mathcal{M}^L(r, w+1-g^{\rm{sp}})^{\rm{sreg}}$. 
Consider the good moduli space 
\begin{align}\label{t:gerb}
    \mathcal{T} \to T \subset M^L(r, w+1-g^{\rm{sp}}),
\end{align}
where the first morphism is a $\mathbb{C}^{\ast}$-gerbe
and the second one is an open immersion. 
Let $\alpha \in \mathrm{Br}(T)$ be the Brauer class
corresponding to the $\mathbb{C}^{\ast}$-gerbe (\ref{t:gerb}), 
and set $\beta:=\alpha^{-\chi+1-g^{\rm{sp}}}$. 
There is an equivalence 
\begin{align}\label{equiv:T}
    D_{\rm{qcoh}}(\mathcal{T})_{-\chi+1-g^{\rm{sp}}}\simeq 
    D_{\rm{qcoh}}(T, \beta). 
\end{align}
Below we identify the left hand side of (\ref{Mqcoh:phi}) 
with $D_{\rm{qcoh}}(T, \beta)$ by the above equivalence. 

Let $\Phi_{\mathcal{P}^{\rm{sreg}}}^L$ be the left adjoint of $\Phi_{\mathcal{P}^{\rm{sreg}}}$
and denote by 
\begin{align*}
\mathcal{Q} \in D^b(T\times_B T, -\beta \boxtimes \beta)
    \end{align*}
    the kernel object of $\Phi_{\mathcal{P}^{\rm{sreg}}}^L \circ \Phi_{\mathcal{P}^{\rm{sreg}}}$, 
    see Subsection~\ref{subsec:fmBPS}. 
    Note that $\mathcal{Q}$ is an object in 
    $D^b(T\times_B T, -\beta\boxtimes \beta)$ by the formula (\ref{isom:Q}) together with the 
    fact that $\mathcal{P}^{\rm{sreg}}$ is a line bundle.
    The adjunction morphism 
    \[
    \Phi_{\mathcal{P}^{\rm{sreg}}}^L \circ \Phi_{\mathcal{P}^{\rm{sreg}}} \to \id
   \]
    corresponds to a morphism 
    \begin{align*}
        \eta \colon \mathcal{Q} \to \mathcal{O}_{\Delta_T}. 
    \end{align*}
    It is enough to show that $\eta$ is an isomorphism, and we 
    apply Proposition~\ref{prop:CMQ} to prove it. 
    Let $B^{\rm{red}} \subset B$ be the open subset defined 
    in (\ref{open:B}), and set $T^{\circ}:=T \times_B B^{\rm{red}}$. 
    By Lemma~\ref{lem:codim} and Proposition~\ref{prop:CMQ}, it is enough to show that 
    \begin{align}\label{eta:isom}
\eta|_{B^{\rm{red}}} \colon \mathcal{Q}|_{T^{\circ}\times_{B^{\rm{red}}}\times T^{\circ}} \to 
\mathcal{O}_{\Delta_{T^{\circ}}} 
    \end{align}
    is an isomorphism. We next explain that this is indeed the case using the duality of compactified Jacobians of reduced curves of Melo--Rapagnetta--Viviani \cite{MRVF} and window categories.

Take $b \in B^{\rm{red}}$ which corresponds to a reduced 
spectral curve
\begin{align*}
    \mathcal{C}_b=D_1+\cdots+D_k
\end{align*}
where each $D_i$ is an irreducible curve in $S=\mathrm{Tot}_C(L)$ 
and $D_i \neq D_j$ for $1\leq i\neq j\leq k$. 
Let \[\mathcal{C} \to B\] be the universal spectral curve. 
We set 
$\widehat{B}_b:=\Spec \widehat{\mathcal{O}}_{B, b}$, 
and take the fiber products
\begin{align*}
    \widehat{\mathcal{M}}_b :=\mathcal{M}\times_B \widehat{B}_b, \ 
     \widehat{\mathcal{T}}_b :=\mathcal{T}\times_B \widehat{B}_b, \ 
      \widehat{\mathcal{C}}_b :=\mathcal{C}\times_B \widehat{B}_b. 
\end{align*}
Let $\mathcal{L}$ be an ample line bundle on $C$, and 
denote by $\widehat{\mathcal{L}}_b$ its pull-back 
by the projection $\widehat{\mathcal{C}}_b \to C$. 
Note that 
$\widehat{\mathcal{C}}_b \to \widehat{B}_b$ is a flat family of 
curves in $S$ over $\widehat{B}_b$ 
with central fiber $\mathcal{C}_b$, and $\widehat{\mathcal{M}}_b$ is 
identified with the relative moduli stack of 
$\widehat{\mathcal{L}}_b$-semistable sheaves on 
$\widehat{\mathcal{C}}_b \to \widehat{B}_b$ 
with rank one and holomorphic Euler characteristic $\chi$. 

On the other hand, there exist line bundles $\mathcal{L}_j$ 
for $1\leq j\leq k$ on $\widehat{\mathcal{C}}_b$
such that $\det(\mathcal{L}_j|_{D_i})=\delta_{ij}$. 
We perturb $\widehat{\mathcal{L}}_b$ to a $\mathbb{Q}$-line 
bundle on $\widehat{\mathcal{C}}_b$:
\begin{align*}\widehat{\mathcal{L}}_b'=\widehat{\mathcal{L}}_b \otimes 
\bigotimes_{i=1}^k \mathcal{L}_i^{\varepsilon_i}
\end{align*}
for $\varepsilon_i \in \mathbb{Q}$ with $\lvert \varepsilon_i \rvert \ll 1$
and $\sum_{i=1}^k \varepsilon_i=0$. 
We denote by 
$\widehat{\mathcal{M}}_b' \to \widehat{B}_b$
the relative moduli stack of $\widehat{\mathcal{L}}_b'$-semistable 
sheaves on $\widehat{\mathcal{C}}_b \to \widehat{B}_b$
of rank one and holomorphic Euler characteristic $\chi$. 
Then for a generic choice of $\varepsilon_i$, 
the moduli stack $\widehat{\mathcal{M}}_b'$ is fine, i.e. 
it only consists of stable sheaves and the good moduli space 
$\widehat{\mathcal{M}}_b' \to \widehat{M}_b'$
is a trivial $\mathbb{C}^{\ast}$-gerbe, or equivalently  
there is a universal sheaf on $\widehat{M}_b' \times_{\widehat{B}_b} \widehat{\mathcal{C}}_b$. 
The existence of the universal sheaf follows from the existence of $\mathcal{L}_i$ 
and applying~\cite[Theorem~6.5]{Hu}. 
In particular, there is an equivalence 
\begin{align}\label{equiv:hatw}
    D^b(\widehat{\mathcal{M}}_b')_w \simeq D^b(\widehat{M}_b'). 
\end{align}
Note that we have the open immersion 
$\widehat{\mathcal{M}}_b' \subset \widehat{\mathcal{M}}_b$ and 
$\widehat{M}_b' \to \widehat{B}_b$ is a fine relative compactified 
Jacobian of $\widehat{\mathcal{C}}_b \to \widehat{B}_b$. 

Similarly $\widehat{\mathcal{T}}_b \to \widehat{T}_b$ is a trivial 
$\mathbb{C}^{\ast}$-gerbe, and 
hence $\beta|_{\widehat{T}_b}$ is trivial. 
The equivalence 
(\ref{equiv:T}) induces equivalences 
\begin{align}\label{equiv:T2}
    D_{\rm{qcoh}}(\widehat{\mathcal{T}}_b)_{-\chi+1-g^{\rm{sp}}} \simeq 
    D_{\rm{qcoh}}(\widehat{T}_b, \beta|_{\widehat{T}_b}) \simeq 
    D_{\rm{qcoh}}(\widehat{T}_b). 
\end{align}

The quasi-BPS category 
$\mathbb{T}=\mathbb{T}^L(r, \chi)_w$ in 
$D^b(\mathcal{M})_w$ restricts to the quasi-BPS categories, 
see Remark~\ref{rmk:open}:
\begin{align*}\widehat{\mathbb{T}}_b \subset D^b(\widehat{\mathcal{M}}_b)_w, 
\ \widehat{\mathbb{T}}_{b, \rm{qcoh}} \subset D_{\rm{qcoh}}(\widehat{\mathcal{M}}_b)_w. 
\end{align*}
For a generic choice of $\varepsilon_i$, 
the window theorem~\cite{BFK, halp, T}
implies that the composition 
\begin{align}\label{eq:window}
    \widehat{\mathbb{T}}_b \hookrightarrow 
    D^b(\widehat{\mathcal{M}}_b)_w \twoheadrightarrow 
    D^b(\widehat{\mathcal{M}}_b')_w \simeq D^b(\widehat{M}_b')
\end{align}
is an equivalence, see Lemma~\ref{lem:eq:window}. 
Here the second arrow is the pull-back 
of the open immersion 
$\widehat{\mathcal{M}}_b' \subset \widehat{\mathcal{M}}_b$
and the third arrow is given by (\ref{equiv:hatw}). 
We set
\begin{align*}
    \widehat{\mathcal{P}}_b :=\mathcal{P}^{\rm{reg}}|_{\widehat{T}_b \times_{\widehat{B}_b}\widehat{\mathcal{M}}_b}, \ 
     \widehat{\mathcal{P}}_b' :=\mathcal{P}^{\rm{reg}}|_{\widehat{T}_b \times_{\widehat{B}_b}\widehat{\mathcal{M}}'_b}. 
\end{align*}
We have the commutative diagram 
\begin{align}\label{Dqcoh:T}
    \xymatrix{
    D_{\rm{qcoh}}(\widehat{T}_b) \ar[r]^-{\sim} \ar@{=}[d] &
D_{\rm{qcoh}}(\widehat{\mathcal{T}}_b)_{-\chi+1-g^{\rm{sp}}} \ar[r]^-{\Phi_{\widehat{\mathcal{P}}_b}} \ar@{=}[d] & 
\widehat{\mathbb{T}}_{b, \rm{qcoh}} \ar[d]_-{\sim} & \\ 
 D_{\rm{qcoh}}(\widehat{T}_b) \ar[r]^-{\sim} &
D_{\rm{qcoh}}(\widehat{\mathcal{T}}_b)_{-\chi+1-g^{\rm{sp}}} \ar[r]^-{\Phi_{\widehat{\mathcal{P}}_b'}}  & 
D_{\rm{qcoh}}(\widehat{\mathcal{M}}_b')_w \ar[r]^-{\sim} &D_{\rm{qcoh}}(\widehat{M}_b')    
    }
\end{align}
where the left horizontal arrows are equivalences (\ref{equiv:T2})
and the right vertical arrow is induced by (\ref{eq:window}). 
The functor $\Phi_{\widehat{\mathcal{P}}_b}$ has image in 
$\widehat{\mathbb{T}}_{b, \rm{qcoh}}$, see (\ref{prop:sreg}), which itself is a corollary of Proposition \ref{prop:Px}. 
The composition of the bottom horizontal arrow is fully-faithful by~\cite[Theorem~A]{MRVF}
(more precisely, its version for families of reduced planar curves). 
Therefore the top arrows are fully-faithful, which implies that 
the morphism (\ref{eta:isom}) is an isomorphism 
on $\widehat{T}_b \times_{\widehat{B}_b} \widehat{T}_b$. 
Since this holds for any $b \in B^{\rm{red}}$, the morphism 
(\ref{eta:isom}) is an isomorphism. 
\end{proof}

\begin{lemma}\label{lem:eq:window}
The composition functor (\ref{eq:window}) is an equivalence.     \end{lemma}
\begin{proof}
    The lemma is a consequence of a window theorem for smooth Artin stacks which 
    are locally GIT quotient stacks by torus action. 
    As an analogous result is obtained 
    in~\cite[Section~6.4.1]{T}, 
    we only give an outline of the proof. 
    
    We 
    apply~\cite[Theorem~6.3.13]{T} as follows.
    First, as Ext-quivers at closed points of $\mathcal{M}$
    are symmetric, we take a maximal symmetric structure 
    in the statement of~\cite[Theorem~6.3.13]{T}. 
    Next, let 
    $\widehat{\mathcal{E}}_b$ be the universal sheaf on 
    $\widehat{\mathcal{C}}_b \times_{\widehat{B}_b} \times \widehat{\mathcal{M}}_b$
    and define 
    the following $\mathbb{Q}$-line bundle on $\widehat{\mathcal{M}}_b$:
    \begin{align*}
        \Theta:=\det Rp_{\mathcal{M}\ast}(\widehat{\mathcal{E}}_b \boxtimes \widehat{\mathcal{L}}_b') \otimes \det Rp_{\mathcal{M}\ast}(\widehat{\mathcal{E}}_b \boxtimes \widehat{\mathcal{L}}_b)^{-1}. 
    \end{align*}
    Here, we denote by $p_{\mathcal{M}}$ the projection onto the factor $\widehat{\mathcal{M}}_b$. 
    Then an argument similar to~\cite[Lemma~6.4.5]{T} shows that 
   $\widehat{\mathcal{M}}_b'=(\widehat{\mathcal{M}}_b)^{\Theta \text{-ss}}$.
   Also, the condition $\sum_{i=1}^k \varepsilon_i=0$ implies that
   the $\mathbb{Q}$-line bundle $\Theta$ descends to 
   the $\mathbb{C}^{\ast}$-rigidification 
   $\widehat{\mathcal{M}}_b^{\rm{rig}}$ of $\widehat{\mathcal{M}}_b$. We have that $\widehat{M}_b'=(\widehat{\mathcal{M}}_b^{\rm{rig}})^{\Theta\text{-ss}}$. 
   
   For a generic choice of $\varepsilon_i$, the $\mathbb{Q}$-line bundle 
   $\Theta$ on $\widehat{\mathcal{M}}_b^{\rm{rig}}$ satisfies the genericity condition in~\cite[Theorem~6.3.3]{T} as in the proof of~\cite[Lemma~6.4.6]{T}. 
   Indeed, for any closed point 
   $y \in \widehat{\mathcal{M}}_b^{\rm{rig}}$, 
the associated dimension vector of its Ext-quiver is the primitive vector $(1, 1, \ldots, 1)$
as $b \in B^{\rm{red}}$ corresponds to a reduced spectral curve. Thus, for a generic choice of $\varepsilon_i$, 
the $\Aut(y)$-character $\Theta|_{y}$
does not lie on walls in the character lattice of $\Aut(y)$. 

   We also note that $\Omega_{\widehat{\mathcal{M}}_b^{\rm{rig}}}[-1]=\widehat{\mathcal{M}}_b^{\rm{rig}}$ as 
   $\mathcal{M}$ is smooth. Therefore the lemma follows from~\cite[Theorem~6.3.13]{T}. 
\end{proof}

\subsection{The Cohen-Macaulay extension of the Poincaré line bundle}\label{subsec:rank2}

In this subsection, we recall the extension $\mathcal{P}$ of the Poincaré line bundle from the regular locus to the full semistable locus in the rank $2$ case due to Li (and based on previous work of Arinkin). We do not know whether $\mathcal{P}$ is in the corresponding product of BPS categories in general. This property is automatically satisfied in some cases, which allow us to prove equivalence of BPS categories, see Corollary~\ref{cor:r=2}.

For a given tuple $(r, \chi, w)$, we use the short hand notation from \eqref{write:simple}:
\begin{align*}\mathcal{M}=\mathcal{M}^L(r, \chi), \ 
\mathcal{M}'=\mathcal{M}^L(r, w+1-g^{\rm{sp}}).
\end{align*}
We denote by $\mathcal{M}^{\rm{reg}} \subset \mathcal{M}$ the 
regular part as in (\ref{def:reg}). 
We set
\begin{align*}
(\mathcal{M}'\times_B \mathcal{M})^{\sharp}:=
(\mathcal{M}'^{\rm{reg}}\times_B \mathcal{M}) \cup 
(\mathcal{M}'\times_B \mathcal{M}^{\rm{reg}}). 
\end{align*}
The line bundle (\ref{def:Preg}) naturally extends to a
line bundle 
\begin{align}\label{line:Preg}
\mathcal{P}^{\sharp} \to (\mathcal{M}'\times_B \mathcal{M})^{\sharp}
\end{align}
since the formula (\ref{line:2}) makes sense if either $F$ or $E$ is a line bundle. 
In the rank two case, the following result is proved in~\cite{MLi}. 
\begin{thm}\emph{(\cite{MLi})}\label{thm:mao}
If $r=2$ and $l>2g$, the line bundle $\mathcal{P}^{\sharp}$ uniquely extends 
to a maximal Cohen–Macaulay sheaf $\mathcal{P}$ on $\mathcal{M}'\times_B \mathcal{M}$
flat over $\mathcal{M}'$.
\end{thm}

We will use the following lemma: 
\begin{lemma}\label{lem:CMP}
Let $\mathcal{P}$ be the maximal Cohen–Macaulay sheaf in 
    Theorem~\ref{thm:mao}. Then $\mathcal{P}$ is also flat over 
    $\mathcal{M}$. Moreover its derived dual 
    \begin{align*}
        \mathcal{P}^{\vee}:=R\mathcal{H}om(\mathcal{P}, 
        \mathcal{O}_{\mathcal{M}'\times_B \mathcal{M}})
    \end{align*}
    is also maximal Cohen–Macaulay sheaf which is flat over 
    $\mathcal{M}'$ and $\mathcal{M}$. 
\end{lemma}
\begin{proof}
    Let $\iota, \iota^{\sharp}$ be the involutions 
    \begin{align*}
        \iota \colon \mathcal{M}\times_B \mathcal{M}'\stackrel{\cong}{\to}
        \mathcal{M}'\times_B \mathcal{M}, \quad \iota^{\sharp} \colon (\mathcal{M}\times_B \mathcal{M}')^{\sharp}\stackrel{\cong}{\to}
        (\mathcal{M}'\times_B \mathcal{M})^{\sharp}
    \end{align*}
    given by $(x, y) \mapsto (y, x)$. 
    Then $\iota^{\ast}\mathcal{P}$ is the unique Cohen–Macaulay extension 
    of the line bundle $\iota^{\sharp\ast}\mathcal{P}^{\sharp}$, 
    where the latter is given by the formula (\ref{line:2}) 
    switching $(E, F)$ with $(F, E)$.
    Therefore the result of~\cite{MLi} applies to $\iota^{\sharp\ast}\mathcal{P}^{\sharp}$
    and we obtain its maximal Cohen–Macaulay extension 
    $\mathcal{P}^{\iota}$ on $\mathcal{M}\times_B \mathcal{M}'$
    which is flat over $\mathcal{M}$. 
    By the uniqueness of Cohen–Macaulay extension, we have $\mathcal{P}^{\iota} \cong \iota^{\ast}\mathcal{P}$. 
    Therefore $\mathcal{P}$ is also flat over $\mathcal{M}$. 

As $\mathcal{P}$ is maximal Cohen–Macaulay, 
its derived dual $\mathcal{P}^{\vee}$ is also maximal Cohen–Macaulay 
sheaf. We show that it is also flat over $\mathcal{M}$ and $\mathcal{M}'$. 
Let $\mathcal{M}''=\mathcal{M}^L(2, -w+1-g^{\rm{sp}})$. 
We have the isomorphisms 
\begin{align*}
    \eta \colon \mathcal{M}'\times_B \mathcal{M} \stackrel{\cong}{\to} \mathcal{M}''\times_B \mathcal{M}, 
    \quad 
    \eta^{\sharp} \colon (\mathcal{M}'\times_B \mathcal{M})^{\sharp} \stackrel{\cong}{\to} (\mathcal{M}''\times_B \mathcal{M})^{\sharp}
\end{align*}
which sends $(E, F)$ for a pair of sheaves on 
a spectral curve $D$ to 
$(E^{\vee}, F)$. 
The same construction as in $\mathcal{P}^{\sharp}$ 
yields a line bundle 
$(\mathcal{P}')^{\sharp}$ on $(\mathcal{M}''\times_B \mathcal{M})^{\sharp}$.  
Then from the formula (\ref{line:2}), one can check that 
$\eta^{\sharp\ast}(\mathcal{P}')^{\sharp} \cong (\mathcal{P}^{\sharp})^{\vee}$
(cf.~\cite[Lemma~6.2]{Ardual}). 
On the other hand, the result of~\cite{MLi} applies 
to 
the line bundle $(\mathcal{P}')^{\sharp}$ and it uniquely 
extends to a maximal Cohen–Macaulay sheaf $\mathcal{P}'$ on $\mathcal{M}''\times_B \mathcal{M}$
flat over $\mathcal{M}''$ and $\mathcal{M}$. 
By the uniqueness of Cohen–Macaulay extension, we have 
$\eta^{\ast}\mathcal{P}' \cong 
    \mathcal{P}^{\vee}$, 
hence $\mathcal{P}^{\vee}$ is flat over $\mathcal{M}$ and $\mathcal{M}'$. 
\end{proof}

\begin{remark}\label{nonflat}
Note that in~\cite[Remark~4]{MLi} it is mentioned that the Cohen–Macaulay extension of 
$\mathcal{P}^{\sharp}$ is not flat over the second factor. 
This is because in~\cite[Remark~4]{MLi} the Cohen–Macaulay extension is also considered 
including unstable locus in the second factor, and the non-flatness happens at 
unstable points. 
\end{remark}

The following lemma will be used later. 
\begin{lemma}\label{replace}
For a line bundle $\mathcal{L} \in \mathrm{Pic}^0(\mathcal{C})$, replace $\mathcal{E}'$ with $\mathcal{L}\boxtimes \mathcal{E}'$
in the formula (\ref{line:P22}) to obtain another 
line bundle 
\begin{align*}
    \mathcal{P}_{\mathcal{L}}^{\sharp} \to (\mathcal{M}'\times_B \mathcal{M})^{\sharp}. 
\end{align*}
Then 
there is an étale map $g \colon B' \to B$ such that 
$g^{\ast}\mathcal{P}^{\sharp}$ and $g^{\ast}\mathcal{P}_{\mathcal{L}}^{\sharp}$ 
are isomorphic up to tensoring with an element in $\mathrm{Pic}^0(\mathcal{M}\times_B B')$. 
\end{lemma}
\begin{proof}
    For each $b \in B$, we can write $\mathcal{L}=\mathcal{O}_{\mathcal{C}}(\mathcal{D}_2-\mathcal{D}_1)$
    for smooth divisor $\mathcal{D}_i \subset \mathcal{C}$ such that \[\mathcal{D}=\mathcal{D}_1 \cup \mathcal{D}_2 \to B\] is étale at $b$. 
       There is an étale map $(B', b') \to (B, b)$ such that 
    $\mathcal{D}\times_B B'$ is a disjoint union of sections of $\mathcal{C}'\times_B B' \to B'$. 
    Hence we may assume that $\mathcal{D}_1$ and $\mathcal{D}_2$ 
    are disjoint union of sections \[\mathcal{D}_i=\bigsqcup_{j=1}^a \mathcal{D}_{ij}.\] 
  Then an easy computation shows that 
  the line bundles $\mathcal{P}$ and $\mathcal{P}_{\mathcal{L}}$ differ
    by the tensor product of 
    \begin{align*}
    \bigotimes_{j=1}^a (\det(\mathcal{E}|_{\mathcal{D}_{1j}}) \otimes 
    \det(\mathcal{E}|_{\mathcal{D}_{2j}})^{-1}),
    \end{align*}
    which is an element of $\mathrm{Pic}^0(\mathcal{M})$. 
        To make sense of the above formula, we note that $\mathcal{D}_{ij} \times_B \mathcal{M} \hookrightarrow \mathcal{C}\times_B \mathcal{M}$ is a regular closed immersion (thus quasi-smooth), 
    so the derived pull-back $\mathcal{E}|_{\mathcal{D}_{ij}\times_B \mathcal{M}}$ is bounded, hence perfect on 
    $\mathcal{M}$ so its determinant is indeed a well-defined line bundle on $\mathcal{M}$. 
    \end{proof}

We expect that the Cohen-Macaulay sheaf in Theorem~\ref{thm:mao}
gives a correspondence between quasi-BPS categories. 
We state the following assumption: 
\begin{assum}\label{conj:CMext}
For a tuple 
$(2, \chi, w)$ satisfying the BPS condition, 
the maximal Cohen–Macaulay sheaf $\mathcal{P}$ in Theorem~\ref{thm:mao} satisfies 
\begin{align*}
    \mathcal{P} \in \mathbb{T}^L(2, w+1-g^{\rm{sp}})_{\chi+g^{\rm{sp}}-1}\boxtimes_B \mathbb{T}^L(2, \chi)_w. 
\end{align*}
    \end{assum}
    \begin{remark}
        We do not know whether Assumption~\ref{conj:CMext} holds in general. 
        In Theorem~\ref{thm:CMext}, we show that Assumption~\ref{conj:CMext} holds 
        for $C=\mathbb{P}^1$. Note that the assumption also holds automatically if $w+1-g^{\mathrm{sp}}$ and $\chi$ are odd.
    \end{remark}
Under Assumption~\ref{conj:CMext}, there is an
induced functor 
\begin{align}\label{equiv:expect}
    \Phi_{\mathcal{P}} \colon 
   \mathbb{T}^L(2, w+1-g^{\rm{sp}})_{-\chi+1-g^{\rm{sp}}} 
   \to \mathbb{T}^L(2, \chi)_w
\end{align}
which we expect to be the equivalence claimed in Conjecture \ref{conj:T0}.

\begin{thm}\label{thm:r=21}
Under Assumption~\ref{conj:CMext}, Conjecture~\ref{conj:T0} holds for $r=2$ and $l>\mathrm{max}\{3g-2, 2g\}$. 
\end{thm}
\begin{proof}
We discuss the case when $l$ is even, as the case of $l$ odd is similar. 
    The tuple $(2, \chi, w)$ satisfies the BPS condition, 
so either $(2, \chi)$ or $(2, w+1-g^{\rm{sp}})$ are coprime. 
It is enough to consider the case when $(2, w+1-g^{\rm{sp}})$ are coprime. 
We have either that $(2, \chi)$ are coprime or $\gcd(2, \chi)=2$, i.e. 
$\chi$ is even. We only discuss the case when $\chi$ is even, as the other 
case is similar (and even easier). 

So we treat the case when $l$ is even, $\chi$ is even, and $w$ is odd. 
    Under these assumptions, we have 
    \begin{align*}
        \mathbb{T}^L(2, w+1-g^{\rm{sp}})_{-\chi-g^{\mathrm{sp}}+1} \simeq 
        D^b(M^L(2, 1)), \ 
        \mathbb{T}^L(2, \chi)_w \simeq D^b(\mathscr{A})
    \end{align*}
    where $\mathscr{A}$ is the sheaf of non-commutative algebras on $M^L(2, 0)$
    from Proposition~\ref{prop:TA}. 
    Note that $g^{\rm{sp}}=l+2g-1$ is odd, 
    so by Lemma~\ref{lem:period} we may assume that $w=g^{\rm{sp}}$, $\chi=1-g^{\rm{sp}}$. 
    We set $\mathcal{T}=\mathcal{M}^L(2, 1)$ and $T=M^L(2, 1)$. 
    The maximal Cohen–Macaulay sheaf in Theorem~\ref{thm:mao} has bi-weight $(0, g^{\rm{sp}})$,
    so by Assumption~\ref{conj:CMext} it descends to 
    give a maximal Cohen–Macaulay sheaf 
    \begin{align*}
        \mathcal{P} \in (\mathbb{T}^L(2, 1-g^{\rm{sp}})_{g^{\rm{sp}}})_{T}
        \subset D^b(T\times_B \mathcal{M}), 
    \end{align*}
    see \eqref{def:minuscategory} for the notation. 
    The above sheaf induces a Fourier-Mukai functor 
    \begin{align}\label{induce:Phi}
        \Phi_{\mathcal{P}} \colon D^b(T) \to \mathbb{T}^L(2, 1-g^{\rm{sp}})_{g^{\rm{sp}}}. 
    \end{align}

 Let $\Phi_{\mathcal{P}}^L$ be the left adjoint of
   $\Phi_{\mathcal{P}}$ and let $\mathcal{Q} \in D^b(T\times_B T)$
   be the kernel object of $\Phi_{\mathcal{P}}^L \circ 
   \Phi_{\mathcal{P}}$. 
   Note that $\mathcal{Q}$ is an object in $D^b(T\times_B T)$
   by the formula (\ref{isom:Q}) together with Lemma~\ref{lem:CMP} and Lemma~\ref{lem:MCM}. 
   There is a natural map
   \begin{equation}\label{map:Qdelta}
   \mathcal{Q} \to \mathcal{O}_{\Delta_T}.
   \end{equation}
   We next show that this map is an isomorphism. 
   The proof that \eqref{map:Qdelta} is an isomorphism uses the same argument from the proof of Theorem~\ref{thm:ff}, using~\cite[Theorem~A]{MRVF2}
   instead of~\cite[Theorem~A]{MRVF}. First, we have that $\mathcal{Q} \cong 
   \mathcal{O}_{\Delta_T}$ over $B^{\rm{red}} \subset B$. 
Indeed we replace $\mathcal{P}^{\rm{sreg}}$ with $\mathcal{P}$ in the notation of 
the proof of Theorem~\ref{thm:ff}. 
Then for $b \in B^{\rm{red}}$, the diagram (\ref{Dqcoh:T})
is replaced by 
\begin{align}\notag
    \xymatrix{
    D^b(\widehat{T}_b) \ar[r]^-{\sim} \ar@{=}[d] &
D^b(\widehat{\mathcal{T}}_b)_{0} \ar[r]^-{\Phi_{\widehat{\mathcal{P}}_b}} \ar@{=}[d] & 
\widehat{\mathbb{T}}_{b} \ar[d]_-{\sim} & \\ 
 D^b(\widehat{T}_b) \ar[r]^-{\sim} &
D^b(\widehat{\mathcal{T}}_b)_{0} \ar[r]^-{\Phi_{\widehat{\mathcal{P}}_b'}}  & 
D^b(\widehat{\mathcal{M}}_b')_w \ar[r]^-{\sim} &D^b(\widehat{M}_b').   
    }
\end{align}
The composition of the bottom horizontal arrows is an equivalence by~\cite[Theorem~A]{MRVF2}, 
therefore the top arrows are equivalences, which implies that 
the map 
\eqref{map:Qdelta} is an isomorphism over 
$B^{\rm{red}} \subset B$. 

   Note that the Hitchin map 
   $M^L(2, 1) \to B$ is flat and lci, see~\cite[Proposition~2.2.6]{MLi}. 
   Therefore using Proposition~\ref{prop:CMQ} and Lemma~\ref{lem:codim}, 
   we conclude that indeed \eqref{map:Qdelta} is an isomorphism, 
   hence the natural transform $\Phi_{\mathcal{P}}^L \circ \Phi_{\mathcal{P}} \to \id$
   is an isomorphism. 
   It follows that $\Phi_{\mathcal{P}}$ is fully-faithful. 
   Then as the Hitchin map $M^L(2, 1) \to B$ is proper, we can also apply Proposition~\ref{prop:equivT}
   to conclude that $\Phi_{\mathcal{P}}$ is an equivalence.     
\end{proof}

\begin{cor}\label{cor:r=2}
Conjecture~\ref{conj:T0} is true if $r=2$, 
$l>\mathrm{max}\{3g-2, 2g\}$, and $\chi$ and $w+l$ are both odd.
\end{cor}
\begin{proof}
As both $\chi$ and $w+l$ are both odd, 
    we have 
    \begin{align*}\mathbb{T}^L(2, w+1-g^{\rm{sp}})_{\chi+g^{\rm{sp}}-1}=
    D^b(\mathcal{M}')_{\chi+g^{\rm{sp}}-1}, \ 
    \mathbb{T}^L(2, \chi)_w=D^b(\mathcal{M})_w,
    \end{align*}
    thus Assumption~\ref{conj:CMext} is automatically satisfied.
    \end{proof}
    
  \begin{remark}\label{rmk:M21}
In the setting of Corollary~\ref{cor:r=2}, both sides 
of (\ref{equiv:expect}) are naturally equivalent to $D^b(M^L(2, 1))$, 
so $\Phi_{\mathcal{P}}$ gives an autoequivalence 
\begin{align*}
    \Phi_{\mathcal{P}} \colon D^b(M^L(2, 1)) \stackrel{\sim}{\to} D^b(M^L(2, 1)). 
\end{align*}
The above equivalence extends a Fourier-Mukai equivalence on the smooth generic fibers
of $M^L(2, 1) \to B$
given by Poincaré line bundle. 
  \end{remark}

  \subsection{The Cohen–Macaulay extension for $C=\mathbb{P}^1$}
We show that Assumption~\ref{conj:CMext}
is satisfied 
in the case of $C=\mathbb{P}^1$, thus 
proving Conjecture~\ref{conj:T0} for 
$C=\mathbb{P}^1$ and $r=2$, see Theorem~\ref{thm:CMext}.
Below we assume that $C=\mathbb{P}^1$, $r=2$, and $l>0$. 
Note that $g^{\rm{sp}}=l-1$ in this case, see the formula
(\ref{formula:gD}). 

Before we state and prove Theorem~\ref{thm:CMext}, we discuss some preliminary lemmas. The first goal is to show that twisted Abel-Jacobi maps provide étale neighborhoods of $M^L(2,1)$, see Lemma~\ref{lem:AJM}.
\begin{lemma}\label{lem:ideal}
Let $D=2C \subset S$ be the double curve supported 
on the zero section of $S \to C$, i.e. 
\begin{align}\label{double}
    D=\Spec \left(\mathcal{O}_S/\mathcal{O}_S(-2C)\right). 
\end{align}
Let $E$ be a stable sheaf on $D$ 
such that $(\rank(p_{\ast}E), \chi(E))=(2, 1)$ 
where $p \colon D \to C$ is the projection. 
Then \[E\cong I_{D, Z}^{\vee},\]
where $Z \subset D$ is a zero-dimensional subscheme 
of length $l-1$, $I_{D, Z} \subset \mathcal{O}_D$
is the ideal sheaf of $Z$, and 
$I_{D, Z}^{\vee}=\mathcal{H}om(I_{D, Z}, \mathcal{O}_D)$. 
\end{lemma}
\begin{proof}
    Since $\chi(E)=1$, we have $h^0(E)\geq 1$, 
    hence there is a non-zero morphism 
    \[s \colon \mathcal{O}_D \to E.\] 
    Suppose that $s$ is not generically surjective. 
    Then $\mathrm{Im}(s)$ is the structure sheaf of 
    a one-dimensional subscheme in $D$ with curve class $C$, 
    which is reduced as $E$ is pure one-dimensional. 
    Therefore $\mathrm{Im}(s)=\mathcal{O}_C$. 
    However, note that \[\chi(\mathcal{O}_C)=1>\chi(E)/2=1/2,\] 
    (here we use in an essential way that $C=\mathbb{P}^1$)
    which contradicts that $E$ is stable. 
    Therefore $s$ is generically surjective. 
As $\rank(p_{\ast}\mathcal{O}_D)=\rank(p_{\ast}E)=2$, the map 
$s$ is also generically an isomorphism. Hence $s$ is injective 
as $\mathcal{O}_D$ is pure one-dimensional. 

It follows that there is an exact sequence 
\begin{align*}
    0 \to \mathcal{O}_D \to E \to Q \to 0
\end{align*}
for a zero-dimensional sheaf $Q$. 
We apply $R\mathcal{H}om(-, \mathcal{O}_D)$ to the above 
exact sequence. As $E$ is a maximal Cohen–Macaulay sheaf on $D$, 
we obtain the exact sequence
\begin{align*}
    0 \to E^{\vee} \to \mathcal{O}_D \to \mathcal{E}xt^1_D(Q, \mathcal{O}_D) \to 0. 
\end{align*}
    Therefore $E^{\vee}$ is of the form $I_{D, Z}$ for a zero-dimensional 
    subscheme $Z \subset D$, 
    hence $E=I_{D, Z}^{\vee}$. 
    As $\chi(\mathcal{O}_D)=2-l$
    and $\chi(E)=1$, the length of $Z$ is $l-1$. 
\end{proof}

  \begin{lemma}\label{lem:h1}
Let $D \subset S$ be as in Lemma~\ref{lem:ideal}.
For any stable sheaf $E$ on $D$ with $(\rank(p_{\ast}E), \chi(E))=(2, 1)$
and a generic $\mathcal{L} \in \mathrm{Pic}^0(D)$, we have 
$H^1(E\otimes \mathcal{L})=0$. 
  \end{lemma}
  \begin{proof}
      We first reduce to the case that $E$ is a line bundle on $D$. 
      We write $E$ as a rank two $L$-twisted Higgs bundle $(F, \theta)$
      where $\theta \colon F \to F\otimes L$. 
      As $E$ is supported on the zero section, the Higgs field $\theta$ is nilpotent. 
      By Lemma~\ref{lem:ideal},
      $\theta$ is generically of rank one, and $E$ is a line bundle on $D$ 
      if and only if $\theta|_{x}$ is rank one for any $x \in C$. 
      Let $Z \subset C$ be the scheme theoretic zero locus 
      of $\theta$, which is a zero-dimensional subscheme. Let $L'=\mathrm{Ker}(L \twoheadrightarrow L|_{Z})$, 
      which is a line bundle on $C$. 
      Let $\phi$ be the induced map 
      \begin{align*}
          \phi \colon S'=\mathrm{Tot}_C(L') \to S. 
      \end{align*}
      It restricts to the map $D' \to D$, where $D' \subset S'$ is the double of $C$
      as in (\ref{double}). 
      As $\theta|_{Z}=0$, it factors through 
      $\theta'\colon F \to F\otimes L'$ which does not vanish at every point in $C$. 
      Therefore $E=\phi_{\ast}E'$ for a line bundle $E'$ on $D'$. 
      Note that 
      $\phi_{\ast}E' \otimes \mathcal{L}=\phi_{\ast}(E'\otimes \phi^{\ast}\mathcal{L})$. 
      Moreover, the pull-back map 
      \begin{align}\label{back:Pic}
          \phi^{\ast} \colon \mathrm{Pic}^0(D) \to \mathrm{Pic}^0(D')
      \end{align}
      is surjective. 
      Indeed, from the exact sequence 
      \[0\to L^{-1} \to \mathcal{O}_D^{\ast} \to \mathcal{O}_C^{\ast} \to 1,\]
      we obtain $\mathrm{Pic}^0(D)=H^1(L^{-1})$ and the above map (\ref{back:Pic})
      is identified with the surjection 
      $H^1(L^{-1}) \twoheadrightarrow H^1(L'^{-1})$ induced by 
      the inclusion $L' \subset L$. 
      Moreover $E'$ is obviously stable on $D'$, and 
      $\deg L'>0$ as otherwise there is no such $E'$ by Lemma~\ref{lem:ideal}. 
      Therefore we can replace $E$ with 
$E'$ and assume that $E$ is a line bundle on $D$. 

As $E$ is a line bundle on $D$, 
there is an exact sequence 
\begin{align*}
    0 \to E|_{C}\otimes L^{-1} \to E \to E|_{C} \to 0. 
\end{align*}
As $\chi(E)=1$, it follows that 
$E|_{C}=\mathcal{O}_C(l')$ where $l=2l'+1$ 
(here the existence of such a line bundle $E$ implies that $l$ 
is odd). 
Then for any $\mathcal{L} \in \mathrm{Pic}^0(D)$, we have the exact sequence
\begin{align*}
    0 \to \mathcal{O}_C(-l'-1) \to E\otimes \mathcal{L} \to \mathcal{O}_C(l') \to 0. 
\end{align*}
Then $H^1(E \otimes \mathcal{L})$ is isomorphic to the cokernel of the induced map
\begin{align*}
  \eta_{\mathcal{L}} \colon H^0(\mathcal{O}_C(l')) \to H^1(\mathcal{O}_C(-l'-1))=H^0(\mathcal{O}_C(l'-1))^{\vee}
\end{align*}
which is a map from $\mathbb{C}^{l'+1}$ to $\mathbb{C}^{l'}$. 
The line bundle $\mathcal{L}$ corresponds to an element 
\begin{align*}
    g \in \mathrm{Pic}^0(D)=H^1(L^{-1})=H^0(\mathcal{O}_C(2l'-1))^{\vee}
\end{align*}
and we have $\eta_{\mathcal{L}}=\eta_{\mathcal{O}_D}+g(-, -)$, where $g(-, -)$ is 
\begin{align*}
    g(-, -) \colon H^0(\mathcal{O}_C(l')) \otimes H^0(\mathcal{O}_C(l'-1)) \stackrel{\cup}{\to}
    H^0(\mathcal{O}_C(2l'-1)) \stackrel{g}{\to} \mathbb{C}. 
\end{align*}
Then it is straightforward to check that 
$\eta_{\mathcal{L}}$ is surjective for a generic $g$, 
therefore $H^1(E\otimes \mathcal{L})=0$ for generic $\mathcal{L} \in \mathrm{Pic}^0(D)$. 
  \end{proof}   

Denote by
\begin{align*}
\mathrm{Hilb}^d(\mathcal{C}/B) \to B
\end{align*}
the relative 
Hilbert scheme of $d$-points on $\mathcal{C}/B$. 
Over a point $b \in B$ corresponding to the spectral curve $\mathcal{C}_b \subset S$, 
the fiber of the above map parametrizes zero-dimensional 
closed subschemes $Z \subset \mathcal{C}_b$
with length $d$. Define 
\begin{align*}
    \mathrm{Hilb}^d(\mathcal{C}/B)^{\circ} \subset \mathrm{Hilb}^d(\mathcal{C}/B)
\end{align*}
to be 
the open subset consisting of $Z \subset \mathcal{C}_b$
such that $I_{\mathcal{C}_b, Z}^{\vee}$ is stable. 
Consider the Abel-Jacobi map 
\begin{align*}
\mathrm{AJ} \colon 
    \mathrm{Hilb}^{l-1}(\mathcal{C}/B)^{\circ} \to M^L(2, 1)
\end{align*}
which sends $Z \subset \mathcal{C}_b$ to $I_{\mathcal{C}_b, Z}^{\vee}$. 
\begin{lemma}\label{lem:AJM}
For each point $x \in M^L(2, 1)$ over $b \in B$, 
there is an étale map $(B', b') \to (B, b)$ and 
a line bundle $\mathcal{L} \in \mathrm{Pic}^0(\mathcal{C}')$
on $\mathcal{C}'=\mathcal{C}\times_B B'$
such that the twisted Abel-Jacobi map 
\begin{align}\label{map:AJM}
    \mathrm{AJ}_{\mathcal{L}}(-):=(-)\otimes \mathcal{L} \circ \mathrm{AJ}  \colon 
    \mathrm{Hilb}^{l-1}(\mathcal{C}/B)^{\circ}\times_B B' \to 
    M^L(2, 1)\times_B B'
\end{align}
is étale and surjective at a neighborhood of $(x, b')$. 
\end{lemma}
\begin{proof}
We treat the case when $\mathcal{C}_b$ is a double curve. 
The case when $\mathcal{C}_b$ is a reduced divisor follows from the 
same arguments, using~\cite[Proposition~2.5]{MRF3} instead of Lemmas~\ref{lem:ideal} and \ref{lem:h1} in the argument below. 
We may assume that $b=0 \in B$ so that $\mathcal{C}_b=D$ as in Lemma~\ref{lem:ideal}.

Let $x \in M^L(2, 1)$ be a closed point corresponding to a stable sheaf $E$ on $D$. 
By Lemma~\ref{lem:h1},
there is $\mathcal{L}_0 \in \mathrm{Pic}^0(D)$ such that $H^1(E\otimes \mathcal{L}_0^{-1})=0$. 
Since \[\mathcal{P}ic^0(\mathcal{C}/B) \to B\] is smooth, 
there is an étale map $(B', b') \to (B, 0)$ such that 
the line bundle $\mathcal{L}_0$ can be extended to a line bundle $\mathcal{L}$ on 
$\mathcal{C}\times_B B'$. 
Then the $\mathcal{L}$-twisted Abel-Jacobi
map (\ref{map:AJM}) is surjective by Lemma~\ref{lem:ideal}. 
From the proof of Lemma~\ref{lem:ideal}, 
the fiber of (\ref{map:AJM}) at $(x, b')$ is given 
\begin{align}\label{fiber:AJL}
\mathrm{AJ}_{\mathcal{L}}^{-1}(x, b')=\mathbb{P}(H^0(E \otimes \mathcal{\mathcal{L}}_0^{-1})).
\end{align}
As $\chi(E)=1$ and $H^{>0}(E \otimes \mathcal{L}_0^{-1})=0$, 
we have $H^0(E\otimes \mathcal{L}_0^{-1})=\mathbb{C}$, hence 
the fiber (\ref{fiber:AJL}) is a reduced one point. 
Hence the map (\ref{map:AJM}) is étale at $(x, b')$. 
\end{proof}

Let $\mathcal{M}=\mathcal{M}^L(2, 2-l)$, and 
$\mathcal{E} \in \Coh(\mathcal{C}\times_B \mathcal{M})$ be the universal 
sheaf. We set 
\begin{align*}
    \mathcal{E}^{[d]} :=\bigotimes_{i=1}^{d}p_{i}^{\ast}\mathcal{E} \in \Coh(\mathcal{C}^{\times_B d}\times_B \mathcal{M}), 
\end{align*}
where $p_{i}$ is the projection 
from $\mathcal{C}^{\times_B d}\times_B \mathcal{M}$
onto the $i$-th factor $\mathcal{C}\times_B \mathcal{M}$. 
Before we prove Theorem~\ref{thm:CMext}, we discuss the following preliminary result.
\begin{lemma}\label{lem:EinT}
For $d=l-1$, we have 
\begin{align*}
    \mathcal{E}^{[l-1]} \in (\mathbb{T}^L(2, 2-l)_{l-1})_{\mathcal{C}^{\times_B (l-1)}}. 
\end{align*}
\end{lemma}
\begin{proof}
    We may assume that $l$ is even, and write $l=2l'$ for $l' \in \mathbb{Z}$. 
    Let $y \in \mathcal{M}$ be a closed point which is strictly semistable, so it 
    corresponds to either $V \otimes \mathcal{O}_{C'}(-l')$ for $\dim V=2$ 
    and a section $C'$ of $S \to C$, or 
    $\mathcal{O}_{C'}(-l') \oplus \mathcal{O}_{C''}(-l')$
    for sections $C', C''$ of $S \to C$ such that $C' \neq C''$. 
    Denote by $p_\mathcal{M}\colon \mathcal{C}^{\times_B (l-1)}\times_B \mathcal{M}\to\mathcal{M}$ the natural projection.
We show that 
for any $F \in \mathrm{Perf}(\mathcal{C}^{\times_B (l-1)})$, the object 
\begin{align}\label{obj:x}
    Rp_{\mathcal{M}\ast}(F\boxtimes \mathcal{E}^{[l-1]})|_{y}
\end{align}
satisfies the weight condition 
in Lemma~\ref{lem:qbps}. 
Then the desired claim follows from Lemma~\ref{lem:Cix}. 
We only treat the case when $y$ corresponds to $V\otimes \mathcal{O}_{C'}(-l')$, 
so that $\mathrm{Aut}(y)=GL(V)$. The other (easier) case can be treated similarly. 

    We may assume that $C'=C$ is the zero section. 
    Note that the Ext-quiver associated with $V\otimes\mathcal{O}_{C}(-l')$ 
    has one vertex and $(2l'+1)$-loops. Therefore the stack $\mathcal{M}$ 
    and its good moduli space $\mathcal{M} \to M$
    is isomorphic (étale locally on $M$ at $y$) to 
    \begin{align*}\mathfrak{gl}(V)^{\oplus (2l'+1)}/GL(V) \to 
    \mathfrak{gl}(V)^{\oplus (2l'+1)}\ssslash GL(V).
    \end{align*}
    Since 
    $\mathcal{E}|_{\mathcal{C}\times_B \{y\}}=V\otimes \mathcal{O}_C(-l')$, 
    the object (\ref{obj:x})
    is generated by $V^{\otimes (l-1)}$ as a complex of $GL(V)$-representations. 
    Let $T=(\mathbb{C}^{\ast})^2 \subset GL(V)$ be the maximal torus 
    and let $\lambda \colon \mathbb{C}^{\ast} \to T$ be the one parameter subgroup 
    given by $t\mapsto (t^{\lambda_1}, t^{\lambda_2})$. It is enough to show that 
    for any $T$-weight $\chi$ of $V^{\otimes (l-1)}$, we have 
    \begin{align}\label{wt:cond}
        -\frac{1}{2}n_{\lambda} \leq \left\langle 
        \lambda, \chi-\frac{l-1}{2}\det V \right\rangle \leq \frac{1}{2}n_{\lambda}. 
    \end{align}
    In the above, $n_{\lambda}$ is 
    \begin{align*}
        n_{\lambda}:=\Big\langle \lambda, 
        \det\big((\mathfrak{gl}(V)^{\oplus (2l'+1)}\big)^{\lambda>0}) -
        \det\big(\mathfrak{gl}(V)^{\lambda>0}\big)\Big\rangle 
        =2l'\lvert \lambda_1-\lambda_2\rvert. 
    \end{align*}
    Let $M=\mathbb{Z}\beta_1+\mathbb{Z}\beta_2$ be the weight lattice of 
    $T$. Then a weight $\chi$ of $V^{\otimes (l-1)}$ can be written as 
    \begin{align*}\chi=a\beta_1+(l-1-a)\beta_2, \ 0\leq a\leq l-1.
    \end{align*}
    Then we have 
    \begin{align*}
        \left\langle \lambda, \chi-\frac{l-1}{2}\det V  \right\rangle
        =(\lambda_1-\lambda_2)\left(a-\frac{l-1}{2}\right). 
    \end{align*}
    Therefore the inequalities (\ref{wt:cond}) hold.     
\end{proof}

\begin{thm}\label{thm:CMext}
Assumption~\ref{conj:CMext} is satisfied for $C=\mathbb{P}^1$. 
\end{thm}
\begin{proof}
 As in the proof of Theorem~\ref{thm:r=21}, we treat the case of $w=g^{\rm{sp}}=l-1$, $\chi=1-g^{\rm{sp}}=2-l$
 and $l=2l'$ is even.
 Then $\mathcal{M}'=\mathcal{M}^L(2, 1)$ and 
 $\mathcal{M}=\mathcal{M}^L(2, 2-2l')$.
Note that $\mathcal{M}'\to M':=M^L(2, 1)$ is a trivial $\mathbb{C}^{\ast}$-gerbe. 
Since the line bundle $\mathcal{P}^{\sharp}$ in (\ref{line:Preg}) is of 
bi-weight $(0, l-1)$, it 
 descends to a line bundle on $M'\times_B \mathcal{M}$ which 
is also denoted by $\mathcal{P}^{\sharp}$.
The Cohen–Macaulay extension 
\begin{align}\label{CM:ext:P}\mathcal{P} \in \Coh(M'\times_B \mathcal{M})
\end{align}
of $\mathcal{P}^{\sharp}$ is constructed
in~\cite{MLi} using Abel-Jacobi maps
from $\mathrm{Hilb}^d(\mathcal{C}/B)$ for $d\gg 0$, 
which generalizes the construction of Arinkin~\cite{Ardual}. 
Our aim is to show that the sheaf (\ref{CM:ext:P}) 
is an object in $(\mathbb{T}^L(2, 2-l)_{l-1})_{M'}$. 
A key point of our proof is to show that $d=l-1$ works,
hence we can use Lemma~\ref{lem:EinT}, 
by replacing Abel-Jacobi maps with twisted Abel-Jacobi maps. 
An idea of using twisted Abel-Jacobi map also appeared in~\cite{MRVF2}. 

The isospectral Hilbert scheme of Haiman~\cite{Ha} is defined as the reduced scheme of the fiber product
\begin{align*}
\widetilde{\mathrm{Hilb}}^{d}(S):=(\mathrm{Hilb}^d(S)\times_{\mathrm{Sym}^d(S)}S^{d})^{\rm{red}}.
\end{align*}
There are 
projection maps 
\begin{align*}
    \mathrm{Hilb}^d(S) \stackrel{\psi}{\leftarrow} \widetilde{\mathrm{Hilb}}^d(S) 
    \stackrel{\sigma_d}{\to} S^d. 
\end{align*}
Consider the diagram 
\begin{align*}
    \xymatrix{
\mathrm{Hilb}^d(S) &  & \mathcal{C}^{\times_B d} \times_B \mathcal{M} \ar[d]^-{l^d\times \id} \\
\mathrm{Hilb}^d(S)\times \mathcal{M} \ar[u]_-{p_1} & \ar[l]_-{\psi\times \id} \widetilde{\mathrm{Hilb}}^d(S) \times\mathcal{M}
\ar[r]^-{\sigma_d \times \id} & S^{d}\times \mathcal{M},     
    }
\end{align*}
where $l \colon \mathcal{C} \hookrightarrow S \times B$ is the closed immersion and
$p_1$ is the natural projection. 
Set
\begin{align}\label{def:Qd}
    \mathcal{Q}^{[d]}:=((\psi\times\id)_{\ast}L(\sigma_d \times \id)^{\ast}(l^d \times \id)_{\ast}\mathcal{E}^{[d]})^{\rm{sign}}
    \otimes p_1^{\ast}\det(\mathcal{O}_{\mathcal{Z}})^{-1},
\end{align}
where
$(-)^{\mathrm{sign}}$ is the semi-invariant 
part with respect to the character $\mathrm{sign} \colon \mathfrak{S}_d \to \{\pm 1\}$ 
and $\mathcal{Z}$ is the universal subschemes of $S$ over $\mathrm{Hilb}^d(S)$. 
Then $\mathcal{Q}^{[d]}$ is a coherent sheaf on $\mathrm{Hilb}^d(S)\times \mathcal{M}$, see~\cite[Proposition~4.2]{Ardual}. 
As in~\cite[Proposition~3.2.2]{MLi} (which itself relies on~\cite{Ardual})
the sheaf $\mathcal{Q}^{[d]}$ is a Cohen-Macaulay sheaf on 
the closed subscheme 
\begin{align*}\mathrm{Hilb}^d(\mathcal{C}/B)\times_{B} \mathcal{M}
\hookrightarrow \mathrm{Hilb}^d(S) \times \mathcal{M}.
\end{align*}
In~\cite{MLi, Ardual}, it is proved that for $d\gg 0$, the 
sheaf $\mathcal{Q}^{[d]}$ on $\mathrm{Hilb}^d(\mathcal{C}/B)\times_{B} \mathcal{M}$ descends to a sheaf on $M'\times_B \mathcal{M}$ 
via the Abel-Jacobi map, 
which gives a Cohen–Macaulay extension $\mathcal{P}$. 

Instead of $d\gg 0$, we consider $d=l-1$ and twisted Abel-Jacobi maps. 
By Lemma~\ref{lem:AJM}, the twisted Abel-Jacobi maps (\ref{map:AJM}) induce
surjective étale maps. 
Let $\eta \colon B' \to B$ be an étale map and $\mathcal{L}$ a line bundle 
on $\mathcal{C}$ as in Lemma~\ref{lem:AJM}. 
Let $\mathcal{P}_{\mathcal{L}}^{\sharp}$ be the line bundle on $(\mathcal{M}'\times_B \mathcal{M})^{\sharp}$ as in Lemma~\ref{replace}. 
It descends to a line bundle on $(M'\times_B \mathcal{M})^{\sharp}$, which 
is also denoted by $\mathcal{P}_{\mathcal{L}}^{\sharp}$. 
Then the pull-back of $\mathcal{P}_{\mathcal{L}}^{\sharp}$ under the map 
\begin{align*}
   \mathrm{AJ}_{\mathcal{L}} \times \id_{B'} \colon 
    (\mathrm{Hilb}^{l-1}(\mathcal{C}/B)\times_B \mathcal{M})\times_B B' 
    \to (M'\times_B \mathcal{M})\times_B B'
\end{align*}
at the preimage of $(M'\times_B \mathcal{M})^{\sharp}$ coincides with 
$\eta^{\ast}\mathcal{Q}^{[l-1]}$, see~\cite[Proposition~3.2.2]{MLi}. 
Therefore by the uniqueness of Cohen–Macaulay extension and using Lemma~\ref{replace}, we see 
that $\eta^{\ast}\mathcal{Q}^{[l-1]}$ coincides with the pull-back of
the Cohen–Macaulay extension $\mathcal{P}$ up to a tensor product of an element in 
$\mathrm{Pic}^0(\mathcal{M}\times_B B')$, i.e. there is an isomorphism 
\begin{align*}
    \eta^{\ast}\mathcal{Q}^{[l-1]} \cong (\mathrm{AJ}_{\mathcal{L}} \times \id_{B'})^{\ast}
    (\eta^{\ast}\mathcal{P} \boxtimes \mathcal{L}')
\end{align*}
for some $\mathcal{L}' \in \mathrm{Pic}^0(\mathcal{M}\times_B B')$. 
On the other hand, by Lemma~\ref{lem:EinT}, Lemma~\ref{lem:C:funct}, and the construction of 
$\mathcal{Q}^{[l-1]}$ in (\ref{def:Qd}), we have
\begin{align}\label{QinT}
    \eta^{\ast}\mathcal{Q}^{[l-1]} \in (\mathbb{T}^L(2, 2-l)_{l-1})_{\mathrm{Hilb}^{l-1}(\mathcal{C}/B)\times_B B'}. 
\end{align}
Since a tensor product with an element in 
$\mathrm{Pic}^0(\mathcal{M}\times_B B')$ 
preserves the above condition, 
it follows that $(\mathrm{AJ}_{\mathcal{L}} \times \id_{B'})^{\ast}\eta^{\ast}\mathcal{P}$
is an object in the right hand side of (\ref{QinT}). 
Using Lemma~\ref{lem:Cix}, 
we conclude that 
$\mathcal{P}$ lies in 
$(\mathbb{T}^L(2, 2-l)_{l-1})_{M'}$. 
\end{proof}

By Theorem~\ref{thm:r=21} and Theorem~\ref{thm:CMext}, 
we obtain the following corollary:
\begin{cor}\label{cor:CMext}
Conjecture~\ref{conj:T0} is true for $C=\mathbb{P}^1$ and $r=2$. 
\end{cor}

\section{Equivalences via dimension estimates}\label{s6}
In this section, we prove some technical results which generalize
the arguments in~\cite{Ardual} to the non-commutative setting. The results in this 
section were used in the previous section.

\subsection{Base change of semiorthogonal decompositions}\label{subsec:basechange}
The general setting of this section is as follows. Let $C$ be an arbitrary smooth projective curve, let $r\geq 1$, let $\chi, w\in\mathbb{Z}$, and let $L$ be a line bundle of degree $l>2g-2$ and $l>0$.
For simplicity, we write $\mathcal{M}=\mathcal{M}^L(r, \chi)$
and $M=M^L(r, \chi)$. 
Note that $\mathcal{M}$ is smooth and $M$ 
has at worst Gorenstein singularities, see Lemma~\ref{lem:gorenstein}. 
Recall that there are morphisms: 
\begin{align*}
    h_{\mathcal{M}} \colon \mathcal{M} \stackrel{\pi}{\to} M \stackrel{h_M}{\to} B. 
\end{align*}
Let $T$ be a quasi-projective variety with a
morphism 
\[g\colon T \to B.\] 
By base change, there is a diagram 
of classical stacks
\begin{align}\label{dia:MMT}
    \xymatrix{
\mathcal{M}_T  \ar@/^18pt/[rr]^-{h_{\mathcal{M}_T}}\ar[r]^-{\pi_T} \ar[d]_-{g_{\mathcal{M}}} & M_T \ar[r]^-{h_T} \ar[d]_-{g_M} & T \ar[d]_-{g} \\
\mathcal{M}  \ar@/_18pt/[rr]_-{h_{\mathcal{M}}}\ar[r]^-{\pi} & M \ar[r]^-{h_M} & B.     
    }
\end{align}
Below we assume that $\mathcal{M}_T$ is equivalent 
to the derived fiber product. For example, this 
holds if $g$ is flat or $h_{\mathcal{M}}$ is flat (the latter 
is known to hold for $r=2$ by~\cite[Proposition~2.2.6]{MLi}).

Let $\mathcal{C} \subset D^b(\mathcal{M})$ be a pre-triangulated dg-subcategory. 
The category $\mathcal{C}$ is called \textit{$B$-linear} if it is closed under tensoring with objects in $h_\mathcal{M}^*(\mathrm{Perf}(B))$. 
Following~\cite{MR2801403}, we define the subcategories
\begin{align*}
    \mathcal{C}_{T, \rm{perf}} \subset \mathrm{Perf}(\mathcal{M}_T), \ 
    \mathcal{C}_{T, \rm{qcoh}} \subset D_{\rm{qcoh}}(\mathcal{M}_T),
\end{align*}
where $\mathcal{C}_{T, \rm{perf}}$ is the smallest pre-triangulated 
subcategory which contains $\mathrm{Perf}(T) \boxtimes \mathcal{C}$
and is closed under taking direct summands, and 
$\mathcal{C}_{T, \rm{qcoh}}$ is the smallest pre-triangulated 
subcategory which contains $\mathrm{Perf}(T) \boxtimes \mathcal{C}$
and is closed under taking arbitrary direct sums and direct summands. 
We also define 
\begin{align}\label{def:minuscategory}
    \mathcal{C}_{T}^{-} :=\mathcal{C}_{T, \rm{qcoh}} \cap D^{-}(\mathcal{M}_T), \ 
    \mathcal{C}_T:=\mathcal{C}_{T}^- \cap D^b(\mathcal{M}_T). 
\end{align}
When $T=B$ and $g=\id$, we omit $T$ and write 
$\mathcal{C}_{B, \rm{qcoh}}=\mathcal{C}_{\rm{qcoh}}$ etc. 
\begin{lemma}\label{lem:sdo}
Let $D^b(\mathcal{M})=\langle \mathcal{C}_i \,|\, i \in I \rangle$
be a $B$-linear semiorthogonal decomposition. 
Then there are induced semiorthogonal decomposition
\begin{align*}
    D_{\rm{qcoh}}(\mathcal{M}_T)=\langle 
    (\mathcal{C}_i)_{T, \rm{qcoh}} \,|\, i \in I\rangle, \ 
      \mathrm{Perf}(\mathcal{M}_T)=\langle 
    (\mathcal{C}_i)_{T, \rm{perf}} \,|\, i \in I\rangle. 
\end{align*}
Moreover, the first semiorthogonal decomposition 
restricts to the semiorthogonal decompositions
\begin{align}\label{SOD:D-}
    D^{-}(\mathcal{M}_T)=\langle 
    (\mathcal{C}_i)_T^{-} \,|\, i \in I\rangle. 
\end{align}
   
\end{lemma}
\begin{proof}
    The arguments of~\cite[Proposition~4.2, 4.3, 5.1]{MR2801403} apply. 
\end{proof}

For $\mathcal{P} \in D^-(\mathcal{M}_T)$ and 
a closed point $x \in T$, set 
\begin{align*}
    \mathcal{P}_x :=Rg_{\mathcal{M}\ast}(\mathcal{P}\otimes^L Lh_{\mathcal{M}_T}^{\ast}\mathcal{O}_x)
    \in D^-(\mathcal{M}). 
\end{align*}

\begin{lemma}\label{lem:Cix}
In the situation of Lemma~\ref{lem:sdo}, 
suppose that there exists $i\in I$ such that $\mathcal{P}_x \in \mathcal{C}_i$
for all $x \in T$. 
Then $\mathcal{P} \in (\mathcal{C}_i)_T^-$. 
\end{lemma}
\begin{proof}
    Let $\mathcal{P}_j \in D^{-}(\mathcal{M}_T)$ be the semiorthogonal 
    component of $\mathcal{P}$ with respect to 
    the semiorthogonal decomposition (\ref{SOD:D-}). 
    The assumption implies that
    $(\mathcal{P}_j)_x=0$ for any $x \in T$ and $j\neq i$. 
    It follows that 
    for any $b \in B$
    and a point $x \in T_b:=g^{-1}(b)$ we have
    $\mathcal{P}_{j}|_{h_{\mathcal{M}}^{-1}(b) \times \{x\}}=0$. 
      Therefore $\mathcal{P}_j=0$ for $j\neq i$. 
\end{proof}

The next lemma follows from 
the definition of $\mathcal{C}_{T, \rm{qcoh}}$, 
$\mathcal{C}_{T}^-$ and $\mathcal{C}_T$. 
\begin{lemma}\label{lem:C:funct}
    For a morphism $\alpha \colon T_1 \to T_2$ over $B$, 
the pull-back $L\alpha^{\ast}$ induces 
functors 
\begin{align*}
L\alpha^{\ast} \colon \mathcal{C}_{T_2, \rm{qcoh}} \to 
\mathcal{C}_{T_1, \rm{qcoh}}, \ 
L\alpha^{\ast} \colon \mathcal{C}_{T_2}^{-} \to \mathcal{C}_{T_1}^{-}. 
\end{align*}
Moreover, if we have $R\alpha_{\ast}\mathrm{Perf}(T_1) \subset 
\mathrm{Perf}(T_2)$, for example if $\alpha$ is proper
and $T_2$ is smooth, the push-forward $R\alpha_{\ast}$ induces functors 
\begin{align*}
    R \alpha_{\ast} \colon \mathcal{C}_{T_1, \rm{qcoh}} \to 
    \mathcal{C}_{T_2, \rm{qcoh}}, \ 
    R \alpha_{\ast} \colon \mathcal{C}_{T_1}^- \to \mathcal{C}_{T_2}^-, \ R \alpha_{\ast} \colon 
    \mathcal{C}_{T_1} \to \mathcal{C}_{T_2}. 
\end{align*}
\end{lemma}

\subsection{Equivalences with sheaves of non-commutative algebras}
Below we assume that $g \colon T \to B$ is 
a flat and lci morphism. Note that then both of $M_T$ and 
$\mathcal{M}_T$ have at worst Gorenstein singularities, 
though $\mathcal{M}_T$ is not necessary smooth. 

Suppose that $(r, \chi, w)$ satisfies the BPS condition. 
We consider the vector bundle $\mathbb{V}\to \mathcal{M}$ and 
the sheaf of algebra $\mathscr{A}$ on $M$ as in Proposition~\ref{prop:TA}. 
Below we write $\mathbb{T}=\mathbb{T}(r, \chi)_w$
for simplicity. 
\begin{lemma}\label{lem:AT}
Let $\mathbb{V}_T:=g_{\mathcal{M}}^{\ast}\mathbb{V}$ and set 
\begin{align*}
    \mathscr{A}_T :=\pi_{T\ast}\mathcal{E}nd(\mathbb{V}_T). 
\end{align*}
Then $\mathscr{A}_T=g_M^{\ast}\mathscr{A}$ and it is a maximal Cohen-Macaulay $\mathcal{O}_{M_T}$-module. 
Moreover, there is an equivalence
\begin{align}\label{upA}
   \Upsilon_{\mathscr{A}_T} := 
   \pi_{T\ast}\mathcal{H}om(\mathbb{V}_T, -)\colon 
    \mathbb{T}_{T, \rm{qcoh}}\stackrel{\sim}{\to} 
    D_{\rm{qcoh}}(\mathscr{A}_T), 
\end{align}
which restricts to the equivalences 
\begin{align*}\Upsilon_{\mathscr{A}_T} \colon \mathbb{T}_T^{-} \stackrel{\sim}{\to} D^{-}(\mathscr{A}_T), \ 
\Upsilon_{\mathscr{A}_T} \colon \mathbb{T}_T \stackrel{\sim}{\to} D^{b}(\mathscr{A}_T).
\end{align*}
  \end{lemma}
\begin{proof}
The first statement follows from base change and the assumption that 
$T \to B$ is flat and lci. 
For the second statement, since 
$\mathbb{V}_T$ is a generator of $\mathbb{T}_T^-$ local over $M_T$, 
the argument of~\cite[Lemma~3.3]{TU} applies
to show the equivalences of $\Upsilon_{\mathscr{A}_T}$
for $\mathbb{T}_{T, \rm{qcoh}}$ and $\mathbb{T}_T^-$. 
However, we need a little more care for the equivalence of $\Upsilon_{\mathrm{A}_T}$
for $\mathbb{T}_T$, as a t-structure on it is 
not a priori given and the argument of~\cite[Lemma~3.3]{TU}
does not apply. 

It is enough to show that 
the equivalence $\mathbb{T}_{T}^- \stackrel{\sim}{\to} 
D^{-}(\mathscr{A}_T)$ preserves the subcategories 
$\mathbb{T}_T$ and $D^b(\mathscr{A}_T)$.
It is obvious that $\Upsilon_{\mathscr{A}_T}$ takes 
$\mathbb{T}_T$ to $D^b(\mathscr{A}_T)$. It is enough 
to show that 
$\Upsilon_{\mathscr{A}_T}^{-1}(-)=(-)\otimes^L_{\mathscr{A}_T} \mathbb{V}_T$
takes $D^b(\mathscr{A}_T)$ to $\mathbb{T}_T$. 
This is a 
local condition for $T$, so we may assume that 
$g \colon T \to B$ factors through 
$T \hookrightarrow B' \to B$ where $B' \to B$ is a smooth 
morphism and $T \hookrightarrow B'$ is a regular 
closed immersion. 
There is a commutative diagram 
\begin{align}\label{com:TA}
    \xymatrix{
\mathbb{T}_T^{-} \ar[r]_-{\sim}^-{\Upsilon_{\mathscr{A}_T}} \ar[d] & D^{-}(\mathscr{A}_T) \ar[d] \\
\mathbb{T}_{B'}^{-} \ar[r]_-{\sim}^-{\Upsilon_{\mathscr{A}_{B'}}} & 
D^{-}(\mathscr{A}_{B'}).     
    }
\end{align}
Here the vertical arrows are push-forward by the 
closed immersion $T\hookrightarrow B'$. 
Since $\mathscr{A}$ is homologically homogeneous by~\cite[Theorem~1.5.1]{SVdB}
and $B' \to B$ is smooth, the sheaf of algebra 
$\mathscr{A}_{B'}$ is also homologically homogeneous. 
Therefore $\Upsilon_{\mathscr{A}_{B'}}^{-1}$
takes $D^b(\mathscr{A}_{B'})$ to $\mathbb{T}_{B'}$. 
Since the vertical arrows in (\ref{com:TA}) are conservative, 
we conclude that $\Upsilon_{\mathscr{A}_T}^{-1}$
takes $D^b(\mathscr{A}_T)$ to $\mathbb{T}_T$. 
\end{proof}

We will use the short hand notation \[\mathbb{T}'=\mathbb{T}(r, \chi)_{-w}.\]
We next discuss the compatibility of $\Upsilon_{\mathscr{A}_T}$ with dualizing functors. 
\begin{lemma}\label{lem:com:dia}
The following diagram commutes
\begin{align*}
    \xymatrix{
\mathbb{T}_T \ar[r]^-{\sim}_-{\Upsilon_{\mathscr{A}_T}}
\ar[d]_-{\mathbb{D}_{\mathcal{M}_T}=R\mathcal{H}om(-, \mathcal{O}_{\mathcal{M}_T})} & D^b(\mathscr{A}_T) 
\ar[d]^-{\mathbb{D}_{\mathscr{A}_T}=R\mathcal{H}om_{\mathscr{A}_T}(-, \mathscr{A}_T)} \\
\mathbb{T}_T'^{\rm{op}} \ar[r]^-{\sim}_-{\Upsilon_{\mathscr{A}_T^{\rm{op}}}} & D^b(\mathscr{A}_T^{\rm{op}})^{\rm{op}},     
    }
\end{align*}  
where 
$\mathscr{A}_T^{\rm{op}}:=\pi_{T\ast}\mathcal{E}nd(\mathbb{V}_T^{\vee})$ and 
$\Upsilon_{\mathscr{A}_T^{\rm{op}}}:=\pi_{T\ast}\mathcal{H}om(\mathbb{V}_T^{\vee}, -)$. 
\end{lemma}
\begin{proof}
For any $E \in \mathbb{T}_T$, we have 
\begin{align*}
&R\mathcal{H}om_{\mathscr{A}_T}(\pi_{T\ast}\mathcal{H}om(\mathbb{V}_T, E), \mathscr{A}_T)  \\
&\cong R\mathcal{H}om_{\mathscr{A}_T}(\pi_{T\ast}\mathcal{H}om(\mathbb{V}_T, E), \pi_{T\ast}\mathcal{H}om(\mathbb{V}_T, \mathbb{V}_T)) \\
&\cong \pi_{T\ast}R\mathcal{H}om_{\mathcal{M}_T}(E, \mathbb{V}_T) \\
&\cong \pi_{T\ast}R \mathcal{H}om_{\mathcal{M}_T}(\mathbb{V}_T^{\vee}, \mathbb{D}_{\mathcal{M}_T}(E)). 
\end{align*} 
Here the second isomorphism holds since 
the equivalence (\ref{upA}) is local on $M_T$. 
\end{proof}

There is a Cartesian diagram 
\begin{align}\label{dia:MMA}
    \xymatrix{
(M_T, \mathscr{A}_T) \ar[r]_-{r_{\mathscr{A}_T}} \ar[d]_-{g_{\mathscr{A}}} \ar@/^18pt/[rr]^-{h_{\mathscr{A}_T}}& M_T \ar[r]_-{h_{M_T}} \ar[d]_-{g_M} & T \ar[d]_-{g} \\
(M, \mathscr{A}) \ar[r]^-{r_{\mathscr{A}}} \ar@/_18pt/[rr]_-{h_{\mathscr{A}}}& M \ar[r]^-{h_M} & B.     
    }
\end{align}
Here $(M, \mathscr{A})$ is the 
ringed space with underlying space $M$ and sheaf of algebras $\mathscr{A}$. 
The morphism $r_{\mathscr{A}}$ is induced by the natural morphism of sheaves of algebras 
$\mathcal{O}_M \to \mathscr{A}$. 
\begin{lemma}\label{lem:duality}
For any $E \in D^b(\mathscr{A}_T)$, there is a natural isomorphism 
\begin{align*}
    Rh_{\mathscr{A}_T \ast}R\mathcal{H}om_{\mathscr{A}_T}(E, \mathscr{A}_T[g^{\rm{sp}}]) \cong 
    R\mathcal{H}om_T(Rh_{\mathscr{A}_T \ast}E, \mathcal{O}_T). 
\end{align*}
    \end{lemma}
\begin{proof}
    We first show that there is an isomorphism in $\Coh(\mathscr{A}_T)$
    \begin{align}\label{isom:AT}
        \mathscr{A}_T \stackrel{\cong}{\to} R\mathcal{H}om_{M_T}(\mathscr{A}_T, \mathcal{O}_{M_T}). 
    \end{align}
Consider the trace map $\mathrm{tr} \colon \mathcal{E}nd(\mathbb{V}_T) \to \mathcal{O}_{\mathcal{M}_T}$. 
By taking its push-forward, we obtain 
a trace map $\pi_{T\ast}\mathrm{tr} \colon \mathscr{A}_T \to \mathcal{O}_{M_T}$.
It determines the morphism 
\begin{align}\label{isom:tr}
    \mathscr{A}_T \to \mathcal{H}om_{M_T}(\mathscr{A}_T, \mathcal{O}_{M_T}), \ 
    \alpha \mapsto (\beta \mapsto \pi_{T\ast}\mathrm{tr}(\beta \cdot \alpha)). 
\end{align}
Since $\mathscr{A}_T$ is maximal Cohen–Macaulay module over $\mathcal{O}_{M_T}$, we have 
\begin{align*}
  R\mathcal{H}om_{M_T}(\mathscr{A}_T, \mathcal{O}_{M_T})=
   \mathcal{H}om_{M_T}(\mathscr{A}_T, \mathcal{O}_{M_T})
\end{align*}
and it is a maximal Cohen–Macaulay module over $\mathcal{O}_{M_T}$. 
The morphism (\ref{isom:tr}) is an isomorphism over the open locus 
where $\mathcal{M}_T \to M_T$ is a $\mathbb{C}^{\ast}$-gerbe
whose complement has codimension bigger than or equal to two
(except the case that $r=2$, $l=1$ and $\chi$ even where  
(\ref{isom:tr}) is directly checked to be an isomorphism as 
$\mathcal{A}_T$ is the endomorphism algebra of some twisted vector bundle). 
Therefore (\ref{isom:tr}) is an isomorphism, 
as a maximal Cohen–Macaulay sheaf is determined by its restriction to the complement of a closed 
subset of codimension bigger than or equal to two. 
Therefore there is an isomorphism (\ref{isom:AT}). 
It follows that
\begin{align}\notag
    r_{\mathscr{A}_T \ast}R\mathcal{H}om_{\mathscr{A}_T}(E, \mathscr{A}_T) & \cong 
   \label{isom:rA} r_{\mathscr{A}_T \ast}R\mathcal{H}om_{\mathscr{A}_T}(E, R\mathcal{H}om_{M_T}(\mathscr{A}_T, \mathcal{O}_{M_T})) \\
    &\cong R\mathcal{H}om_{M_T}(r_{\mathscr{A}_T \ast}E, \mathcal{O}_{M_T}). 
\end{align}
Then the lemma holds by applying Grothendieck duality 
for the map $M_T \to T$
for the right hand side of the above isomorphism 
and noticing that $\omega_{M_T/T}\cong \mathcal{O}_{M_T}$.     
\end{proof}

\subsection{Fourier-Mukai functors to BPS categories}\label{subsec:fmBPS}
Given an object $\mathcal{P} \in \mathbb{T}_T$, the diagram (\ref{dia:MMT}) induces the Fourier-Mukai functor 
\begin{align}\label{funct:PhiP}
    \Phi_{\mathcal{P}} :=Rg_{\mathcal{M}\ast}(\mathcal{P} \otimes^L Lh_{\mathcal{M}_T}^{\ast}(-))
    \colon D_{\rm{qcoh}}(T) \to \mathbb{T}_{\rm{qcoh}}. 
\end{align}
Let $\mathcal{P}_{\mathscr{A}}:=\Upsilon_{\mathscr{A}_T}(\mathcal{P})$ 
in $D^b(\mathscr{A}_T)$. 
The diagram (\ref{dia:MMA})
induces the functor 
\begin{align*}
    \Phi_{\mathcal{P}_{\mathscr{A}}} :=Rg_{\mathscr{A}\ast}(\mathcal{P}_{\mathscr{A}} \otimes^L_{\mathcal{O}_{M_T}} Lh_{M_T}^{\ast}(-))
    \colon D_{\rm{qcoh}}(T) \to D_{\rm{qcoh}}(\mathscr{A}). 
\end{align*}
In this subsection, we describe its left adjoint 
functor. 

\begin{lemma}\label{lem:funct:com}
The following diagram commutes
\begin{align}\label{dia:qcoh}
    \xymatrix{
D_{\rm{qcoh}}(T) \ar[r]^-{\Phi_{\mathcal{P}}} \ar@{=}[d] & \mathbb{T}_{\rm{qcoh}} \ar[d]_-{\sim}^-{\Upsilon_A} \\
D_{\rm{qcoh}}(T) \ar[r]_-{\Phi_{\mathcal{P}_{\mathscr{A}}}} & D_{\rm{qcoh}}(\mathscr{A}).     
    }
\end{align}
\end{lemma}
\begin{proof}
For $E \in D_{\rm{qcoh}}(T)$, we have 
\begin{align*}
    \pi_{\ast}\mathcal{H}om(\mathbb{V}, Rg_{\mathcal{M}\ast}(Lh_{\mathcal{M}_T}^{\ast}E \otimes^L \mathcal{P}))
    &\cong \pi_{\ast}Rg_{\mathcal{M}\ast}\mathcal{H}om(\mathbb{V}_T, Lh_{\mathcal{M}_T}^{\ast}E \otimes^L \mathcal{P}) \\
     &\cong Rg_{M\ast}\pi_{T\ast}\mathcal{H}om(\mathbb{V}_T, \pi_T^{\ast}Lh_{M_T}^{\ast}E \otimes^L \mathcal{P}) \\
    & \cong Rg_{M\ast}(\pi_{T\ast}\mathcal{H}om(\mathbb{V}_T, \mathcal{P})\otimes^L Lh_{M_T}^{\ast}E) \\
    &\cong Rg_{\mathscr{A}\ast}(\mathcal{P}_{\mathscr{A}} \otimes^L Lh_{M_T}^{\ast}E). 
\end{align*}
Therefore the diagram (\ref{dia:qcoh}) commutes. 
\end{proof}

\begin{lemma}\label{lem:lad}
The functor $\Phi_{\mathcal{P}_{\mathscr{A}}}$ admits a left adjoint 
given by 
\begin{align*}
    \Phi^{L}_{\mathcal{P}_{\mathscr{A}}}=Rh_{M_T \ast}(g_{\mathscr{A}}^{\ast}(-) \otimes^L \mathcal{P}_{\mathscr{A}}^{\vee})[g^{\rm{sp}}],
\end{align*}
where $\mathcal{P}_{\mathscr{A}}^{\vee}:=R\mathcal{H}om_{\mathscr{A}_T}(\mathcal{P}_{\mathscr{A}}, \mathscr{A}_T)$. 
\end{lemma}
\begin{proof}
    The lemma follows from a standard argument using the relative Serre
    duality for $h_{\mathscr{A}_T}$ established in Lemma~\ref{lem:duality}. 
    Indeed the functor
    \begin{align*}
        Rh_{\mathscr{A}_T \ast}[g^{\rm{sp}}] \colon 
        D_{\rm{qcoh}}(\mathscr{A}_T) \to D_{\rm{qcoh}}(T)
    \end{align*}
    give a left adjoint of $Lh_{\mathcal{M}_T}^{\ast}$
    by Lemma~\ref{lem:duality}
    and noticing that objects in $D_{\rm{qcoh}}(\mathscr{A}_T)$, 
    $D_{\rm{qcoh}}(T)$ are given by colimits of 
    objects from $\mathrm{Perf}(\mathscr{A}_T)$, 
    $\mathrm{Perf}(T)$ respectively, and 
    $Rh_{\mathscr{A}_{T\ast}}$ preserves 
    colimits. 
\end{proof}
Let $\Phi_{\mathcal{P}}^L \colon \mathbb{T}_{\rm{qcoh}} \to D_{\rm{qcoh}}(T)$ be given by 
\begin{align}\label{PhiPL}
    \Phi_{\mathcal{P}}^L(-)=Rh_{\mathcal{M}_T \ast}(\mathcal{P}^{\vee}\otimes^L g_{\mathcal{M}}^{\ast}(-))[g^{\rm{sp}}], 
\end{align}
where $\mathcal{P}^{\vee}=R\mathcal{H}om(\mathcal{P}, \mathcal{O}_{\mathcal{M}_T})$. 
\begin{lemma}\label{lem:PhiL}
The functor $\Phi_{\mathcal{P}}^L$ is a left adjoint 
functor of $\Phi_{\mathcal{P}}$ in (\ref{funct:PhiP}). 
\end{lemma}
\begin{proof}
Using Lemma~\ref{lem:com:dia}, 
a computation similar to Lemma~\ref{lem:funct:com}
shows that the following diagram commutes
\begin{align}\label{dia:qcoh2}
    \xymatrix{
    \mathbb{T}_{\rm{qcoh}} \ar[r]^-{\Phi_{\mathcal{P}}^L} & D_{\rm{qcoh}}(T) \ar@{=}[d] \\
    D_{\rm{qcoh}}(\mathscr{A}) \ar[r]_-{\Phi_{\mathcal{P}_{\mathscr{A}}}^L} \ar[u]_-{\Upsilon_A^{-1}} & D_{\rm{qcoh}}(T). 
    }
\end{align}
Then the lemma follows from Lemma~\ref{lem:lad}. 
\end{proof}

We consider the following diagram 
\begin{align*}
    \xymatrix{
& (M_{T\times_B T}, \mathscr{A}_{T\times_B T}) \ar[ldd]_-{p_1} 
\ar[rd] & & (M_{T\times_B T}, \mathscr{A}_{T\times_B T}^{\rm{op}}) \ar[ld] \ar[rdd]_-{p_2} & \\
 & & M_{T\times_B T} \ar[d]_-{h_{T\times_B T}} & &  \\
(M_T, \mathscr{A}_T)& & T\times_B T.& &  (M_T, \mathscr{A}_T^{\rm{op}})  
    }
\end{align*}
The composition $\Phi_{\mathcal{P}_{\mathscr{A}}}^L \circ \Phi_{\mathcal{P}_{\mathscr{A}}}$ is given 
by the following kernel object
\begin{align*}
   \mathcal{Q}=Rh_{T\times_B T\ast}(p_{1}^{\ast}\mathcal{P}_{\mathscr{A}} \otimes^L p_{2}^{\ast}\mathcal{P}_{\mathscr{A}}^{\vee})[g^{\rm{sp}}]
   \in D^{-}(T\times_B T). 
\end{align*}
By the commutative diagrams (\ref{dia:qcoh}), (\ref{dia:qcoh2}), we also have that
\begin{align}\label{isom:Q}
    \mathcal{Q}=Rp_{13 \ast}(p_{12}^{\ast}\mathcal{P} \otimes^L p_{23}^{\ast}\mathcal{P}^{\vee})[g^{\rm{sp}}] \in D^{-}(T\times_B T). 
\end{align}
Here $p_{ij}$ are the projections
\begin{align}\label{dia:Mpij}
\xymatrix{
& T\times_B \mathcal{M}\times_B T \ar[ld]_-{p_{12}} \ar[d]_-{p_{13}} \ar[rd]^-{p_{23}} & \\
T\times_B \mathcal{M} & T\times_B T & \mathcal{M}\times_B T. 
}    
\end{align}

\subsection{Criterion for fully-faithfullness}
We continue with the notation from the previous subsection.
In this subsection, we give a criterion for the fully-faithfulness of
the functor \eqref{funct:PhiP}  
by generalizing the arguments in~\cite{Ardual}. 
We first discuss some preliminary lemmas. 
Recall the notation in the diagram (\ref{dia:Mpij}). 
\begin{lemma}\label{assum}
For $E_1 \in \mathbb{T}_T$ and $E_2 \in \mathbb{T}_T'$, there is an isomorphism 
\begin{align}\label{isom:assum}
    \pi_{T\times_B T \ast}\mathbb{D}_{\mathcal{M}_{T\times_B T}}(p_{12}^{\ast}E_1\otimes^L p_{23}^{\ast}E_2)
    \cong  
    \mathbb{D}_{M_{T\times_B T}}(\pi_{T\times_B T \ast}(p_{12}^{\ast}E_1\otimes^L p_{23}^{\ast}E_2)). 
\end{align}
    \end{lemma}
    \begin{proof}
        The isomorphism (\ref{isom:assum}) 
        holds when $E_2$ is perfect
        by Lemma~\ref{assum:perf} below. 
        In general, 
        we write $E_2=\colim_{i\in I} P_i$
        for $P_i \in \mathbb{T}_T'$ which is 
        perfect. Then we are reduced to showing that there is an isomorphism:
        \begin{align*}
            \pi_{T\times_B T \ast}\lim_{i\in I}
            \mathbb{D}_{\mathcal{M}_{T\times_B T}}(p_{12}^{\ast}E_1 \otimes^L p_{23}^{\ast}P_i)
            \cong \lim _{i\in I} \mathbb{D}_{M_{T\times_B T}}\pi_{T\times_B T \ast}(p_{12}^{\ast}E_1 \otimes^L p_{23}^{\ast}P_i)
        \end{align*}
        The above isomorphism follows from 
        the isomorphism (\ref{isom:assum}) for 
        $E_2=P_i$ and using that the limit commutes with 
        $\pi_{T\times_B T \ast}$. 
    \end{proof}

\begin{lemma}\label{assum:perf}
For $E_1, E_2 \in \mathbb{T}_T$, suppose that $E_2$ is perfect. 
Then there is a natural isomorphism 
\begin{align}\notag
    \pi_{T\ast}R\mathcal{H}om(E_1, E_2) \stackrel{\cong}{\to} R\mathcal{H}om(\pi_{T\ast}R\mathcal{H}om(E_2, E_1), \mathcal{O}_{M_T}). 
\end{align}
\end{lemma}
\begin{proof}
As $E_2$ is perfect, there is a trace map 
$\mathrm{tr} \colon R\mathcal{H}om(E_2, E_2) \to \mathcal{O}_{\mathcal{M}_T}$, 
hence we obtain a natural map 
\begin{align}\label{isom:natisom}
    \pi_{T\ast}\mathcal{H}om(E_2, E_2) \to \mathcal{O}_{M_T}. 
\end{align}
    It induces a map 
    \begin{align*}
      \pi_{T\ast}R\mathcal{H}om(E_1, E_2) \to R\mathcal{H}om(\pi_{T\ast}R\mathcal{H}om(E_2, E_1), \mathcal{O}_{M_T}).
    \end{align*}
    The above map is an isomorphism for $E_2=\mathbb{V}_T$ 
    by the isomorphism (\ref{isom:rA}) and the equivalence (\ref{upA})
    which holds locally on $M_T$. 
    In general, \'etale locally on $M_T$ the object $E_2$ is a direct summand of a bounded 
    complex of direct sums of $\mathbb{V}_T$. 
    Hence (\ref{isom:natisom}) is an isomorphism locally on $M_T$, hence 
    it is an isomorphism. 
\end{proof}

\begin{lemma}\label{lem:dualQ}
There is an isomorphism  
\begin{align*}
\mathbb{D}_{T\times_B T}(\mathcal{Q}) \cong 
Rh_{\mathcal{M}_{T\times_B T \ast}}\mathbb{D}_{\mathcal{M}_{T\times_B T}}(p_{12}^{\ast}\mathcal{P}\otimes^L p_{23}^{\ast}\mathcal{P}^{\vee}). 
\end{align*}
\end{lemma}
\begin{proof}
Note that $\mathcal{Q}$ is given by (\ref{isom:Q}). 
    The lemma holds by applying $Rh_{M_{T\times_B T}\ast}$ to the 
    isomorphism in Lemma~\ref{assum}
    for $E_1=\mathcal{P}$, 
    $E_2=\mathcal{P}^{\vee}$, and 
    using the Grothendieck duality 
    for $h_{M_{T\times_B T}}$. 
\end{proof}

\begin{lemma}\label{lem:MCM}
Let $\mathcal{P}$ be a maximal Cohen–Macaulay sheaf on $\mathcal{M}_T$ so that 
$\mathcal{P}^{\vee}$ is also a maximal Cohen–Macaulay sheaf. 
Assume that $\mathcal{P}$ and $\mathcal{P}^{\vee}$ are flat over $\mathcal{M}$. 
Then we have 
\begin{align}\label{TBT}
    p_{12}^{\ast}\mathcal{P}\otimes^L p_{23}^{\ast}\mathcal{P}^{\vee}=
       p_{12}^{\ast}\mathcal{P}\otimes p_{23}^{\ast}\mathcal{P}^{\vee}
\end{align}
and $p_{12}^{\ast}\mathcal{P}\otimes p_{23}^{\ast}\mathcal{P}^{\vee}$ is a maximal Cohen–Macaulay sheaf on $\mathcal{M}_{T\times_B T}$. 
\end{lemma}
\begin{proof}
    For a closed point $x \in \mathcal{M}$ over $b \in B$, 
    let $T_b=g^{-1}(b)$ and 
    denote by $i_x$, $j_1$ and $j_2$ the closed immersions 
    \begin{align*}
        i_x \colon T_b \times \{x\} \times T_b \hookrightarrow 
        \mathcal{M}_{T\times_B T}, \ 
        j_1 \colon T_b \times \{x\} \hookrightarrow 
        T\times_B T, \ 
        j_2 \colon \{x\} \times T_b \hookrightarrow T\times_B T. 
    \end{align*}
    Then we have 
    \begin{align*}
        Li_x^{\ast}(p_{12}^{\ast}\mathcal{P}\otimes^L p_{23}^{\ast}\mathcal{P}^{\vee})
        \cong Lj_1^{\ast}\mathcal{P}\boxtimes Lj_2^{\ast}\mathcal{P}^{\vee} 
    \end{align*}
    which is a coherent sheaf by the flatness assumption of $\mathcal{P}$ and $\mathcal{P}^{\vee}$ over 
    $\mathcal{M}$. As the above isomorphism holds for any closed point $x \in \mathcal{M}$, 
    we have the isomorphism (\ref{TBT}). 

    We have the isomorphism 
    \begin{align*}
        Lj_1^{\ast}R\mathcal{H}om(\mathcal{P}, \mathcal{O}_{\mathcal{M}_T})
        \cong R\mathcal{H}om(Lj_1^{\ast}\mathcal{P}, \mathcal{O}_{T_b})
    \end{align*}
    and the left hand side is a sheaf by the assumption that $\mathcal{P}^{\vee}$ is flat 
    over $\mathcal{M}$. 
    Therefore $Lj_1^{\ast}\mathcal{P}$ is a maximal Cohen–Macaulay sheaf on $T_b$, 
    and similarly $Lj_2^{\ast}\mathcal{P}^{\vee}$ is also a maximal Cohen–Macaulay sheaf. 
    Therefore, the push-forward of the sheaf (\ref{TBT}) to $(T\times_B T) \times \mathcal{M}$ 
    is a flat family of Cohen-Macaulay sheaf over $\mathcal{M}$, hence 
    it is a Cohen-Macaulay sheaf by Lemma~\ref{lem:MM} below. 
\end{proof}
We have used the following lemma from~\cite{Ardual}. 
\begin{lemma}\label{lem:MM}\emph{(\cite[Lemma~2.1]{Ardual})}
Let $X$ and $Y$ be schemes of pure dimension, and suppose that $Y$ is 
Cohen-Macaulay. Suppose that a coherent sheaf $F$ on $X \times Y$
is flat over $Y$ and that for every point $y \in Y$, the restriction $F|_{X\times \{y\}}$ is 
Cohen-Macaulay of some fixed codimension $d$. Then $F$ is Cohen-Macaulay of codimension $d$. 
\end{lemma}

Recall the object $\mathcal{Q}$ in (\ref{isom:Q}).  
\begin{lemma}\label{lem:Qvanish}
    In the setting of Lemma~\ref{lem:MCM}, 
    we have 
    \begin{align*}
\mathcal{H}^{>0}(\mathcal{Q})=0, \ 
\mathcal{H}^{>g^{\rm{sp}}}(\mathbb{D}_{T\times_B T}(\mathcal{Q}))=0. 
\end{align*}
\end{lemma}
\begin{proof}
    The first vanishing holds as 
    $p_{12}^{\ast}\mathcal{P}\otimes^L p_{23}^{\ast}\mathcal{P}^{\vee}$ is a 
    coherent sheaf on $\mathcal{M}_{T\times_B T}$, and the fact that 
    for any coherent sheaf $F$ on $\mathcal{M}_{T\times_B T}$ we have 
    \begin{align*}
    R^i h_{\mathcal{M}_{T\times_B T}\ast}F \cong 
    R^i h_{M_{T\times_B T}\ast} \pi_{\mathcal{M}_{T\times_B T}\ast} F=0
    \end{align*}
    for $i>g^{\rm{sp}}$, 
    where the first isomorphism holds as 
    $\pi_{\mathcal{M}_{T\times_B T \ast}}$ is exact. 
    
    The second vanishing holds similarly by noting Lemma~\ref{lem:dualQ} together with 
    \begin{align*}
        \mathbb{D}_{\mathcal{M}_{T\times_B T}}(p_{12}^{\ast}\mathcal{P}\otimes^L p_{23}^{\ast}\mathcal{P}^{\vee}) \in \Coh(\mathcal{M}_{T\times_B T})
    \end{align*}
    by Lemma~\ref{lem:MCM}. 
\end{proof}
Let $\Delta_{T} \subset T\times_B T$ be the diagonal embedding. 
 
\begin{prop}\label{prop:CMQ}
 In the setting of Lemma~\ref{lem:MCM}, 
suppose furthermore that the morphism
$g \colon T \to B$ has relative dimension $g^{\rm{sp}}$ and that
\begin{enumerate}
    \item there is an open subset $B^{\circ} \subset B$ whose complement 
    has codimension at least two such that 
    $\mathcal{Q}|_{T^{\circ}\times_{B^{\circ}} T^{\circ}} \cong 
    \mathcal{O}_{\Delta_{T^{\circ}}}$, and 
    \item the subset 
    \begin{align*}
        \{(x_1, x_2) \in T\times_B T \setminus \Delta_T : \mathcal{Q}|_{(x_1, x_2)} \neq 0\}
        \subset T\times_B T
    \end{align*}
    has codimension bigger than $g^{\rm{sp}}$. 
\end{enumerate}
    Then we have $\mathcal{Q} \cong \mathcal{O}_{\Delta_T}$. 
\end{prop}
\begin{proof}
  By the assumption (2), 
  we have $\mathrm{codim}(\mathrm{Supp}(\mathcal{Q})) \geq g^{\rm{sp}}$. It follows that by   
  Lemma~\ref{lem:Qvanish} and Lemma~\ref{lem:CMcond} below, 
  the sheaf $\mathcal{Q}$ is a Cohen-Macaulay sheaf of codimension $g^{\rm{sp}}$. 
  Then assumption (2) further implies that the support of $\mathcal{Q}$ is 
  $\Delta_T$. 
  Its scheme theoretic support is generically reduced by (1), hence 
  it is a push-forward of a maximal Cohen–Macaulay sheaf on $\Delta_T$ by~\cite[Lemma~8.2]{MRVF}. 
As it is isomorphic to $\Delta_T$ outside a closed subset of codimension 
bigger than or equal to two, and is maximal Cohen–Macaulay, we conclude that 
$\mathcal{Q} \cong \mathcal{O}_{\Delta}$.   
\end{proof}

We have used the following lemma from~\cite{Ardual}:
\begin{lemma}\label{lem:CMcond}\emph{(\cite[Lemma~7.6]{Ardual})}
Let $X$ be a scheme of pure dimension. 
Suppose that $F \in D^b(X)$
satisfies that $\mathrm{codim}(\mathrm{Supp}(F)) \geq d$, 
$\mathcal{H}^{>0}(F)=0$ and $\mathcal{H}^{>d}(\mathbb{D}_X(F))=0$. 
Then $F$ is a Cohen-Macaulay sheaf of codimension $d$. 
\end{lemma}

\subsection{Fourier-Mukai equivalences for BPS categories}
Suppose that $g \colon T\to B$ is proper. 
Then the functors $\Phi_{\mathcal{P}}$, $\Phi_{\mathcal{P}}^L$ in (\ref{funct:PhiP}), (\ref{PhiPL})
restrict to the 
functors 
\begin{align}\label{funct:PhiP2}
\Phi_{\mathcal{P}} \colon D^b(T) \to \mathbb{T}, \ 
\Phi_{\mathcal{P}}^L \colon \mathbb{T} \to D^b(T)
\end{align}
such that $\Phi_{\mathcal{P}}^L$ is the left adjoint 
of $\Phi_{\mathcal{P}}$. 
The following proposition is an analogue of~\cite[Theorem~5.4]{Brequiv}. 
\begin{prop}\label{prop:equivT}
Suppose that the functor $\Phi_{\mathcal{P}}$
in (\ref{funct:PhiP2}) is fully-faithful. Then $\Phi_{\mathcal{P}}$ is an equivalence. 
\end{prop}
\begin{proof}
    Suppose that $\Phi_{\mathcal{P}}$ is fully-faithful. 
    Since $\Phi_{\mathcal{P}}^L$ is its left adjoint, we have the 
    semiorthogonal decomposition of the form 
    \begin{align*}
        \mathbb{T}=\langle \mathrm{Im}(\Phi_{\mathcal{P}}), \mathcal{C} \rangle.
    \end{align*}
    Since the kernel objects of $\Phi_{\mathcal{P}}$ and $\Phi_{\mathcal{P}}^L$ are 
    objects from $T\times_B \mathcal{M}$, the above semiorthogonal decomposition is 
    $B$-linear. Then by Lemma~\ref{lem:sodT} below, 
    we have $\mathcal{C}=0$, hence 
    $\Phi_{\mathcal{P}}$ is an equivalence. 
\end{proof}

\begin{lemma}\label{lem:sodT}
There is no non-trivial $B$-linear 
semiorthogonal decomposition \[\mathbb{T}=\langle \mathcal{C}_1, \mathcal{C}_2 \rangle.\] 
\end{lemma}
\begin{proof}
Let $\mathbb{T}=\langle \mathcal{C}_1, \mathcal{C}_2 \rangle$
be a $B$-linear semiorthogonal decomposition.
We will show that either $\mathcal{C}_1=0$ or $\mathcal{C}_2=0$. 
Let $\mathcal{E}_i \in \mathcal{C}_i$ with $i=1, 2$. 
Because the semiorthogonal decomposition $\mathbb{T}=\langle \mathcal{C}_1, \mathcal{C}_2\rangle$ is $B$-linear, we have that
\[Rh_{\mathcal{M}\ast}\mathcal{H}om(\mathcal{E}_2, \mathcal{E}_1)=0.\] 
Then, by Corollary~\ref{cor:proper0},
we have $Rh_{\mathcal{M}\ast}\mathcal{H}om(\mathcal{E}_1, \mathcal{E}_2)=0$, hence 
we have an orthogonal decomposition 
\[\mathbb{T}=\mathcal{C}_1 \oplus \mathcal{C}_2.\] 

Since the above decomposition is $B$-linear, it restricts 
to the decomposition $D^b(\mathcal{M}^{\rm{sm}})_w=\mathcal{C}_1^{\rm{sm}} \oplus\mathcal{C}_2^{\rm{sm}}$
on the pull-back of $B^{\rm{sm}} \subset B$ corresponding to smooth 
spectral curves. We first show that either $\mathcal{C}_1^{\rm{sm}}=0$ or $\mathcal{C}_2^{\rm{sm}}=0$. 
Note that \[M^{\rm{sm}} \to B^{\rm{sm}}\] is a smooth abelian fibration
with connected fibers, 
and $D^b(\mathcal{M}^{\rm{sm}})_w$ is the derived
category of $\alpha^w$-twisted sheaves of a Brauer class $\alpha$ on $M^{\rm{sm}}$. 
Since $\alpha$ is trivial étale locally on $B$, by taking 
étale cover of $B^{\rm{sm}}$ we may assume that $\alpha^w$ is trivial. 
Then as $\mathcal{O}_{M^{\rm{sm}}}$ is indecomposable, we 
may assume that $\mathcal{O}_{M^{\rm{sm}}}\in \mathcal{C}_1^{\rm{sm}}$. 
Then for any point $x \in M^{\rm{sm}}$, we have $\mathcal{O}_x \in \mathcal{C}_1^{\rm{sm}}$
as it is indecomposable and admits a surjection from $\mathcal{O}_{M^{\rm{sm}}}$. 
Then any indecomposable object in $D^b(M^{\rm{sm}})$ lies in $\mathcal{C}_1^{\rm{sm}}$ as 
it admits a non-zero morphism from some shift of $\mathcal{O}_x$ for some $x \in M^{\rm{sm}}$. 
Therefore we have $\mathcal{C}_2^{\rm{sm}}=0$. 

By the equivalence $\mathbb{T} \stackrel{\sim}{\to} D^b(\mathscr{A})$ in Proposition~\ref{prop:TA}, 
the above decomposition may be written as \[D^b(\mathscr{A})=\mathcal{C}_1 \oplus \mathcal{C}_2.\] 
Let $\mathscr{A}=\mathscr{A}_1 \oplus \mathscr{A}_2$ be the corresponding decomposition. 
Then $\mathscr{A}_2|_{M^{\rm{sm}}}=0$ as $\mathcal{C}_2^{\rm{sm}}=0$, hence $\mathscr{A}_2=0$
since $\mathscr{A}$ is a maximal Cohen–Macaulay sheaf on $M$. Therefore 
$\mathscr{A} \in \mathcal{C}_1$, and similarly $\mathscr{A} \otimes_{\mathcal{O}_M} \mathcal{L} \in \mathcal{C}_1$
for any line bundle $\mathcal{L}$ on $M$. 
For any object $E \in \Coh(\mathscr{A})$, there is a surjection 
of the form \[(\mathscr{A}\otimes_{\mathcal{O}_M}\mathcal{L})^{\oplus k} \twoheadrightarrow E\]
for some integer $k >0$ and a line bundle $\mathcal{L}$ on $M$. 
Therefore $E \in \mathcal{C}_1$ and we conclude that $\mathcal{C}_2=0$. 
    \end{proof}

  \section{The conjectural $\mathrm{SL}/ \mathrm{PGL}$ duality of quasi-BPS categories}
  In this section, we propose a SL/PGL version 
  of duality of quasi-BPS categories inspired by the mirror symmetry for Higgs bundles of Hausel--Thaddeus~\cite{HauTha}, and prove 
  Theorems~\ref{thm:intro2} and \ref{thm:intro3}. 
  
    \subsection{SL-Higgs moduli stacks}
    Let $C$ be a smooth projective curve of genus $g$ and let $L$ be a line bundle on $C$ 
      such that $l=\deg L>2g-2$ or $L=\Omega_C$. 
For each decomposition $r=r_1+\cdots+r_k$, 
set 
\begin{align*}
    \mathrm{SL}(r_{\bullet}):=\Ker\left(\prod_{i=1}^k \GL(r_i) \stackrel{\det}{\to} \mathbb{C}^{\ast}\right), \ 
    (g_i)_{1\leq i\leq k} \stackrel{\det}{\mapsto} \prod_{i=1}^k \det g_i. 
\end{align*}
For a tuple of integers $\chi_{\bullet}=(\chi_1, \ldots, \chi_k)$, the moduli stack of 
$(L, \chi_{\bullet})$-twisted $\mathrm{SL}(r_{\bullet})$-Higgs bundles is given as follows. 
Consider the closed substack 
\begin{align*}
\mathcal{M}^L(r_{\bullet}, \chi_{\bullet})^{\rm{tr}=0}
    \subset \prod_{i=1}^k \mathcal{M}^L(r_i, \chi_i)
\end{align*}
given by the derived zero locus of 
\begin{align*}
  \mathrm{tr} \colon   \prod_{i=1}^k \mathcal{M}^L(r_i, \chi_i) \to H^0(C, L)^{\rm{der}}, \ 
  \{(F_i, \theta_i)\}_{1\leq i\leq k}
\to \sum_{i=1}^k \mathrm{tr}(\theta_i).
\end{align*}
Here $H^0(C, L)^{\rm{der}}$ is defined as follows: 
\begin{align*}
    H^0(C, L)^{\rm{der}}=\Spec \mathrm{Sym}(R\Gamma(C, L)^{\vee})
    =\begin{cases} H^0(C, L), & l>2g-2 \\
    H^0(C, \omega_C) \times \Spec \mathbb{C}[\epsilon], & L=\Omega_C \end{cases}
\end{align*}
where $\deg \epsilon=-1$.
There is a map 
\begin{align*}
    \det \colon \mathcal{M}^L(r_{\bullet}, \chi_{\bullet})^{\rm{tr}=0}
    \to \mathcal{P}ic(C), \ \{(F_i, \theta_i)\}_{1\leq i\leq k}
    \mapsto \bigotimes_{i=1}^k \det F_i. 
\end{align*}
In the above, $\mathcal{P}ic(C)$ is the Picard stack of 
line bundles on $C$. Its good moduli space
\begin{align*}
    \mathcal{P}ic(C) \to \mathrm{Pic}(C)
\end{align*} is a trivial $\mathbb{C}^{\ast}$-gerbe. 

We fix a line bundle $A$ on $C$ of degree $\chi+r(g-1)$. 
The $(L, \chi_{\bullet})$-twisted $\mathrm{SL}(r_{\bullet})$-Higgs moduli stack is 
defined as the Cartesian product 
\begin{align*}
    \xymatrix{
\mathcal{M}_{\mathrm{SL}(r_{\bullet})}^L(\chi_{\bullet})  \ar[r] \ar[d] & \mathcal{M}^L(r_{\bullet}, 
\chi_{\bullet})^{\rm{tr}=0} 
\ar[d]^-{\rm{det}} \\
\Spec \mathbb{C} \ar[r] & \mathcal{P}ic(C),
     }
\end{align*}
where the bottom arrow corresponds to $A$. 
The stack
$\mathcal{M}_{\mathrm{SL}(r_{\bullet})}^L(\chi_{\bullet})$ is smooth for $l>2g-2$ and it is quasi-smooth for $L=\Omega_C$. 
In particular, for $k=1$, we obtain 
the $(L, \chi)$-twisted $\mathrm{SL}(r)$-Higgs moduli stack $\mathcal{M}^L_{\mathrm{SL}(r)}(\chi)$.
\begin{remark}\label{rmk:SLbundle}
    Recall that a $L$-twisted $\mathrm{SL}(r)$-Higgs bundle 
    is a principal $\mathrm{SL}(r)$-bundle 
    $\mathcal{P} \to C$ together with a Higgs field 
    $\theta \in \Gamma(C, \mathrm{ad}(\mathcal{P}) \otimes L)$. It corresponds to a pair $(F, \theta)$ where $F \to C$
    is a vector bundle with $\det (F) \cong \mathcal{O}_C$
    and $\theta \in \Hom(F, F \otimes L)$ with 
    $\mathrm{tr}(F)=0$. In particular $\chi(F)=r(1-g)$, and 
    it corresponds to $(L, \chi=r(1-g))$-twisted 
    $\mathrm{SL}(r)$-Higgs bundle in our convention. 
\end{remark}

Note that the center of $\mathrm{SL}(r_{\bullet})$ is 
\begin{align*}
    Z(\mathrm{SL}(r_{\bullet}))=\Ker\left( (\mathbb{C}^{\ast})^k \to \mathbb{C}^{\ast} \right), 
    (t_i)_{1\leq i\leq k} \mapsto \prod_{i=1}^k t_i^{r_i}, 
\end{align*}
and we have 
\begin{align*}
    \Hom(Z(\mathrm{SL}(r_{\bullet})), \mathbb{C}^{\ast})
    =\mathbb{Z}^{\oplus k}/(r_1, \ldots, r_k)\mathbb{Z}. 
\end{align*}
Note that $\mathcal{M}_{\mathrm{SL}(r_{\bullet})}^L(\chi_{\bullet})$ is a $Z(\mathrm{SL}(r_{\bullet}))$-gerbe.
Thus there is a corresponding decomposition
\begin{align*}
    D^b(\mathcal{M}_{\mathrm{SL}(r_{\bullet})}^L(\chi_{\bullet}))=\bigoplus_{w_{\bullet} \in \mathbb{Z}^{\oplus k}/(r_1, \ldots, r_k)\mathbb{Z}}
    D^b(\mathcal{M}_{\mathrm{SL}(r_{\bullet})}^L(\chi_{\bullet}))_{w_{\bullet}}
\end{align*}
in summands corresponds to the weight $(w_1, \ldots, w_k)$-component with 
respect to the
action of $Z(\mathrm{SL}(r_{\bullet}))$.

\subsection{PGL-Higgs moduli stacks}
For each decomposition $r=r_1+\cdots+r_k$, set
\begin{align*}
    \mathrm{PGL}(r_{\bullet}):=\left(\prod_{i=1}^k \GL(r_i)\right)/\mathbb{C}^{\ast}. 
\end{align*}
There is a natural action 
of $\mathcal{P}ic(C)$ 
on the disjoint union of 
$\mathcal{M}^L(r_{\bullet}, \chi_{\bullet})^{\rm{tr}=0}$ for $\chi_{\bullet} \in \mathbb{Z}^k$
by the tensor product. 
The moduli stack of $L$-twisted $\mathrm{PGL}(r_{\bullet})$-Higgs bundles is given by
\begin{align*}
    \mathcal{M}^L_{\mathrm{PGL}(r_{\bullet})} &:=
    \left( \coprod_{\chi_{\bullet} \in \mathbb{Z}^k}\mathcal{M}^L(r_{\bullet}, \chi_{\bullet})^{\rm{tr}=0} \right)/\mathcal{P}ic(C) \\
    &=\coprod_{\chi_{\bullet} \in \mathbb{Z}^{\oplus k}/(r_1, \ldots, r_k)\mathbb{Z}}
    \mathcal{M}^L_{\mathrm{PGL}(r_{\bullet})}(\chi_{\bullet}),
\end{align*}
where each component is given by 
\begin{align}\label{quot:Pic}
   \mathcal{M}^L_{\mathrm{PGL}(r_{\bullet})}(\chi_{\bullet})
   =\mathcal{M}^L(r_{\bullet}, \chi_{\bullet})^{\rm{tr}=0}/\mathcal{P}ic^0(C). 
   \end{align}
In particular for $k=1$, we obtain the $L$-twisted $\mathrm{PGL}(r)$-Higgs moduli 
stack $\mathcal{M}_{\mathrm{PGL}(r)}^L(\chi)$ for $\chi \in \mathbb{Z}/r\mathbb{Z}$. 

By taking the quotient by $\mathrm{Pic}^0(C)$ in (\ref{quot:Pic}) instead 
of the quotient by $\mathcal{P}ic^0(C)$, we 
obtain the $\mathbb{C}^{\ast}$-gerbe
\begin{align}\label{pgl:gerbe}
     \widetilde{\mathcal{M}}^L_{\mathrm{PGL}(r_{\bullet})}(\chi_{\bullet})
     \to \mathcal{M}^L_{\mathrm{PGL}(r_{\bullet})}(\chi_{\bullet}). 
\end{align}
There is a decomposition of the derived category 
into the $\mathbb{C}^{\ast}$-weight components:
\begin{align}\label{decom:pgl}
  D^b(\widetilde{\mathcal{M}}^L_{\mathrm{PGL}(r_{\bullet})}(\chi_{\bullet}))=
  \bigoplus_{w\in \mathbb{Z}} D^b(\mathcal{M}^L_{\mathrm{PGL}(r_{\bullet})}(\chi_{\bullet}))_w.
\end{align}
Note that the $w=0$ component is equivalent to $D^b(\mathcal{M}^L_{\mathrm{PGL}(r_{\bullet})}(\chi_{\bullet}))$.

Note that the center of $\mathrm{PGL}(r_{\bullet})$ is 
\begin{align*}
    Z(\mathrm{PGL}(r_{\bullet}))=(\mathbb{C}^{\ast})^k/\mathbb{C}^{\ast}
\end{align*}
and we have 
\begin{align*}
    \Hom(Z(\mathrm{PGL}(r_{\bullet})), \mathbb{C}^{\ast})=
    \Ker\left(\mathbb{Z}^{\oplus k} \to \mathbb{Z}   \right), \ 
    (w_1, \ldots, w_k) \mapsto \sum_{i=1}^k w_i. 
\end{align*}
There is a corresponding decomposition 
\begin{align*}
    D^b(\mathcal{M}^L_{\mathrm{PGL}(r_{\bullet})}(\chi_{\bullet}))_w
    =\bigoplus_{w_1+\cdots+w_k=w} 
    D^b(\mathcal{M}^L_{\mathrm{PGL}(r_{\bullet})}(\chi_{\bullet}))_{w_{\bullet}} 
\end{align*}
where each summand corresponds to weight $(w_1, \ldots, w_k)$-component with 
respect to the
action of $Z(\mathrm{PGL}(r_{\bullet}))$. 

\subsection{Semiorthogonal decompositions of SL/PGL-Higgs moduli stacks}\label{subsec:SLPGL}
We next state the SL/PGL versions of the semiorthogonal decomposition from Theorem~\ref{thm:sod}.
Recall the Hitchin base
\[B=B_r:=\bigoplus_{i=1}^r
    H^0(C, L^{\otimes i})\]
    and consider its affine subspace
\[B_{\geq 2}=B_{\geq 2,r}:=\bigoplus_{i=2}^r
    H^0(C, L^{\otimes i}).\]    
\begin{thm}\label{thm:SL}
For each $r_{\bullet} \in \mathbb{Z}^k$, $\chi_{\bullet} \in \mathbb{Z}^k$ and 
$w_{\bullet} \in \mathbb{Z}^k/(r_1, \ldots, r_k)\mathbb{Z}$, 
there is a subcategory
\begin{align*}
    \mathbb{T}^L_{\mathrm{SL}(r_{\bullet})}(\chi_{\bullet})_{w_{\bullet}} \subset D^b(\mathcal{M}^L_{\mathrm{SL}(r_{\bullet})}(\chi_{\bullet}))_{w_{\bullet}}
\end{align*}
such that there is a semiorthogonal decomposition 
\begin{align}\label{sod:SL}
D^b(\mathcal{M}^L_{\mathrm{SL}(r)}(\chi))_w=
\left\langle \mathbb{T}^L_{\mathrm{SL}(r_{\bullet})}(\chi_{\bullet})_{w_{\bullet}} \,\Big|\,
\frac{v_1}{r_1}<\cdots<\frac{v_k}{r_k} \right\rangle. 
\end{align}
The right hand side is after all partitions $r=r_1+\cdots+r_k$, 
$\chi=\chi_1+\cdots+\chi_k$
such that $\chi_i/r_i=\chi/r$ for $1\leq i\leq k$, 
and 
$w_{\bullet} \in \mathbb{Z}^k/(r_1, \ldots, r_k)\mathbb{Z}$ with $w_1+\cdots+w_r=w$ in 
$\mathbb{Z}/r\mathbb{Z}$, and 
$v_{\bullet} \in \mathbb{Q}^{\oplus k}/(r_1, \ldots, r_k)\mathbb{Q}$ 
is determined by
\begin{align*}
    v_i=w_i-\frac{l}{2}r_i \left(\sum_{i>j}r_j-\sum_{i<j} r_j  \right). 
\end{align*}
The fully-faithful functor 
\begin{align*}
       \mathbb{T}^L_{\mathrm{SL}(r_{\bullet})}(\chi_{\bullet})_{w_{\bullet}} 
       \to D^b(\mathcal{M}^L_{\mathrm{SL}(r)}(\chi))_w
\end{align*}
is induced by the categorical Hall product. 
\end{thm}
\begin{proof}
For a decomposition $(r, \chi)=(r_1, \chi_1)+\cdots+(r_k, \chi_k)$ with $\chi_i/r_i=\chi/r$, 
    there is a commutative diagram 
    \begin{align}\label{dia:Mfil}
\xymatrix{
\mathcal{F}il^L(r_{\bullet}, \chi_{\bullet}) \ar[r]^-{p} \ar[d]_-{q} & 
\mathcal{M}^L(r, \chi) \ar[dd]^-{(h, \det)} \\
\times_{i=1}^k \mathcal{M}^L(r_i, \chi_i) \ar[d] & \\
\times_{i=1}^k (B_{r_i}\times \mathrm{Pic}(C)) \ar[r]^-{(+, \otimes)} & B_r \times 
\mathrm{Pic}(C). 
}    
    \end{align}
The pull-back along 
$B_{\geq 2} \times \{A\} \hookrightarrow B_r \times \mathrm{Pic}(C)$ gives the diagram 
\begin{align*}
    \xymatrix{
\mathcal{F}il^L_{\mathrm{SL}}(r_{\bullet}, \chi_{\bullet}) \ar[r]^-{p} \ar[d]_-{q} 
& \widetilde{\mathcal{M}}^L_{\mathrm{SL}(r)}(\chi) \ar[d] \\
\widetilde{\mathcal{M}}^L_{\mathrm{SL}(r_{\bullet})}(\chi_{\bullet}) \ar[r] &
B_{\geq 2}.     
    }
\end{align*}
Here $\widetilde{\mathcal{M}}^L_{\mathrm{SL}(r)}(\chi)$
and $\widetilde{\mathcal{M}}^L_{\mathrm{SL}(r_{\bullet})}(\chi_{\bullet})$
fit into Cartesian diagrams 
\begin{align*}
\xymatrix{
\mathcal{M}^L_{\mathrm{SL}(r)}(\chi) \ar[r] \ar[d] & 
\widetilde{\mathcal{M}}^L_{\mathrm{SL}(r)}(\chi) \ar[d] \\
\Spec \mathbb{C} \ar[r] & B\mathbb{C}^{\ast},
} \
\xymatrix{
\mathcal{M}^L_{\mathrm{SL}(r_{\bullet})}(\chi_{\bullet}) \ar[r] \ar[d] & 
\widetilde{\mathcal{M}}^L_{\mathrm{SL}(r_{\bullet})}(\chi_{\bullet}) \ar[d] \\
\Spec \mathbb{C} \ar[r] & B\mathbb{C}^{\ast}. 
}
\end{align*}
There is a decomposition
\begin{align}\notag
    D^b(\widetilde{\mathcal{M}}^L_{\mathrm{SL}(r_{\bullet})}(\chi_{\bullet}))
    =\bigoplus_{w_{\bullet} \in \mathbb{Z}^{\oplus k}}
    D^b(\widetilde{\mathcal{M}}^L_{\mathrm{SL}(r_{\bullet})}(\chi_{\bullet}))_{w_{\bullet}}
\end{align}
such that 
\begin{align}\label{equiv:SLtilde}
     D^b(\widetilde{\mathcal{M}}^L_{\mathrm{SL}(r_{\bullet})}(\chi_{\bullet}))_{w_{\bullet}}
     = D^b(\mathcal{M}^L_{\mathrm{SL}(r_{\bullet})}(\chi_{\bullet}))_{\overline{w}_{\bullet}},
\end{align}
    where $\overline{w}_{\bullet}$ is the class of $w_{\bullet}$ in $\mathbb{Z}^{\oplus k}/(r_1, \ldots, r_k)\mathbb{Z}$. 

For $w_{\bullet} \in \mathbb{Z}^{\oplus k}$, let 
\begin{align*}
    \widetilde{\mathbb{T}}_{\mathrm{SL}(r_{\bullet})}(\chi_{\bullet})_{w_{\bullet}}
    \subset D^b(\widetilde{\mathcal{M}}^L_{\mathrm{SL}(r_{\bullet})}(\chi_{\bullet}))_{w_{\bullet}}
\end{align*}
be the smallest pre-triangulated subcategory which 
contains the image of the base-change functor 
\begin{align*}
\boxtimes_{i=1}^k \mathbb{T}^L(r_i, \chi_i)_{w_i} \to 
D^b(\widetilde{\mathcal{M}}^L_{\mathrm{SL}(r_{\bullet})}(\chi_{\bullet}))_{w_{\bullet}}
    \end{align*}
and closed under taking direct summands. 
    By Theorem~\ref{thm:sod} and Lemma~\ref{lem:sdo}, 
    there is an induced semiorthogonal decomposition 
  \begin{align*} D^b(\widetilde{\mathcal{M}}^L_{\mathrm{SL}(r)}(\chi))_w=
\left\langle \widetilde{\mathbb{T}}^L_{\mathrm{SL}(r_{\bullet})}(\chi_{\bullet})_{w_{\bullet}} \,\Big|\,
\frac{v_1}{r_1}<\cdots<\frac{v_k}{r_k} \right\rangle. 
\end{align*}
Here the right hand side is after all partitions 
$(r, \chi, w)=(r_1, \chi_1, w_1)+\cdots+(r_k, \chi_k, w_k)$
with $\chi_i/r_i=\chi/r$.
Then the desired semiorthogonal decomposition (\ref{sod:SL}) follows from (\ref{equiv:SLtilde}). 
\end{proof}

\begin{thm}\label{thm:PGL}
For each $r_{\bullet} \in \mathbb{Z}^k$, 
$\chi_{\bullet} \in \mathbb{Z}^k/(r_1, \ldots, r_k)\mathbb{Z}$ 
and $w_{\bullet} \in \mathbb{Z}^k$, 
there is a subcategory 
\begin{align*}
      \mathbb{T}^L_{\mathrm{PGL}(r_{\bullet})}(\chi_{\bullet})_{w_{\bullet}} \subset D^b(\mathcal{M}^L_{\mathrm{PGL}(r_{\bullet})}(\chi_{\bullet}))_{w_{\bullet}}
\end{align*}
such that there is a semiorthogonal decomposition 
\begin{align}\label{sod:PGL}
    D^b(\mathcal{M}^L_{\mathrm{PGL}(r)}(\chi))_w=
\left\langle \mathbb{T}^L_{\mathrm{PGL}(r_{\bullet})}(\chi_{\bullet})_{w_{\bullet}} \,\Big|\,
\frac{v_1}{r_1}<\cdots<\frac{v_k}{r_k} \right\rangle. 
\end{align}
The right hand side is after all partitions $r=r_1+\cdots+r_k$, 
$\chi=\chi_1+\cdots+\chi_k$ 
such that 
$\chi_1/r_1=\cdots=\chi_k/r_k$%\textcolor{blue}{(I think the quotient is by the diagonal $\mathbb{Q}$, so the condition is the same as $\frac{\chi_1}{r_1}=\ldots=\frac{\chi_k}{r_k}$)},
and $w_{\bullet} \in \mathbb{Z}^{\oplus k}$ with $w_1+\cdots+w_k=w$, 
and where $v_i\in \frac{1}{2}\mathbb{Z}$
is defined by
\begin{align*}
    v_i=w_i-\frac{l}{2}r_i \left(\sum_{i>j}r_j-\sum_{i<j} r_j  \right). 
\end{align*}
The fully-faithful functor 
\begin{align*}
       \mathbb{T}^L_{\mathrm{PGL}(r_{\bullet})}(\chi_{\bullet})_{w_{\bullet}} 
       \to D^b(\mathcal{M}^L_{\mathrm{PGL}(r)}(\chi))_w
\end{align*}
is induced by the categorical Hall product. 
\end{thm}
\begin{proof}
    The pull-back of the diagram (\ref{dia:Mfil})
    by $B_{\geq 2} \times \mathrm{Pic}(C) \hookrightarrow B_r \times \mathrm{Pic}(C)$ 
    gives the diagram 
    \begin{align*}
        \xymatrix{
\mathcal{F}il^L(r_{\bullet}, \chi_{\bullet})^{\rm{tr}=0} \ar[r]^-{p} \ar[d]_-{q} & 
\mathcal{M}^L(r, \chi)^{\rm{tr}=0} \\
\mathcal{M}^L(r_{\bullet}, \chi_{\bullet})^{\rm{tr}=0}.         
        }
    \end{align*}
    By Theorem~\ref{thm:sod} and Lemma~\ref{lem:sdo}, there is a 
    semiorthogonal decomposition 
    \begin{align}\label{sod:tr=0}
        D^b(\mathcal{M}^L(r, \chi)^{\rm{tr}=0})
        =\left\langle \left(\boxtimes_{i=1}^k \mathbb{T}^L(r_i, \chi_i)_{w_i}\right)^{\rm{tr}=0} \,\Big|\,
        \frac{v_1}{r_1}<\cdots<\frac{v_k}{r_k} \right\rangle. 
    \end{align}
    Here the right hand side is after all partitions $(r, \chi)=(r_1, \chi_1)+\cdots+(r_k, \chi_k)$ with $\chi_i/r_i=\chi/r$ and $w_{\bullet} \in \mathbb{Z}^k$. The subcategory   
    \begin{align*}
    \left(\boxtimes_{i=1}^k \mathbb{T}^L(r_i, \chi_i)_{w_i}\right)^{\rm{tr}=0} 
    \subset D^b(\mathcal{M}^L(r_{\bullet}, \chi_{\bullet})^{\rm{tr}=0})_{w_{\bullet}}
    \end{align*}
    is the smallest pre-triangulated subcategory which contains the image of the 
    base-change functor 
    \begin{align*}
        \boxtimes_{i=1}^k \mathbb{T}^L(r_i, \chi_i)_{w_i} \to 
        D^b(\mathcal{M}^L(r_{\bullet}, \chi_{\bullet})^{\rm{tr}=0})_{w_{\bullet}} 
    \end{align*}
    and closed under taking the direct summands. 
    The semiorthogonal decomposition (\ref{sod:tr=0}) is preserved under 
    the action of $\mathrm{Pic}^0(C)$. Therefore by Lemma~\ref{lem:sod:descend} below, 
    it descends to the 
    semiorthogonal decomposition 
    \begin{align*}
        D^b(\widetilde{\mathcal{M}}_{\mathrm{PGL}(r)}(\chi))=
        \left\langle \mathbb{T}^L_{\mathrm{PGL}(r_{\bullet})}(\chi_{\bullet})_{w_{\bullet}} \,\Big|\, 
        \frac{v_1}{r_1}<\cdots<\frac{v_k}{r_k}\right\rangle. 
    \end{align*}
Here we set 
\begin{align*}
    \mathbb{T}^L_{\mathrm{PGL}(r_{\bullet})}(\chi_{\bullet})_{w_{\bullet}}
    =\left(\boxtimes_{i=1}^k \mathbb{T}^L(r_i, \chi_i)_{w_i}\right)^{\rm{tr}=0} /\mathrm{Pic}^0(C)
\end{align*}
in the notation of Lemma~\ref{lem:sod:descend}. 
Then the desired semiorthogonal decomposition (\ref{sod:PGL}) follows from (\ref{decom:pgl}).     
\end{proof}

We have used the following lemma. 
\begin{lemma}\label{lem:sod:descend}
Let $D^b(\mathcal{M})=\langle \mathcal{C}_i \,|\, i \in I\rangle$ be a 
semiorthogonal decomposition. Assume $G$ is an algebraic groups acting on 
$\mathcal{M}$ such that, for any $g \in G$ and $i\in I$, we have 
$g^{\ast}(\mathcal{C}_i) \subset \mathcal{C}_i$. 
Then the above semiorthogonal decomposition descends 
to the semiorthogonal decomposition 
\begin{align*}
    D^b(\mathcal{M}/G)=\langle \mathcal{C}_i/G \,|\, i \in I \rangle.
\end{align*}
Here $\mathcal{C}_i/G$ is the subcategory of $\mathcal{E} \in D^b(\mathcal{M}/G)$
such that $\pi^{\ast}\mathcal{E} \in \mathcal{C}_i$, where 
$\pi \colon \mathcal{M}\to \mathcal{M}/G$ is the quotient morphism. 
\end{lemma}
\begin{proof}
    We have
    \begin{align}\label{lim:M}
        D^b(\mathcal{M}/G)=\lim_{n\geq 0}D^b(\mathcal{M}\times G^{\times n})
    \end{align}
    and we have the semiorthogonal decomposition 
    \begin{align}\label{sod:Gn}
        D^b(\mathcal{M}\times G^{\times n})=\langle 
        \mathcal{C}_i\boxtimes D^b(G^{\times n}) \,|\, i \in I\rangle. 
    \end{align}
    By the assumption, the limit diagram in (\ref{lim:M})
    preserves $\mathcal{C}_i\boxtimes D^b(G^{\times n})$ and 
    \begin{align*}
\mathcal{C}_i/G=\lim_{n\geq 0}\,   (\mathcal{C}_i\boxtimes D^b(G^{\times n})). 
\end{align*}
The desired semiorthogonal decomposition holds by taking the limit of the 
semiorthogonal decompositions (\ref{sod:Gn}). 
\end{proof}

\subsection{The SL/PGL symmetry conjecture}
Consider the Hitchin maps 
\begin{align}\label{dia:SLPGL}
    \xymatrix{
 \mathcal{M}^L_{\mathrm{SL}(r)}(\chi) \ar[rd] & & \ar[ld] \mathcal{M}^L_{\mathrm{PGL}(r)}(\chi) \\
    & B_{\geq 2}, &
    }
\end{align}
where recall that $B_{\geq 2} =\bigoplus_{i=2}^r H^0(C, L^{\otimes i})$. 
The maps in (\ref{dia:SLPGL}) have relative dimension 
$g^{\rm{sp}}-g$. The following is a version of Conjecture~\ref{conj:T0} for the pair of Langlands dual groups SL/PGL: 

\begin{conj}\label{conj:PS}
Suppose the tuple $(r, \chi, w)$ satisfies the BPS condition.
Then there is an equivalence 
\begin{align*}
    \mathbb{T}^L_{\mathrm{PGL}(r)}(w+1-g^{\rm{sp}})_{-\chi+1-g^{\rm{sp}}} \stackrel{\sim}{\to} 
     \mathbb{T}^L_{\mathrm{SL}(r)}(\chi)_w.
\end{align*}
    \end{conj}
As before, we regard the above derived equivalence as an extension over the Hitchin base $B_{\geq 2}$ of the Poincaré equivalence for dual abelian group schemes over the locus of smooth spectral curves.

Note that, if the tuple $(r,\chi,w)$ satisfies the BPS condition, then the BPS categories $\mathbb{T}^L_{\mathrm{PGL}(r)}(w+1-g^{\rm{sp}})_{-\chi+1-g^{\rm{sp}}}$ and $\mathbb{T}^L_{\mathrm{SL}(r)}(\chi)_w$ are smooth, proper, Calabi-Yau categories over the Hitchin base $B_{\geq 2}$, see the proofs of Theorems \ref{thm:proper0} and \ref{thm:proper}.

In what follows, we give evidence of Conjecture~\ref{conj:PS} 
for $r=2$.
We write $\mathcal{M}=\mathcal{M}^L(r, \chi)$, 
$\mathcal{M}'=\mathcal{M}^L(r, w+1-g^{\rm{sp}})$. 
Recall the line bundle $\mathcal{P}^{\sharp}$ on $(\mathcal{M}'\times_B \mathcal{M})^{\sharp}$
as in (\ref{line:Preg}). 
We also set 
\begin{align*}\mathcal{M}_{\mathrm{SL}}=\mathcal{M}_{\mathrm{SL}(r)}^L(\chi), \ \mathcal{M}_{\mathrm{PGL}}'=\mathcal{M}_{\mathrm{PGL}(r)}^L(w+1-g^{\rm{sp}}).
\end{align*}

\begin{lemma}\label{P:descend}
Suppose that there is a maximal Cohen–Macaulay extension 
$\mathcal{P} \in \Coh(\mathcal{M}'\times_B \mathcal{M})$ of 
$\mathcal{P}^{\sharp}$. 
Then $\mathcal{P}$ descends to a maximal Cohen–Macaulay sheaf $\overline{\mathcal{P}}$ on 
$\mathcal{M}'_{\mathrm{PGL}} \times_{B_{\geq 2}} \mathcal{M}_{\mathrm{SL}}$.     
\end{lemma}
\begin{proof}
    Let $i$ be the composition of the closed immersions
    \begin{align*}
        i \colon (\mathcal{M}')^{\mathrm{tr}=0}\times_{B_{\geq 2}} \mathcal{M}_{\mathrm{SL}}
        \hookrightarrow  (\mathcal{M}')^{\mathrm{tr}=0}\times_{B_{\geq 2}}
        \mathcal{M}^{\mathrm{tr}=0} \hookrightarrow \mathcal{M}'\times_B \mathcal{M}. 
    \end{align*}
    Then $Li^{\ast}\mathcal{P}=i^{\ast}\mathcal{P}$ 
    and it is a maximal Cohen–Macaulay sheaf by~\cite[Lemma~2.3]{Ardual}. 
    Over the locus $B_{\geq 2} \cap B^{\rm{ell}}$, the 
    sheaf $i^{\ast}\mathcal{P}$ is equivariant under the 
    action of $\mathcal{P}ic^0(C)$ on the first factor of 
    $(\mathcal{M}')^{\mathrm{tr}=0}\times_{B_{\geq 2}} \mathcal{M}_{\mathrm{SL}}$, 
    see~\cite[Lemma~4.3]{GS}. 
    Since $i^{\ast}\mathcal{P}$ is a maximal Cohen–Macaulay sheaf and 
    the complement of the locus over $B_{\geq 2}\cap B^{\rm{ell}}$ is of codimension 
    bigger than or equal to two, the sheaf $i^{\ast}\mathcal{P}$
    is equivariant under the action of $\mathcal{P}ic^0(C)$. 
    Therefore it descends to a maximal Cohen–Macaulay sheaf on 
    $\mathcal{M}'_{\mathrm{PGL}} \times_{B_{\geq 2}} \mathcal{M}_{\mathrm{SL}}$. 
\end{proof}

The following result is proved along with the 
same argument of Theorem~\ref{thm:r=21}. 
\begin{thm}\label{thm:SLPGL2}
Under Assumption~\ref{conj:CMext}, 
Conjecture~\ref{conj:PS} holds for $r=2$ and $l>2g$. 
\end{thm}
\begin{proof}
Suppose that Assumption~\ref{conj:CMext} holds. 
By Lemma~\ref{P:descend}, $\mathcal{P}$ descends to the object 
\begin{align*}
    \mathcal{P} \in \mathbb{T}_{\mathrm{PGL}(r)}^L(w+1-g^{\rm{sp}})_{\chi+g^{\rm{sp}}-1}\boxtimes 
    \mathbb{T}_{\mathrm{SL}(r)}^L(\chi)_w
\end{align*}
which is of maximal Cohen–Macaulay. 
It defines the Fourier-Mukai functor 
\begin{align}\label{funct:Pbar}
    \Phi_{\overline{\mathcal{P}}} \colon 
    \mathbb{T}_{\mathrm{PGL}(r)}^L(w+1-g^{\rm{sp}})_{-\chi+1-g^{\rm{sp}}+1}
    \to \mathbb{T}_{\mathrm{SL}(r)}^L(\chi)_w. 
\end{align}
The computation as in Lemma~\ref{lem:codim} shows that, for $r=2$ and $l>2g$, we have 
\begin{align*}
    \mathrm{codim}(B_{\geq 2} \setminus (B_{\geq 2} \cap B^{\rm{red}}))>g^{\rm{sp}}-g,
    \end{align*}
    where the right hand side is the relative dimensions of Hitchin morphisms in (\ref{dia:SLPGL}). 
    The argument in the proof of Theorem~\ref{thm:r=21} 
    and \cite[Theorem B]{FHR} show that 
    the functor (\ref{funct:Pbar}) is an equivalence 
    over the locus $B_{\geq 2} \cap B^{\rm{red}}$. 
    Then similarly to the proof of Theorem~\ref{thm:r=21}, 
    we can apply Proposition~\ref{prop:CMQ}
    and Proposition~\ref{prop:equivT} to show that 
       the functor (\ref{funct:Pbar}) is an equivalence. 
\end{proof}

Let $(r, \chi)$ be coprime. Then 
the SL-Higgs moduli stack $\mathcal{M}^L_{\mathrm{SL}(r)}(\chi)$ is a 
trivial $\mu_r$-gerbe
\begin{align*}
    \mathcal{M}^L_{\mathrm{SL}(r)}(\chi) \to M^L_{\mathrm{SL}(r)}(\chi)
\end{align*}
for a smooth quasi-projective variety $M^L_{\mathrm{SL}(r)}(\chi)$. 
The PGL-moduli space is the smooth Deligne-Mumford stack 
\begin{align*}
    M_{\mathrm{PGL}(r)}^L(\chi)=\mathcal{M}_{\mathrm{PGL}(r)}^L(\chi)=M^L_{\mathrm{SL}(r)}(\chi)/\Gamma[r]
\end{align*}
where $\Gamma[r] \subset \mathrm{Pic}^0(C)$ is the subgroup of $r$-torsion points. As a corollary 
of Theorem~\ref{thm:SLPGL2}, we obtain the following: 

\begin{cor}\label{thm:r=2:PS}
Assume $r=2$, $l>2g$, and $\chi$ and $w+l$ are odd. Then
Conjecture~\ref{conj:PS} is true, 
namely there are equivalences 
\begin{align*}
    &D^b(M_{\mathrm{PGL}(2)}^L(1), \alpha) \stackrel{\sim}{\to} D^b(M_{\mathrm{SL}(2)}^L(1)) \,\text{ if } 
    l \mbox{ is even}, \\
  &D^b(M_{\mathrm{PGL}(2)}^L(1)) \stackrel{\sim}{\to} D^b(M_{\mathrm{SL}(2)}^L(1)) \,\text{ if } 
    l \mbox{ is odd}.  
\end{align*}
In the above, $\alpha$ is the Brauer class corresponding to the gerbe (\ref{pgl:gerbe}). 
\end{cor}

Using Theorem~\ref{thm:CMext}, we obtain
the following:

\begin{cor}\label{r=2:PS:P1}
Conjecture~\ref{conj:PS} is true for
$r=2$ and $C=\mathbb{P}^1$. 
\end{cor}

\subsection{Deformations of $\mathbb{Z}/2$-periodic derived categories}
In the rest of this section, we prove Theorem~\ref{thm:intro4}. 

We first discuss deformations of the categories of matrix factorizations 
via deformations of potentials. 
Let $M$ be a 
smooth variety (or Deligne-Mumford stack) and let $f \colon M \to \mathbb{C}$ 
be a regular function such that 
the critical locus
\begin{align*}
    Y=\mathrm{Crit}(f) \subset M
\end{align*}
is smooth. 

Let $D^{\mathbb{Z}/2}(Y)$ be the (idempotent completed)
$\mathbb{Z}/2$-periodic 
derived category of coherent sheaves on $Y$: 
\begin{align*}
    D^{\mathbb{Z}/2}(Y)=\mathrm{MF}(Y, 0). 
\end{align*}
\begin{remark}\label{rmk:idemp}
Up to idempotent completion, 
the category $\mathrm{MF}(Y, 0)$ 
is equivalent to the quotient category, see~\cite{T4}:
\begin{align*}
    D^b(Y \times \Spec \mathbb{C}[\epsilon])/\mathcal{C}_{Y\times \{0\}}.
\end{align*}
In the above, $\deg \epsilon=-1$ and $\mathcal{C}_{Y\times \{0\}}$ is the 
subcategory of objects with singular support contained in
\begin{align*}
    Y\times \{0\} \subset Y\times \mathbb{A}^1 =\Omega_{Y\times \Spec \mathbb{C}[\epsilon]}[-1]^{\rm{cl}}. 
\end{align*}
\end{remark}
Let $N=N_{Y/M}$ be the normal bundle. 
The Hessian of $f$ determines the non-degenerate 
symmetric bilinear  
form 
\begin{align*}
    q_{N} \colon N\otimes_{\mathcal{O}_Y} N \to \mathcal{O}_Y. 
\end{align*}
Denote by $w_i(N, q_N) \in H^{i}(Y, \mathbb{Z}/2)$ the Stiefel-Whitney class
associated with the quadratic vector bundle $(N, q_N)$. 
\begin{thm}\emph{(\cite[Theorem~2]{Teleman})}
Suppose that $N$ is of even rank and assume $w_i(N, q_N)=0$
for $i=1, 2$. Then $D^{\mathbb{Z}/2}(Y)$ and $\mathrm{MF}(M, f)$ are 
deformation equivalent. 
\end{thm}

In~\cite{Teleman}, the deformation of $D^{\mathbb{Z}/2}(Y)$ into $\mathrm{MF}(M, f)$ 
is 
described in terms of solutions of Maurer-Cartan equations. 
Here we give another descriptions of the deformation equivalence of these categories
via deformations of super-potentials in the following special case. 
Let $\mathcal{V} \to M$ be a vector bundle with a non-degenerate 
quadratic form 
\begin{align*}
    q_{\mathcal{V}} \colon \mathcal{V} \otimes_{\mathcal{O}_M} \mathcal{V} \to \mathcal{O}_M. 
\end{align*}
Assume that the rank of $\mathcal{V}$ is 
even and that $w_i(\mathcal{V}, q_{\mathcal{V}})=0$ 
for $i=1, 2$. 
Let $s \in \Gamma(M, \mathcal{V})$ be a regular 
section such that 
the zero locus 
\begin{align*}
    Y=s^{-1}(0) \subset M
\end{align*}
is smooth. Let $f$ be the composition 
\begin{align*}
    f \colon M \stackrel{s}{\to} \mathcal{V} \xrightarrow{q_{\mathcal{V}}(-, -)} \mathbb{C}. 
\end{align*}
Then we also have that $Y=\mathrm{Crit}(f)$. 
We consider two regular functions 
$g_i \colon \mathcal{V} \to \mathbb{C}$ defined as follows: 
\begin{align*}
    g_1(x, v)=q_{\mathcal{V}}(s(x), v), \ 
    g_2(x, v)=f(x)+q_{\mathcal{V}}(v, v)
\end{align*}
for $x \in M$ and $v \in \mathcal{V}|_{x}$. 
Note that 
\begin{align*}
    \mathrm{Crit}(g_1)=\mathrm{Crit}(g_2)=Y 
\end{align*}
and that $Y$ lies in the zero section of $\mathcal{V} \to M$. 
\begin{prop}\label{prop:MF2}
There are equivalences 
\begin{align*}
    D^{\mathbb{Z}/2}(Y)\stackrel{\sim}{\to} \mathrm{MF}(\mathcal{V}, g_1), \ 
    \mathrm{MF}(M, f) \stackrel{\sim}{\to} \mathrm{MF}(\mathcal{V}, g_2). 
\end{align*}
\end{prop}
\begin{proof}
    The first equivalence follows from the $\mathbb{Z}/2$-periodic version of 
    Koszul equivalence~\cite{T4}. 
    As for the second equivalence, 
    by the assumption that $\mathcal{V}$ is of even rank and $w_i(\mathcal{V}, q_{\mathcal{V}})=0$ for $i=1, 2$, 
    there is a
    $D^{\mathbb{Z}/2}(M)$-linear equivalence by~\cite[Theorem~1]{Teleman}
    \begin{align*}
        \mathrm{MF}(\mathcal{V}, q_{\mathcal{V}}) 
        \simeq D^{\mathbb{Z}/2}(M).
    \end{align*}
    Then the desired equivalence follows by 
    applying $\otimes_{D^{\mathbb{Z}/2}(M)} \mathrm{MF}(M, f)$ to the equivalence above and by using the Thom-Sebastiani theorem for categories of matrix factorizations \cite{Preygel}. 
\end{proof}

Let $H$ be the function 
\begin{align*}
    H \colon \mathcal{V} \times \mathbb{C} \to \mathbb{C}, \ 
    H((x, v), t)=t g_1(x, v)+(1-t)g_2(x, v). 
\end{align*}
Here $x \in M$, $v \in \mathcal{V}|_{x}$ and $t \in \mathbb{C}$. 
Let $\Delta \subset \mathbb{A}^1$ be the open 
subset of $t \in \mathbb{A}^1$ such that the critical locus 
of $h_t(-):=H(-, t)$ equals to $Y$. 
By Proposition~\ref{prop:MF2}, the category over $\Delta$
\begin{align*}
    \mathrm{MF}(\mathcal{V} \times \Delta, H)
\end{align*}
gives a deformation of $D^{\mathbb{Z}/2}(Y)$ into 
$\mathrm{MF}(M, f)$. 

\subsection{A Categorical deformation equivalence for SL/PGL Higgs moduli spaces}

We now provide some evidence towards Conjecture~\ref{conj:PS}.

\begin{thm}\label{thm:eq:can}
    The categories 
    $D^{\mathbb{Z}/2}(M_{\mathrm{PGL}(2)}(1), \alpha)$
    and $D^{\mathbb{Z}/2}(M_{\mathrm{SL}(2)}(1))$
    are deformation equivalent. 
\end{thm}
\begin{proof}
    Let $p_1, \ldots, p_{2k} \in C$ be 
    distinct points and 
    set $L=\Omega_C(p_1+\cdots+p_{2k})$. 
    The natural embedding 
    $\Omega_C \hookrightarrow L$ induces the 
    closed immersion 
    \begin{align*}
        M_{\mathrm{SL}(2)}(1) \hookrightarrow 
        M^L_{\mathrm{SL}(2)}(1)
    \end{align*}
    given by 
    \begin{align*}
        (F \to F \otimes \Omega_C) \mapsto 
        (F \to F\otimes \Omega_C \hookrightarrow F \otimes L). 
    \end{align*}
Below we fix isomorphisms $L|_{p_i} \cong \mathbb{C}$. 
Let $\mathcal{F} \in \Coh(C \times M^L_{\mathrm{SL}(2)}(1))$
be a universal sheaf, and 
set $\mathcal{F}_p:=\mathcal{F}|_{\{p\} \times 
M^L_{\mathrm{SL}(2)}(1)}$. 
The vector bundle 
\begin{align*}
\mathcal{V}=
    \bigoplus_{i=1}^{2k}\mathcal{E}nd(\mathcal{F}_{p_i})_0
    \to M_{\mathrm{SL}(2)}^L(1)
\end{align*}
admits a regular section $s$
such that $s(F, \theta)=\{\theta|_{p_i}\}$, 
and its zero locus coincides with 
$M_{\mathrm{SL}(2)}(1)$. 
The vector bundle $\mathcal{V}$ admits a non-degenerate 
quadratic form $q_{\mathcal{V}}$
given by 
\begin{align*}
q_{\mathcal{V}} \in \mathrm{Sym}^2(\mathcal{V}^{\vee}), \ 
    q_{\mathcal{V}}((v_1, \ldots, v_{2k}), (v_1', \ldots, v_{2k}'))
    =\sum_{i=1}^{2k}\mathrm{tr}(v_i v_i'). 
\end{align*}
We see that $w_i(\mathcal{V}, q_{\mathcal{V}})=0$ for 
$i=1, 2$ for $k\geq 2$. Indeed, 
the quadratic vector bundles 
$(\mathcal{E}nd(\mathcal{F}_{p_i}), \mathrm{tr}(v_i v_i'))$
are deformation equivalent, so they have the 
Stiefel-Whitney class $1+w_1+w_2+\cdots$ independent of the choice of $p_i$. 
Therefore from the equality 
\begin{align*}
    1+w_1(\mathcal{V}, q_{\mathcal{V}})
    +w_2(\mathcal{V}, q_{\mathcal{V}})+\cdots
    =(1+w_1+w_2+\cdots)^{2k}
\end{align*}
in $H^{\ast}(M, \mathbb{Z}/2)$, we have 
$w_i(\mathcal{V}, q_{\mathcal{V}})=0$ for $i=1, 2$.

The function 
$f=q_{\mathcal{V}} \circ s$ on
$M^L_{\mathrm{SL}(2)}(1)$ is given by 
\begin{align*}
    f \colon M^L_{\mathrm{SL}(2)}(1) \to \mathbb{C}, \ 
    (F, \theta) \mapsto \sum_{i=1}^{2k}\mathrm{tr}(\theta|_{p_i}^2). 
\end{align*}
    The above function factors through the 
    Hitchin map 
    \begin{align*}
        f \colon 
        M_{\mathrm{SL}(2)}^L(1) \to         
        B^L_{\geq 2}:=H^0(L^{\otimes 2}) 
        \stackrel{g}{\to} \mathbb{C}
    \end{align*}
    where $g(u)=\sum_{i=1}^k u|_{p_i}$
    under the fixed isomorphisms 
    $L|_{p_i} \cong \mathbb{C}$. 

Recall the $B^L_{\geq 2}$-linear equivalence from Corollary~\ref{thm:r=2:PS}: 
\begin{align*}
    D^b(M^L_{\mathrm{PGL}(2)}(1), \alpha)
    \stackrel{\sim}{\to} D^b(M_{\mathrm{SL}(2)}^L(1)).
\end{align*}
By taking the tensor product of the above categories with 
$\mathrm{MF}(B, g)$, 
we obtain the equivalence: 
\begin{align*}
    \mathrm{MF}((M^L_{\mathrm{PGL}(2)}(1), \alpha), f)
    \stackrel{\sim}{\to} \mathrm{MF}(M^L_{\mathrm{SL}(2)}(1), f). 
\end{align*}
By Proposition~\ref{prop:MF2} for $\mathcal{M}_{\mathrm{PGL}(2)}(1)$, there is an equivalence:
\[\mathrm{MF}((M^L_{\mathrm{PGL}(2)}(1), \alpha), f)\cong D^{\mathbb{Z}/2}(M_{\mathrm{PGL}(2)}(1), \alpha).\]
By Proposition~\ref{prop:MF2} for $M_{\mathrm{SL}(2)}(1)$, 
the category $\mathrm{MF}(M^L_{\mathrm{SL}(2)}(1), f)$ is deformation equivalent to 
$D^{\mathbb{Z}/2}(M_{\mathrm{SL}(2)}(1))$. The conclusion then follows.
\end{proof}

 \bibliographystyle{amsalpha}
\bibliography{math}

\medskip

\textsc{\small Tudor P\u adurariu: Sorbonne Université and Université Paris Cité, CNRS, IMJ-PRG, F-75005 Paris, France.}\\
\textit{\small E-mail address:} \texttt{\small padurariu@imj-prg.fr}\\

\textsc{\small Yukinobu Toda: Kavli Institute for the Physics and Mathematics of the Universe (WPI), University of Tokyo, 5-1-5 Kashiwanoha, Kashiwa, 277-8583, Japan.}\\
\textit{\small E-mail address:} \texttt{\small yukinobu.toda@ipmu.jp}\\

 \end{document}